\documentclass[reqno, 11pt]{amsart}
\usepackage[utf8]{inputenc}
\usepackage{amssymb, amsmath, amsfonts, amsthm,bbm, color, enumerate, tikz, tikz-3dplot}
\usepackage{verbatim}
\usepackage[mathscr]{eucal}

\usepackage{hyperref}

\usepackage[explicit,indentafter]{titlesec}
\titleformat{\section}{\centering\normalfont\scshape}{\thesection.}{.5em}{#1}
\titleformat{\subsection}[runin]{\normalfont\itshape}{\textnormal{\thesubsection.}}{.5em}{#1.}
\titleformat{\subsubsection}[runin]{\normalfont\itshape}{\thesubsubsection.}{.5em}{#1.}
\titlespacing{\section}{0em}{1em}{0.5em}
\titlespacing{\subsection}{0em}{.5em}{0.5em}

\usepackage{makecell}
\setcellgapes{2pt}

\usepackage[normalem]{ulem}
\usepackage{todonotes}

\tdplotsetmaincoords{75}{15}

\usepackage[thinlines]{easytable}
\usetikzlibrary{patterns}

\def\La{\Lambda}

\def\bbone{{\mathbbm 1}}

\newcommand{\sH}{\mathscr H}

\newcommand{\ci}[1]{_{{}_{\!\scriptstyle{#1}}}}

\usepackage{float}

\newcommand{\Be}{\begin{equation}}
\newcommand{\Ee}{\end{equation}}

\newcommand{\Bm}{\begin{multline}}
\newcommand{\Em}{\end{multline}}
\newcommand{\Bea}{\begin{eqnarray}}
\newcommand{\Eea}{\end{eqnarray}}
\newcommand{\Beas}{\begin{eqnarray*}}
\newcommand{\Eeas}{\end{eqnarray*}}
\newcommand{\Benu}{\begin{enumerate}}
\newcommand{\Eenu}{\end{enumerate}}
\newcommand{\Bi}{\begin{itemize}}
\newcommand{\Ei}{\end{itemize}}

\def\dyad{{\text{\rm dyad}}}
\def\sh{{\text{\rm sh}}}

\def\intslash{\rlap{\kern  .32em $\mspace {.5mu}\backslash$ }\int}
\def\qsl{{\rlap{\kern  .32em $\mspace {.5mu}\backslash$ }\int_{Q_x}}}

\def\vth{\vartheta}

\def\N{\mathbb N}

\def\emph#1{{\it #1 }}

\def\cf{{\it cf}}

\def\supp{{\text{\rm supp}}}

\def\inn#1#2{\langle#1,#2\rangle}

\def\lc{\lesssim}
\def\gc{\gtrsim}

\def\eps{\varepsilon}

\def\ka{\kappa}
            
\def\la{\lambda}             \def\La{\Lambda}

              \def\Om{\Omega}

\def\fA{{\mathfrak {A}}}

\def\fM{{\mathfrak {M}}}

\def\fP{{\mathfrak {P}}}
\def\fQ{{\mathfrak {Q}}}

\def\fS{{\mathfrak {S}}}

\def\fn{{\mathfrak {n}}}

\def\fz{{\mathfrak {z}}}

\def\bbN{{\mathbb {N}}}

\def\bbR{{\mathbb {R}}}

\def\bbZ{{\mathbb {Z}}}

\def\cA{{\mathcal {A}}}

\def\cG{{\mathcal {G}}}

\def\cM{{\mathcal {M}}}

\def\cQ{{\mathcal {Q}}}

\def\cS{{\mathcal {S}}}

\def\cU{{\mathcal {U}}}

\def\be#1{\begin{equation}\label{ #1}}
\def\endeq{\end{equation}}
\def\endal{\end{align}}
\def\bas{\begin{align*}}
\def\eas{\end{align*}}
\def\bi{\begin{itemize}}
\def\ei{\end{itemize}}

\def\eps{\varepsilon}

\def\emph#1{{\it #1}}
\def\textbf#1{{\bf #1}}

\def\beq{\begin{equation}}
\def\endeq{\end{equation}}

\numberwithin{equation}{section}

\newcommand{\R}{\mathbb{R}}
\newcommand{\Z}{\mathbb{Z}}
\newcommand{\ud}{\, \mathrm{d}}

\theoremstyle{plain}
\newtheorem{theorem}{Theorem}[section]
\newtheorem{lemma}[theorem]{Lemma}
\newtheorem{proposition}[theorem]{Proposition}

\newtheorem{corollary}[theorem]{Corollary}
\newtheorem*{key estimate}{Key estimate}

\theoremstyle{definition}

\newtheorem*{key example}{Key example}

\theoremstyle{remark}
\newtheorem*{remark}{Remark}

\title[
Variation bounds for spherical averages]{
Variation bounds  for spherical averages}
\author[D. Beltran, R. Oberlin, L. Roncal, A. Seeger, B. Stovall]{David Beltran  \ \  \ \ Richard Oberlin \ \  \ \ Luz Roncal \\ \\    Andreas Seeger \ \ \ \ Betsy Stovall}

\address{David Beltran: Department of Mathematics, University of Wisconsin, 480 Lincoln Drive, Madison, WI, 53706, USA.}
\email{dbeltran@math.wisc.edu}
\address{Richard Oberlin: 
Department of Mathematics, 
Florida State University, 
Tallahassee, FL 32306-4510, USA} 
\email{roberlin@math.fsu.edu}
\address{Luz Roncal: Basque Center for Applied Mathematics (BCAM), 48009, Bilbao, Spain and Ikerbasque, Basque Foundation for Science, 48011 Bilbao, Spain}
\email{lroncal@bcamath.org}
\address{Andreas Seeger: Department of Mathematics, University of Wisconsin, 480 Lincoln Drive, Madison, WI, 53706, USA.}
\email{seeger@math.wisc.edu}
\address{Betsy Stovall: Department of Mathematics, University of Wisconsin, 480 Lincoln Drive, Madison, WI, 53706, USA.}
\email{stovall@math.wisc.edu}

\date{\today}


\begin{document}

\begin{abstract}
We consider $r$-variation operators for the family of spherical means, with special emphasis on $L^p\to L^q$ estimates.
\end{abstract}

\maketitle

\section{Introduction}
Given a subset $E \subset \R$ and a family of complex valued functions $t \mapsto a_t$ defined on $E$, the $r$-variation of $a=\{a_t\}_{t \in E}$ is defined by
$$
| a |_{V_r(E)} := \sup_{N \in \N} \,\, \sup_{\substack{t_1 < \cdots < t_N \\ t_j \in E}} \Big(\sum_{j=1}^{N-1} |a_{t_{j+1}} - a_{t_j}|^r \Big)^{1/r}
$$
for all $1 \leq r < \infty$, and replacing the $\ell^r$-sum by a $\sup$ in the case $r=\infty$. When $E=\bbR$ we simply use the notation $V_r$ for $V_r(\R)$. A norm on the space $V_r(E)$ is given by $\|a \|_{V_r(E)}:=\|a\|_\infty+|a|_{V_r(E)}$.  Variation norms have received considerable attention in analysis as they are used to strengthen pointwise convergence results for families of operators $\{A_t\}$.
Of particular interest is L\'epingle's inequality on the $r$-variation of martingales for $r>2$  \cite{Lepingle1976} (see also \cite{PisierXu}, \cite{Bourgain-ergodic1989}, \cite{JSW}, \cite{MirekSteinZorin-jumps}) and its consequences on families of operators in ergodic theory and harmonic analysis; see e.g. the papers \cite{JonesKaufmanRosenblattWierdl}, \cite{JSW}, \cite{OSTTW}, \cite{MirekSteinTrojan2017}, \cite{GuoRoosYung}, \cite{MSZK-apde}, \cite{MSZK-advances}, which contain many other references.

In this paper we focus on local and global $r$-variation estimates for the family of spherical averages
${A}=\{A_t\}_{t>0}$, given by 
$$
A_t f(x) = \int_{S^{d-1}} f(x-ty)  \ud \sigma (y)
$$
where  $\ud \sigma$ denotes the normalized surface measure on the unit sphere $S^{d-1}$. By a classical result of Stein \cite{Stein1976} ($d \geq 3$) and Bourgain \cite{Bourgain1986} ($d=2$)  the spherical maximal function $Sf(x):=\sup_{t>0} |A_tf(x)|$ defines a  bounded operator on $L^p(\R^d)$ if and only if $p>\frac{d}{d-1}$. Thus, for $p$ in this range, we have $\lim_{t\to 0} A_tf(x)=f(x) $ a.e. for all $f\in L^p(\bbR^d)$. A strengthening of this result 
can be obtained  by considering the variation norm operator $V_rA$ given by 
\[V_rA f(x) \equiv V_r{[Af]}  (x):=| {Af}(x) |_{V_r((0,\infty))};\] 
note that $V_r[Af] (x) \geq \sup_t|A_t f(x)-A_{t_0} f(x)|$ for all $x \in \R^d$, $t_0\in \bbR$. In this context, Jones,  Wright and one of the authors  \cite{JSW} obtained an almost optimal result, namely  $V_r{A}$ is bounded on $L^p(\R^d)$ for all $r>2$ if $\frac{d}{d-1}<p \leq 2d$, and both the condition $r>2$ and the $p$-range  are sharp. 
In the range $p>2d$, it was shown in \cite{JSW} that $V_rA$ is  $L^p$ bounded if $r>p/d$, and fails to be bounded if $r<p/d$, but no information was known for the critical case $r=p/d$, $p>2d$.  Here we show an endpoint result for $V_{p/d}A$ in three and higher dimensions.
\begin{theorem} \label{endptvarthm} Let $d\ge 3$, $p>2d$.  Then the operator $V_{p/d}A$ is of restricted weak type $(p,p)$, i.e. maps $L^{p,1}(\bbR^d)$ to $L^{p,\infty}(\bbR^d)$.
\end{theorem}
We conjecture that a similar endpoint result holds true in two dimensions, but this remains open. 

Our main focus will be on $L^p\to L^q$ results when $p<q$ for  local $r$-variation operators, that is, 
when the variation is taken over a compact subinterval $I$ of $(0,\infty)$; without loss of generality we take $I=[1,2]$. Scaling reasons quickly reveal that one needs to consider compact intervals for $L^p \to L^q$ bounds to hold if $p<q$.  While this is an interesting problem in itself, it is also motivated by a question posed by Lacey \cite{LaceySpherical} concerning sparse domination for the global $V_rA$ operator (see also \cite[Problem 3.1]{aimPL}). See Theorem  \ref{cor:sparse} below.

Results for the local variation operators are  meant to improve on existing  $L^p\to L^q$ results for the spherical local maximal function $S^I f(x):=\sup_{1 \leq t \leq 2} A_tf(x)$, which we will now review. 
 Schlag \cite{Schlag1997} (see also \cite{SS1997}) showed that if $d\ge 2$
 %if $d=2$,  
 there are $L^p(\R^d)\to L^q(\R^d)$ bounds if $(1/p,1/q)$ lies in the interior of $\fQ_d$, which denotes the quadrangle formed by the vertices
\Be\label{vertices-quadrangle} \begin{aligned}
&Q_1=(0,0),   & Q_2&=(\tfrac{d-1}d, \tfrac{d-1}d), \\
&Q_3=(\tfrac{d-1}{d}, \tfrac 1d), & Q_4&=( \tfrac{d(d-1)}{d^2+1}, \tfrac{d-1}{d^2+1}).
\end{aligned}
\Ee
Moreover, $S^I$ fails to be bounded from $L^p(\R^d)$ to $L^q(\R^d)$ outside the closure of $\fQ_d$. Note that $Q_2$ coincides with $Q_3$ when $d=2$, so the quadrangle becomes a triangle in two dimensions.

The boundary segment $p=q$ amounts to the classical results of Stein and Bourgain for $S$. $L^p$-boundedness fails  at the endpoint
$Q_2$ but Bourgain showed in dimensions $d\ge 3$ that $S$ is of restricted weak type at $Q_2$, i.e. bounded from $L^{\frac{d}{d-1},1}$ to $L^{\frac{d}{d-1},\infty}$  in dimensions $d\ge 3$ (and any better Lorentz estimate fails). The restricted weak type estimate at $Q_2$  fails in two dimensions \cite{STW} (even though it is true for radial functions \cite{Leckband1987}).  For the remaining boundary cases  Lee \cite{Lee2003} showed that $S^I$ is of restricted weak type at $Q_4$, and also at $Q_3$ in  dimensions  $d \ge 3$.
The two-dimensional restricted weak type endpoint result at $Q_4$ was also shown in \cite{Lee2003}, and relied  on the deep work by Tao \cite{TaoCone} on  endpoint bilinear Fourier extension bounds for the cone. The restricted weak type inequalities imply  $L^p\to L^q$ boundedness  on $[Q_1,Q_4)$ and on $(Q_3,Q_4)$, however on $(Q_2,Q_3)$ the operator is of restricted strong type and no better (the necessity follows from the standard counterexample; for the positive result one uses  real interpolation on a vertical line, with a constant target exponent). Incidentally,  for the local operator $S^I$ this also implies restricted strong type at $Q_2$, which improves over the restricted weak type of $S$ at $Q_2$.

Here we explore the existence of $L^p(\R^d)\to L^q(\R^d)$ inequalities for \[V_r^IA f(x):=| Af(x) |_{V_r([1,2])}. \]
In two dimensions the values of $r$ are restricted to $r>2$ (see \S\ref{sec:sharpness}) but in higher dimensions all $r \in [1,\infty]$ may occur. For our sparse domination inequality for the global $V_r$, the version for $r>2$ is most relevant because  L\'epingle's result requires the restriction $r>2$ (see \cite{Qian1998}); indeed  this necessary condition can  be shown to carry over to other results for the global $V_r$.

We start stating our results for $d \geq 3$. We first focus on the range $r>\frac{d^2+1}{d(d-1)}$ which is the reciprocal of the $1/p$ coordinate of the point $Q_4$ in \eqref{vertices-quadrangle}. Note that this large range includes $r>2$, so the following sharp $L^p\to L^q$ results for $V_r^I A$ will yield, in particular, satisfactory results for  the sparse domination problem in dimension $d \geq 3$.
\begin{theorem}\label{dge3thm}
Suppose $d \geq 3$ and $r> \frac{d^2+1}{d(d-1)} $. 
Let $\fP_d(r)$ be the pentagon (Figure~\ref{fig:big r}) with vertices
\begin{align*}
&P(r)= (\tfrac 1{r},\tfrac 1{rd} ), \quad  Q_1(r)=(\tfrac 1{rd}, \tfrac 1{rd}),
\quad Q_2= (\tfrac{d-1}{d}, \tfrac {d-1}{d} ) \\
& \qquad \quad Q_3=(\tfrac{d-1}{d} ,\tfrac{1}{d}) , \quad Q_4= (\tfrac{d(d-1)}{d^2+1},\tfrac{d-1}{d^2+1}).
\end{align*}
Then

(i) $V_r^IA: L^p\to L^q$ is bounded for all  $(\tfrac 1p,\tfrac 1q)$ in the interior of $\fP_d(r)$  and unbounded for all $(\tfrac 1p,\tfrac 1q)\notin \fP_d(r)$.

(ii) $V_r^IA: L^p\to L^q$ is bounded for all 
$(\tfrac 1p,\tfrac 1q)$ on the half open line segment $[Q_1(r), Q_2)$,
on the closed line segment $[P(r), Q_1(r)]$, on the half open line segment $[P(r), Q_4)$, and on the open line segment
$(Q_4,Q_3)$.

(iii) $V_r^IA :L^{p,1} \to L^q$ is bounded (i.e. of  restricted strong type $(p,q)$)  if $(\tfrac 1p, \tfrac 1q)$ belongs to the half open line segment
$[Q_2,Q_3)$. $V_r^I A$  fails to be of strong type on $[Q_2,Q_3]$.

(iv) $V_r^IA: L^{p,1}\to L^{q,\infty}$  is bounded (i.e. of restricted weak type $(p,q)$)  if $(\tfrac 1p,\tfrac 1q) \in \{Q_3, \, Q_4\}$.
\end{theorem}
For an explicit description of the various conditions at the boundary see \S\ref{sec:edges}.

%%%%%%%%%%%%%%%%%%%%%%%%%%%%
%%
%   figure big r
%%
%%%%%%%%%%%%%%%%%%%%%%%%%%%%%

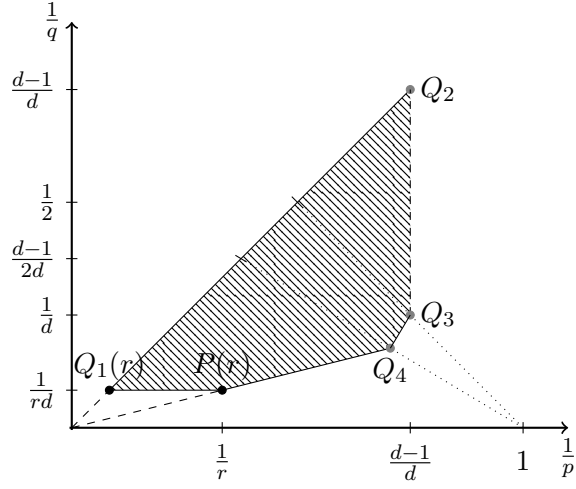
\begin{figure}[h]
\begin{tikzpicture}[scale=1.5] 

\begin{scope}[scale=2]
\draw[thick,->,black] (-.01,0) -- (2.2,0) node[below] {$ \frac 1 p$};
\draw[thick,->] (0,-.01) -- (0,1.80) node[left] {$ \frac 1 q$};

\draw (2,-0.025) -- (2,.025) node[below= 0.25cm] {$ 1$};

\draw (.025,1) -- (-.025,1) node[left] {$ \frac{1}{2}$};

\draw[ dotted]  (2,0)  -- (1,1);

\node[circle,draw=black, fill=black, inner sep=0pt,minimum size=3pt] (b) at (2/3,1/6) {}; %Dot at P(r).

\draw (2/3,1/6) node[above] {\textcolor{black}{$P(r)$}}; %Label P(r)

\draw (2/3,-0.025) -- (2/3, .025); %x-axis mark P(r)
\draw (2/3,0) node[below = 0.08cm] {\textcolor{black}{$ \frac{1}{r}$} }; %x-axis label P(r)

\node[circle,draw=black, fill=black, inner sep=0pt,minimum size=3pt] (b) at (1/6,1/6) {}; %dot at Q_1(r)

\draw (1/6,1/6) node[above] {\textcolor{black}{$Q_1(r)$}}; %label at Q_1(r)

\draw (-0.025,1/6) -- (.025,1/6); %y-axis mark Q_1(r)
\draw (0,1/6) node[left = 0.08cm] {\textcolor{black}{$ \frac{1}{rd}$} }; %y-axis label Q_1(r)
\draw[color=black] (2/3,1/6) -- (1/6,1/6) ; %Solid line from Q_1(r) to P(r)
\draw[color=black] (1/6,1/6) -- (3/2,3/2); %Solid  line from Q_1(r) to Q_2
\draw[color=black] (2/3,1/6) -- (24/17, 6/17); %Solid line from P(r) to Q_4
\draw[ dashed] (0,0) -- (1/6,1/6); %dashed line from (0,0) to Q_1(r)

\fill[pattern=north west lines, pattern color=black] (1/6, 1/6) %Q_1(r)
-- (2/3,1/6) %P(r)
-- (24/17, 6/17) %Q_4
-- (3/2,1/2) %Q_3
-- (3/2,3/2) %Q_2
-- (1/6,1/6); %Q_1(r)

%%%% conjectured point

\draw (3/4-.025,3/4+.015) -- (3/4+.025,3/4-.015);
\draw (1-.025,1+.025) -- (1+.025,1-.025);

\draw (-0.025,3/4) -- (.025,3/4); %y-axis tick at (d-1)/2d
\draw (0,3/4) node[left = 0.08cm] {\textcolor{black}{$ \frac{d-1}{2d}$} }; %y-axis label at (d-1)/2d

\draw[ dotted, color=black] (2,0) -- (3/4,3/4) ; %dotted line from (1,0) to (d-1)/2d

\draw[dashed] (0,0) -- (2/3,1/6); %dashed line (0,0) to P(r)

\node[circle,draw=gray, fill=gray, inner sep=0pt,minimum size=3pt] (b) at (24/17, 6/17 ) {}; %dot at Q_4

\draw (24/17, 6/17 ) node[below] {\textcolor{black}{$Q_4$}}; %label Q_4

%%%%% lower triangle

%\draw[dashed] (0,0) -- (1/8,1/8);
%\draw[dashed] (0,0) -- (1,1/8);

%\fill[pattern=north west lines, pattern color=gray] (1/8, 1/8) -- (1,1/8) -- (0,0) --  (1/8,1/8);

%%%%% endpoint spherical

\draw (-0.025,3/2) -- (.025,3/2); %y-axis tick at (d-1)/d
\draw (0,3/2) node[left = 0.08cm] {\textcolor{black}{$ \frac{d-1}{d}$} }; %y-axis label at (d-1)/d
%\draw[dashed] (1,1) -- (6.75/4,6.75/4);
\node[circle,draw=gray, fill=gray, inner sep=0pt,minimum size=3pt] (b) at (3/2,3/2) {}; %dot at Q_2
\draw (3/2,3/2) node[right] {\textcolor{black}{$Q_2$}}; %label Q_2

\draw[dashed, color=black]  (3/2,1/2) -- (3/2,3/2); %dashed line from Q_3 to Q_2

\node[circle,draw=gray, fill=gray, inner sep=0pt,minimum size=3pt] (b) at (3/2,1/2) {}; %dot at Q_3

\draw (3/2,1/2) node[right] {\textcolor{black}{$Q_3$}}; %label Q_3

%%%%% at the line q=p'

%\node[circle,draw=black, fill=black, inner sep=0pt,minimum size=3pt] (b) at (3/2,1/2) {}; %dot at Q_3 [redundant]
\draw (3/2, -0.025) -- (3/2, .025); %x-axis tick (d-1)/d
\draw (3/2, 0) node[below = 0.08cm] {\textcolor{black}{$ \frac{d-1}{d}$} }; %x-axis label (d-1)/d
\draw (-0.025,1/2) -- (.025,1/2);%y-axis tick 1/d
\draw (0,1/2) node[left = 0.08cm] {\textcolor{black}{$ \frac{1}{d}$} }; %y-axis label 1/d

%%% closing the trapezoid/subpentagon

\draw[color=black]  (3/2,1/2) -- (24/17, 6/17); %line from Q_3 to Q_4
%\fill[pattern=north west lines, pattern color=orange] (1,1/8) -- (1 + 1/2 + 1/10, 1/8+1/16 + 1/80) -- (6.75/4,1.25/4) -- (6.75/4,6.75/4)  -- (1,1) --(1,1/8);

\end{scope}

\end{tikzpicture}

\caption{The pentagon $\mathfrak{P}_d(r)$ for $r > \frac{d^2+1}{d^2-d}$ and $d \geq 3$ (Theorem~\ref{dge3thm}).  The outer (dashed) quadrangle shows the region of boundedness as $r \to \infty$, i.e.\ for the maximal operator.  Shown with $d=4$ and $r=3$.}
\label{fig:big r}
%\vspace{1cm}
\end{figure}

We leave open what exactly happens  at  the points $Q_3$ and $Q_4$; it is not even known whether the local maximal function is of restricted strong type at $Q_3$ and whether it is any better than restricted weak type at $Q_4$. If we take $r=\infty$ we recover the known theorem for the local spherical maximal operator. Note that both $P(r)$ and $Q_1(r)$ tend to $Q_1=(0,0)$ as $r\to \infty$.

Theorem \ref{dge3thm} covers an interesting consequence for a sharp strong type estimate at the lower edge $q^{-1}=p^{-1}/d$  of the type set for the maximal function.

\begin{corollary}\label{cor:lower boundary edge}
Let $d\ge 3$ and let $\frac{d^2+1}{d(d-1)}<p<\infty$. Then $V^I_rA:  L^p\to L^{pd}$ is bounded if and only if $r\ge p$.
\end{corollary}

When the value of $r$ is between the reciprocal of the $1/p$ coordinate of $Q_4$ and $Q_3$, that is, $\tfrac{d}{d-1} < r \leq \tfrac{d^2+1}{d(d-1)}$, we obtain the following. 

\begin{theorem}\label{thm:intermediate r}
Suppose $d\ge 3 $ and 
$\frac{d}{d-1}< r\le \frac{d^2+1}{d(d-1)} $.
%or $d=3$ and $r> 4/3$. 
Let $\fP_d(r)$ be the pentagon (Figure~\ref{fig:intermediate r}) with vertices
\begin{align*}
& Q_1(r)= \big( \tfrac{1}{rd}, \tfrac{1}{rd}\big),
\quad Q_2= \big(\tfrac{d-1}{d}, \tfrac {d-1}{d} \big),
\quad Q_3=\big(\tfrac{d-1}{d} ,\tfrac{1}{d}\big)\\
&P(r)= \big(\tfrac 1r,\tfrac{d+1-r(d-1)}{r(d-1)}\big), \quad 
Q_4(r)=(1- \tfrac{d+1}{rd(d-1)},\tfrac{1}{rd}).
\end{align*}
Then 

(i) $V_r^I A: L^p\to L^q$ is bounded for $(\tfrac 1p,\tfrac 1q)$ in the interior of $\fP_d(r) $ and unbounded  for $(\tfrac 1p,\tfrac 1q)\notin \fP_d(r)$. 

(ii) $V_r^I A: L^p\to L^q$ is bounded for $(\tfrac 1p,\tfrac 1q)$ on the half open line segment $(Q_4(r), Q_1(r)]$  and on the half open line segment $[Q_1(r), Q_2)$.

(iii) $V_r^IA$ is of restricted strong type $(p,q)$ if $(\tfrac 1p, \tfrac 1q)$ belongs to the half open line segment
$[Q_2,Q_3)$. $V_r^I A$  fails to be of strong type on $[Q_2,Q_3]$. 

(iv) $V_r^IA$ is of restricted weak type $(p,q)$ if  $(\tfrac 1p,\tfrac 1q)=Q_3$.
\end{theorem}

Note that for $r=\tfrac{d^2+1}{d(d-1)}$ the pentagon $\fP_d(r)$ in Figure~\ref{fig:intermediate r} degenerates to a quadrangle, as $P(r)=Q_4(r)=Q_4$. We leave open what happens at the closed boundary segment $[Q_4(r), P(r)]$ and the half-open boundary segment $[P(r),Q_3)$.

%%%%%%%%%%%%%%%%%%%%%%%%%%%%
%%
%   figure intermediate r
%%
%%%%%%%%%%%%%%%%%%%%%%%%%%%%%

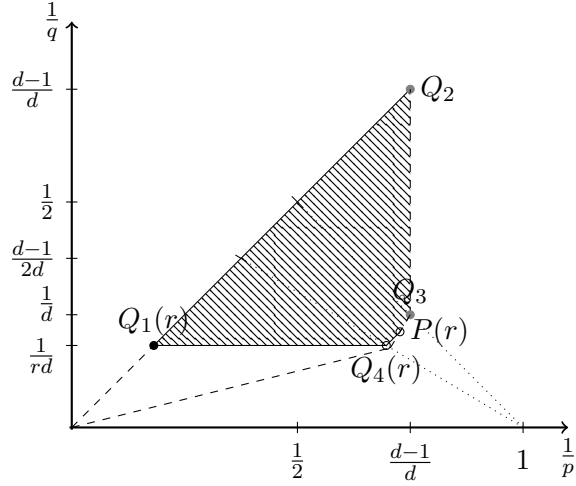
\begin{figure}[h]
\begin{tikzpicture}[scale=1.5] 

\begin{scope}[scale=2]
\draw[thick,->,black] (-.01,0) -- (2.2,0) node[below] {$ \frac 1 p$};
\draw[thick,->] (0,-.01) -- (0,1.80) node[left] {$ \frac 1 q$};  %x-axis

\draw (2,-0.025) -- (2,.025) node[below= 0.25cm] {$ 1$}; 
%\draw (-0.025,2) -- (.025,2) node[left= 0.25cm] {$ 1$}; 
%\node[circle,draw=black, fill=black, inner sep=0pt,minimum size=3pt] (b) at (2,2) {}; %Dot at (1,1) [deleted]
%\node[circle,draw=black, fill=black, inner sep=0pt,minimum size=3pt] (b) at (0,0) {};%dot at (0,0) [deleted]

%\draw (1,-0.025) -- (1,.025) node[below= 0.25cm] {$\frac 1 2$}; 
\draw (.025,1) -- (-.025,1) node[left] {$ \frac{1}{2}$};
%\node[circle,draw=black, fill=black, inner sep=0pt,minimum size=3pt] (b) at (1,1) {}; %dot at (1/2,1/2) replaced by tick

\draw (1,.025) -- (1,-.025) node[below] {$ \frac{1}{2}$};

\draw[ dotted]  (2,0)  -- (1,1);

\draw (16/11,14/33) circle (.5pt);%empty circle at P(r)

\draw[dashed] (24/17,6/17) -- (3/2, 1/2); %dashed from Q_4 to  Q_3 

\draw (16/11,14/33) node[right] {\textcolor{black}{$P(r)$}}; %label at P(r)

\node[circle,draw=black, fill=black, inner sep=0pt,minimum size=3pt] (b) at (4/11,4/11) {}; %dot at Q_1(r)

\draw (4/11,4/11) node[above] {\textcolor{black}{$Q_1(r)$}}; %label at Q_1(r)

%\draw (-0.025,1/10) -- (.025,1/10);
%\draw (0,1/10) node[left = 0.08cm] {\textcolor{blue}{$ \frac{1}{rd}$} };
\draw[color=black] (46/33,4/11) -- (4/11,4/11) ; %line from Q_4(r) to Q_1(r)
\draw[color=black] (46/33,4/11) -- (16/11,14/33) ; %line from Q_4(r) to P(r)
\draw[color=black] (4/11,4/11) -- (3/2,3/2); %line from Q_1(r) to Q_2
%\draw[dashed] (3.15/4,1/10) -- (1 + 1/2 + 1/10, 1/8+1/16 + 1/80);
\draw[ dashed] (0,0) -- (4/11,4/11); %dashed from (0,0) to Q_1(r)

\fill[pattern=north west lines, pattern color=black] 
(4/11,4/11) %Q_1(r)
-- (46/33,4/11) %Q_4(r)
-- (16/11,14/33) %P(r)
-- (3/2,1/2) %Q_3
-- (3/2,3/2) %Q_2
-- (4/11,4/11); %Q_1(r)

%%%% conjectured point

\draw (3/4-.025,3/4+.015) -- (3/4+.025,3/4-.015);
\draw (1-.025,1+.025) -- (1+.025,1-.025);

\draw (-0.025,3/4) -- (.025,3/4); %y-axis tick at (d-1)/2d
\draw (0,3/4) node[left = 0.08cm] {\textcolor{black}{$ \frac{d-1}{2d}$} }; %y-axis label at (d-1)/2d

\draw[ dotted, color=black] (2,0) -- (3/4,3/4) ; %dashed from (1,0) to (d-1/2d,d-1/2d)

%\draw[dashed] (0,0) -- (1,1/8) -- (1 + 1/2 + 1/10, 1/8+1/16 + 1/80);

\draw[dashed] (0,0) -- (24/17,6/17); %dashed from (0,0) to Q_4

%\node[circle,draw=lightgray, fill=lightgray, inner sep=0pt,minimum size=3pt] (b) at (46/33,4/11) {}; %dot at Q_4(r)

\draw (46/33,4/11) circle (.5pt);%empty circle at Q_4(r)

\draw (46/33,4/11) node[below] {\textcolor{black}{$Q_4(r)$}}; %label at Q_4

%%%%% endpoint spherical

\draw (-0.025,3/2) -- (.025,3/2); %y-axis tick at (d-1)/d
\draw (0,3/2) node[left = 0.08cm] {\textcolor{black}{$ \frac{d-1}{d}$} }; %y-axis label at (d-1)/d
%\draw[dashed] (1,1) -- (6.75/4,6.75/4);
\node[circle,draw=gray, fill=gray, inner sep=0pt,minimum size=3pt] (b) at (3/2,3/2) {}; %dot at Q_2
\draw (3/2,3/2) node[right] {\textcolor{black}{$Q_2$}}; %label at Q_2

\draw[dashed, color=black]  (3/2,1/2) -- (3/2,3/2); %dotted from Q_3 to Q_2

\node[circle,draw=gray, fill=gray, inner sep=0pt,minimum size=3pt] (b) at (3/2,1/2) {}; %dot at Q_3

\draw (3/2,1/2) node[above] {\textcolor{black}{$Q_3$}}; %label at Q_3

%%%%% at the line q=p'

%\node[circle,draw=black, fill=black, inner sep=0pt,minimum size=3pt] (b) at (3/2,1/2) {}; %dot at Q_3 [redundant]
\draw (3/2, -0.025) -- (3/2, .025); %x-axis tick at (d-1)/d
\draw (3/2, 0) node[below = 0.08cm] {\textcolor{black}{$ \frac{d-1}{d}$} }; %x-axis label at (d-1)/d
\draw (-0.025,1/2) -- (.025,1/2); %y-axis tick at 1/d
\draw (0,1/2+.05) node[left = 0.08cm] {\textcolor{black}{$ \frac{1}{d}$} }; %y-axis label at 1/d + .05
\draw (-0.025,4/11) -- (.025,4/11); %y-axis tick at 1/rd
\draw (0,3.5/11) node[left = 0.08cm] {\textcolor{black}{$ \frac{1}{rd}$} }; %y-axis label at 1/rd - .05

%%% closing the trapezoid/subpentagon

\draw[dashed, color=black]  (3/2,1/2) -- (16/11,14/33); %dashed from Q_3 to P(r)

%\draw (-.1,0) node[below] {$Q_1$};
\end{scope}

\end{tikzpicture}

\caption{The pentagon $\mathfrak{P}_d(r)$ for  $\frac{d}{d-1} < r \leq \frac{d^2+1}{d^2-d}$ and $d \geq 3$ (Theorem~\ref{thm:intermediate r}).  The outer (dashed) quadrangle is the region of boundedness for the maximal operator.  Shown with $d=4$ and $r=\frac{11}8$.} 
\label{fig:intermediate r}
%\vspace{1cm}
\end{figure}

%%%%%%%%%%%%%%%%%%%%%%%%%%%%
%%
%   figure small r
%%
%%%%%%%%%%%%%%%%%%%%%%%%%%%%%

\begin{figure}[h]
\begin{tikzpicture}[scale=1.5] 

\begin{scope}[scale=2]
\draw[thick,->] (-.01,0) -- (2.2,0) node[below] {$ \frac 1 p$};
\draw[thick,->] (0,-.01) -- (0,1.85) node[left] {$ \frac 1 q$};

\draw (2,-0.025) -- (2,.025) node[below= 0.25cm] {$ 1$};

\draw (.025,1) -- (-.025,1) node[left] {$ \frac{1}{2}$};

\draw[ dotted]  (2,0)  -- (1,1);

\draw (22/15,8/15) circle (.5pt);%empty circle at Q_3(r)

\draw (22/15,8/15) node[right] {\textcolor{black}{$Q_3(r)$}}; %label at Q_3(r)

\draw (22/15,22/15) circle (.5pt);%empty circle at Q_2(r)

\draw (22/15,22/15) node[above] {\textcolor{black}{$Q_2(r)$}}; %label at Q_2(r)

\draw[dashed] (22/15,22/15) -- (22/15,8/15); %dashed line Q_2(r) to Q_3(r)
\draw[color=black, dashed] (4/3,2/5) -- (22/15,8/15); %dashed Q_4(r) to Q_3(r)

\node[circle,draw=black, fill=black, inner sep=0pt,minimum size=3pt] (b) at (2/5,2/5) {}; %dot at Q_1(r)

\draw (2/5,2/5) node[above] {\textcolor{black}{$Q_1(r)$}}; %label at Q_1(r)

%\draw (-0.025,1/10) -- (.025,1/10);
%\draw (0,1/10) node[left = 0.08cm] {\textcolor{blue}{$ \frac{1}{rd}$} };
\draw[color=black] (4/3,2/5) -- (2/5,2/5) ;%Q_4(r) to Q_1(r)
\draw[color=black] (2/5,2/5) -- (22/15,22/15); %Q_1(r) to Q_2(r)
\draw[dashed] (22/15,22/15) -- (3/2,3/2); %dashed Q_2(r) to Q_2
%\draw[dashed] (3.15/4,1/10) -- (1 + 1/2 + 1/10, 1/8+1/16 + 1/80);
\draw[ dashed] (0,0) -- (2/5,2/5); %dashed (0,0) to Q_1(r)

\fill[pattern=north west lines, pattern color=black] 
(2/5,2/5) %Q_1(r)
-- (4/3,2/5) %Q_4(r)
-- (22/15,8/15) %Q_3(r)
--  (22/15,22/15) %Q_2(r)
-- (2/5,2/5);%Q_1(r)

%%%% conjectured point

%\node[circle,draw=black, fill=black, inner sep=0pt,minimum size=3pt] (b) at (3/4,3/4) {}; %((d-1)/2d,(d-1)/2d)

\draw (3/4-.025,3/4+.015) -- (3/4+.025,3/4-.015);
\draw (1-.025,1+.025) -- (1+.025,1-.025);

\draw (-0.025,3/4) -- (.025,3/4); %y-axis tick at (d-1)/2d
\draw (0,3/4) node[left = 0.08cm] {\textcolor{black}{$ \frac{d-1}{2d}$} }; %y-axis label at (d-1)/2d

\draw[ dotted, color=black] (2,0) -- (3/4,3/4) ; %(1,0) to ((d-1)/2d,(d-1)/2d)

%\draw[dashed] (0,0) -- (1,1/8) -- (1 + 1/2 + 1/10, 1/8+1/16 + 1/80);

\draw[dashed] (0,0) -- (24/17,6/17); %dashed to Q_4

%\node[circle,draw=black, fill=black, inner sep=0pt,minimum size=3pt] (b) at (24/17,6/17) {}; %dot at Q_4 [deleted]

%\draw (24/17,6/17 ) node[below] {\textcolor{black}{$Q_4$}}; %label at Q_4[deleted]

%\node[circle,draw=lightgray, fill=lightgray, inner sep=0pt,minimum size=3pt] (b) at (4/3,2/5) {}; %dot at Q_4(r)

\draw (4/3,2/5) circle (.5pt);%empty circle at Q_4(r)

\draw (4/3,2/5) node[below] {\textcolor{black}{$Q_4(r)$}}; %label at Q_4(r)

%%%%% endpoint spherical

\draw (-0.025,3/2) -- (.025,3/2); %y-axis tick at (d-1)/d
\draw (0,3/2) node[left = 0.08cm] {\textcolor{black}{$ \frac{d-1}{d}$} }; %y-axis labl at (d-1)/d

\draw[dashed, color=black]  (3/2,1/2) -- (3/2,3/2); %dashed Q_3 to Q_2

%%%%% at the line q=p'

%\node[circle,draw=black, fill=black, inner sep=0pt,minimum size=3pt] (b) at (3/2,1/2) {}; %dot at Q_3 [deleted]
\draw (3/2, -0.025) -- (3/2, .025); %x-axis tick at (d-1)/d
\draw (3/2, 0) node[below = 0.08cm] {\textcolor{black}{$ \frac{d-1}{d}$} }; %x-axis label at (d-1)/d
\draw (-0.025,1/2) -- (.025,1/2);%y-axis tick at 1/d
\draw (0,1/2+.05) node[left = 0.08cm] {\textcolor{black}{$ \frac{1}{d}$} }; %y-axis label at 1/d

\draw (-0.025,2/5) -- (.025,2/5); %y-axis tick at 1/rd
\draw (0,2/5-.05) node[left = 0.08cm] {\textcolor{black}{$ \frac{1}{rd}$} }; %y-axis label at 1/rd

%%% closing the trapezoid/subpentagon

\draw[dashed, color=black]  (3/2,1/2) -- (24/17,6/17); %dashed Q_3 to Q_4

%\draw (-.1,0) node[below] {$Q_1$}; %label Q_1 [deleeted]

\end{scope}

\end{tikzpicture}

\caption{The quadrangle $\mathfrak{Q}_d(r)$ for  $1 \leq r \leq \frac{d}{d-1}$ and $d \geq 4$ (Theorem~\ref{thm:small r}).  The outer (dashed) quadrangle is the boundedness region for the maximal function. Shown  with $d=4$ and $r=\frac54$. }
\label{fig:small r}
%\vspace{1cm}
\end{figure}
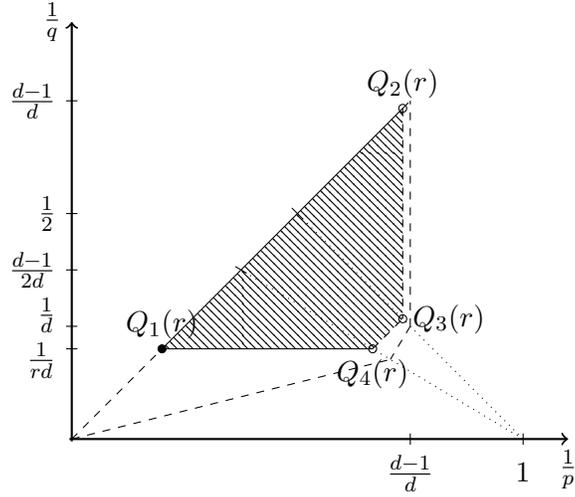

Finally, we address small values of $r$.

\begin{theorem}\label{thm:small r}  Suppose that  either  $d\ge 4$ and  $1\le r\le \tfrac d{d-1}$ or $d=3$ and $\tfrac{4}{3} < r \leq \tfrac{3}{2}$.
Let $\fQ_d(r)$ be the quadrangle (Figure~\ref{fig:small r})  with vertices
\begin{align*}
% Q_1= \big(\tfrac 1d, \tfrac 1d\big), \quad 
& Q_1(r)= \big(\tfrac{1}{rd}, \tfrac{1}{rd}\big), \quad 
Q_2(r)= \big(\tfrac{r(d-1)-1}{r(d-1)}, \tfrac{r(d-1)-1}{r(d-1)}), \\
& Q_3(r)=(\tfrac{r(d-1)-1}{r(d-1)} , \tfrac{1}{r(d-1)} ), \quad Q_4(r)= (1-\tfrac{d+1}{rd(d-1)}, \tfrac{1}{rd}).
\end{align*}
Then 

(i) $V_r^I A: L^p\to L^q$ is bounded for $(\tfrac 1p,\tfrac 1q)$ in the interior of $\fQ_d(r) $  and unbounded  for $(\tfrac 1p,\tfrac 1q)\notin \fQ_d(r)$.

(ii) $V_r^I A:L^p\to L^q$ is bounded if $(\frac 1p,\frac 1q)$ is in the half open line segment $(Q_4(r), Q_1(r)]$  and $[Q_1(r), Q_2(r))$.

(iii) For the case $r=1$, $d\ge 4$,  the operator $V_1^IA$ is of restricted weak type 
$(\tfrac{d-1}{d-2}, d-1)$ (that is, at $Q_3(1)$) and of restricted strong type $(\tfrac{d-1}{d-2}, q)$  for $\tfrac{d-1}{d-2}\le q<d-1$ (that is, on $[Q_2(1), Q_3(1))$). In three dimensions, $V_1^IA:L^2(\bbR^3)\to L^2(\bbR^3)$ is bounded.

\end{theorem}

We leave open what happens at the closed boundary segments $[Q_2(r), Q_3(r)]$ for $1 < r \leq \frac{d}{d-1}$ and $[Q_3(r),Q_4(r)]$ for $1\leq  r \leq \frac{d}{d-1}$.

\begin{figure}[H]
\includegraphics[width=4in]{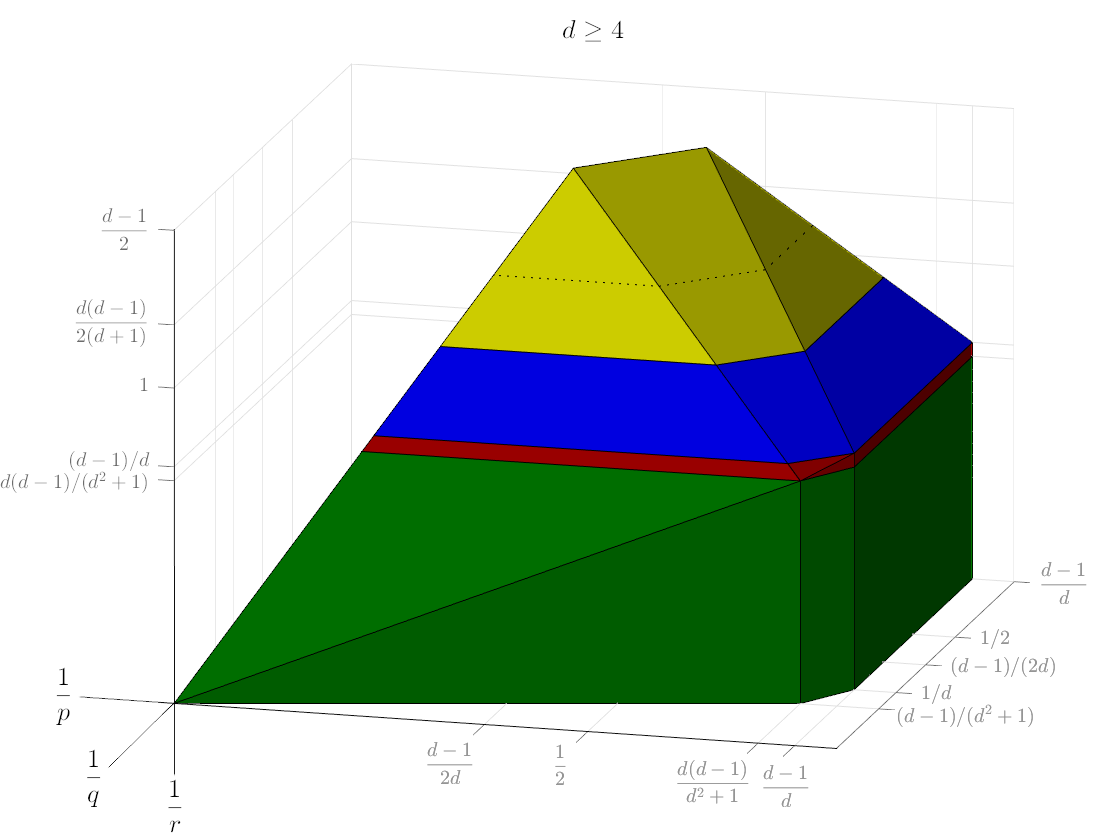}
\caption{A diagram of the typeset of $V_r^IA$ in $(\frac{1}{p},\frac{1}{q},\frac{1}{r})$-space for large values of $d$. The green region corresponds to Theorem \ref{dge3thm} (Figure \ref{fig:big r}), the red region corresponds to Theorem \ref{thm:intermediate r} (Figure \ref{fig:intermediate r}), and the blue region corresponds to Theorem \ref{thm:small r} (Figure \ref{fig:small r}). The yellow region is conjectural.}
\centering
\label{fig:3D d>=4}
\end{figure}

\begin{figure}[H]
\includegraphics[width=4in]{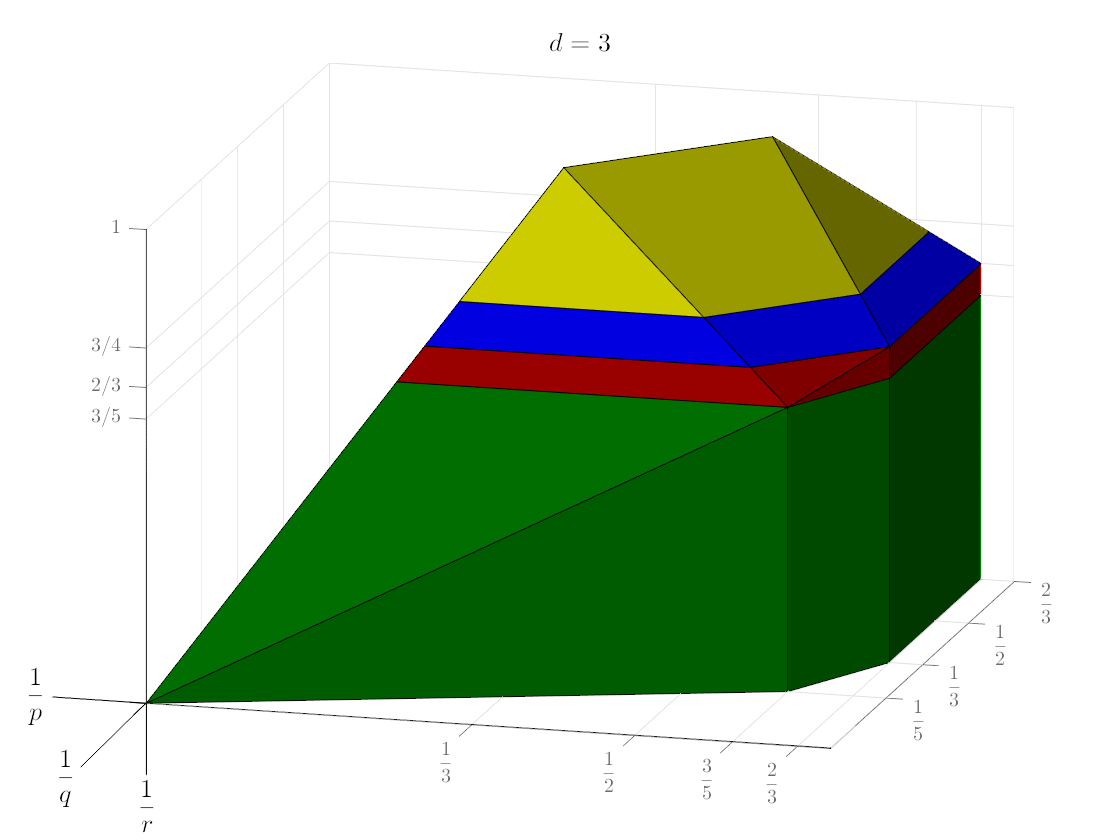}
\caption{A diagram of the typeset of $V_r^IA$ in $(\frac{1}{p},\frac{1}{q},\frac{1}{r})$-space for $d=3$. The green region corresponds to Theorem \ref{dge3thm} (Figure \ref{fig:big r}), the red region corresponds to Theorem \ref{thm:intermediate r} (Figure \ref{fig:intermediate r}), and the blue region corresponds to Theorem \ref{thm:small r} (Figure \ref{fig:small r}). The yellow region is conjectural.}
\centering
\label{fig:3D d=3}
\end{figure}

Note that there is a discrepancy in our results between $d=3$, for which we only obtain sharp results in the partial range $\frac{4}{3} < r \leq \frac{d}{d-1}$ and the case $d \geq 4$, where results are obtained for all $1 \leq r \leq \frac{d}{d-1}$. The reason is because we restrict ourselves to the traditional range $1 \leq r \leq \infty$ for the variation norm. The definition of $V_r$ can be extended, with modifications, to the range $0 < r <1$ (see for example \cite{BerghPeetre}). In that context, one can formulate conjectural results for $V_r^I A$ for $\frac{2}{d-1}< r < 1$ (see Figure \ref{fig:3D d>=4}) for $d \geq 4$. We remark that a positive solution to Sogge's local smoothing conjecture \cite{Sogge91} in $d+1$ dimensions would imply a complete result up to endpoints. Partial results in the range $ r > \frac{2(d+1)}{d(d-1)}$ can be proved using the techniques of this paper. We shall address issues for $r<1$ in a follow up paper. 

Similarly, in three dimensions, the range $1 \leq r \leq 4/3$ remains open as a conjecture (see Figure \ref{fig:3D d=3}). Note that here we are in the traditional range for the $V_r$ spaces.

In dimension 2, due to the recent full resolution of Sogge's problem in $2+1$ dimensions by Guth, Wang and Zhang \cite{GWZ}, that is, $$
\partial_t^{1/2-\varepsilon} A: L^4 \to L^4(L^4),
$$ it is possible to get an almost optimal result (up to endpoints) for the variation norm estimates. 
\begin{theorem} \label{d=2thm}
Let $d=2$.

(i) If $r> 5/2$ then $V_r^IA:L^p\to L^q$ is bounded if $(\tfrac 1p,\tfrac 1q) $ is either in the interior of the quadrangle $\fQ_2(r)$ (Figure~\ref{fig:d=2 big r-first}) formed by the vertices 
\begin{align*}
&P(r)= (\tfrac 1{r},\tfrac 1{2r} ), \quad  Q_1(r)=(\tfrac 1{2r}, \tfrac 1{2r}),\\
&  Q_2=Q_3= (\tfrac{1}{2}, \tfrac {1}{2} ) , \quad Q_4= (\tfrac{2}{5},\tfrac{1}{5})
\end{align*}
or in the open line segment between $Q_2=Q_3$ and $Q_1(r)$. 

(ii) If $2<r \leq 5/2$ then  $V_r^IA:L^p\to L^q$ is bounded 
if $(\tfrac 1p,\tfrac 1q) $ is either in the interior of the quadrangle $\fQ_2(r)$ (Figure~\ref{fig:d=2 big r-second}) formed by the vertices 
\begin{align*}
% Q_1= \big(\tfrac 1d, \tfrac 1d\big), \quad 
&Q_1(r)= \big(\tfrac{1}{2r}, \tfrac{1}{2r}\big), \quad 
Q_2=Q_3 = \big(\tfrac{1}{2}, \tfrac{1}{2}), \\
&P(r)=(\tfrac{1}{r}, \tfrac{3-r}{r}), \quad Q_4(r)= (1-\tfrac{3}{2r}, \tfrac{1}{2r})
\end{align*}
or in the open line segment between $Q_2=Q_3$ and $Q_1(r)$.

(iii) If {$r < 2$} then $V_r^I A$ does not map any  $L^p(\bbR^2)$ to any $L^q(\bbR^2)$. 
\end{theorem}

%%%%%%%%% figure 2 dimensions and r=5

\begin{figure}[H]
\begin{tikzpicture}[scale=1.8] 

\begin{scope}[scale=2]
\draw[thick,->] (-.01,0) -- (2.2,0) node[below] {$ \frac 1 p$};
\draw[thick,->] (0,-.01) -- (0,1.2) node[left] {$ \frac 1 q$};

\draw (1,-0.025) -- (1,.025); 
\draw (1,0) node[below= 0.08cm] {$\frac 1 2$}; \draw (.025,1) -- (-.025,1) node[left] {$ \frac{1}{2}$};
%\node[circle,draw=lightgray, fill=lightgray, inner sep=0pt,minimum size=3pt] (b) at (1,1) {};
\draw (1,1) circle (.5pt);%empty circle at Q_2=Q_3

\draw[ dotted]  (1,0)  -- (1,1); 
\draw[ dashed] (0,0) -- (1,1);
%\draw[ dashed] (1,0) -- (1,1);

%\node[circle,draw=black, fill=black, inner sep=0pt,minimum size=3pt] (b) at (0,0) {};

%%%% conjectured point

%\node[circle,draw=black, fill=black, inner sep=0pt,minimum size=3pt] (b) at (1/2,1/2) {};

\draw (1/2+.025,1/2-.025/3) -- (1/2-.025,1/2+.025/3);

\draw (-0.025,1/2) -- (.025,1/2);
\draw (0,1/2+.05) node[left = 0.08cm] {\textcolor{black}{$ \frac{1}{4}$} };

\draw[ dotted, color=black] (2,0) -- (1/2,1/2) ;

\draw (2, -0.025) -- (2, .025);
\draw (2,0) node[below = 0.08cm] {\textcolor{black}{$1$} };

%%%% 5/2

%\node[circle,draw=lightgray, fill=lightgray, inner sep=0pt,minimum size=3pt] (b) at (4/5,2/5) {}; %dot at Q_4
\draw (4/5,2/5) circle (.5pt);%empty circle at Q_4
\draw (4/5,2/5) node[right = 0.08cm] {\textcolor{black}{$Q_4$} }; %label at Q_4

\draw (1,1) node[above = 0.08cm] {\textcolor{black}{$Q_2=Q_3$} }; %label at Q_2=Q_3

\draw (4/5, -0.025) -- (4/5, .025);
\draw (4/5,0) node[below = 0.08cm] {\textcolor{black}{$ \frac{2}{5}$} }; 

\draw (-0.025,2/5) -- (.025,2/5);
\draw (0,2/5) node[left = 0.08cm] {\textcolor{black}{$ \frac{1}{5}$} }; %y-axis tick at 1/2r

\draw (-0.025,1/5) -- (.025,1/5);
\draw (0,1/5) node[left = 0.08cm] {\textcolor{black}{$ \frac{1}{2r}$} }; %y-axis label at 1/2r

\draw (1/5,1/5) node[above = 0.08cm] {\textcolor{black}{$Q_1(r)$} }; %label at Q_1(r)
%\node[circle,draw=lightgray, fill=lightgray, inner sep=0pt,minimum size=3pt] (b) at (1/5,1/5) {}; %dot at Q_1(r)
\draw (1/5,1/5) circle (.5pt);%empty circle at Q_1(r)

\draw (2/5,1/5) node[right = 0.08cm] {\textcolor{black}{$P(r)$} }; %label at P(r)
%\node[circle,draw=lightgray, fill=lightgray, inner sep=0pt,minimum size=3pt] (b) at (2/5,1/5) {}; %dot at P(r)
\draw (2/5,1/5) circle (.5pt);%empty circle at P(r)

%\draw[ dashed, color=green] (4/5,0) -- (1/2,1/2) ;

%%%%%% closing the triangle

\draw[ dashed] (1,1) -- (4/5,2/5) -- (0,0) ; %dashed from 0 to Q_4

\fill[pattern=north west lines, pattern color=black] (1,1) %Q_2=Q_3
-- (1/5,1/5) %Q_1(r)
-- (2/5,1/5) %P(r)
-- (4/5,2/5) %Q_4
-- (1,1); %Q_2=Q_3

\draw[dashed, color=black] (1/5,1/5) -- (2/5,1/5);
\draw[ color=black] (1/5,1/5) -- (1,1) ; %dashed Q_1(r) to P(r)

\draw (2/5, -0.025) -- (2/5, .025); %tick at 1/r
\draw (2/5,0) node[below = 0.08cm] {\textcolor{black}{$ \frac{1}{r}$} }; %label at 1/r

\end{scope}

\end{tikzpicture}

\caption{The region $\fQ_2(r)$ if $r >5/2$ (Theorem~\ref{d=2thm}~i).    The outer (dashed) triangle is the region of boundedness for the maximal operator. Shown with $r=5$.}
\label{fig:d=2 big r-first}
%\vspace{1cm}
\end{figure}
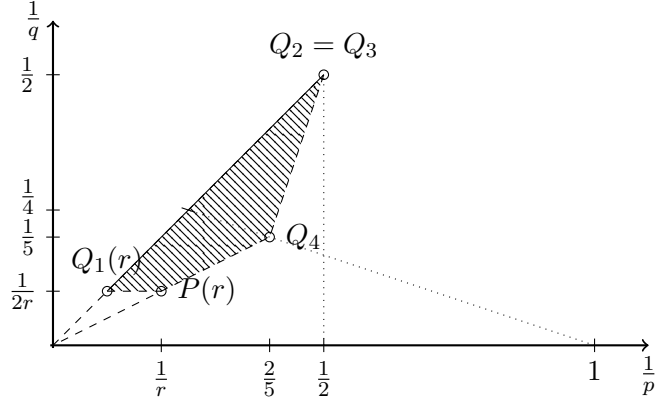

Note that, as for the circular maximal function theorem, the points $Q_2$ and $Q_3$ coincide if $d=2$; therefore the pentagon (Figures~\ref{fig:big r} and \ref{fig:intermediate r}) in Theorems \ref{dge3thm} and \ref{thm:intermediate r} becomes a quadrangle for $r>2$. Moreover, $P(5/2)=Q_4(5/2)=Q_4$, so the quadrangle becomes a triangle for $r=5/2$. The bounds are subsumed in Figure \ref{fig:3D d=2}; note that in contrast with $d \geq 3$, the blue/yellow region disappears, as $\frac{d}{d-1}=\frac{2}{d-1}$ coincide for $d=2$.

It is also possible to show unboundedness for $r=2$ via an argument involving the Besicovitch set, which will be addressed in a forthcoming paper.

We note that an affirmative answer to \textit{endpoint versions} of Sogge's problem as formulated and conjectured in \cite{HNS2011} would also settle strong type bounds on the half-open boundary segment $(Q_4(r), Q_1(r)]$. Unfortunately such endpoint bounds in Sogge's problem are currently only available in  dimensions four and higher.

%%%%%%%%% figure 2 dimensions and r=2.2

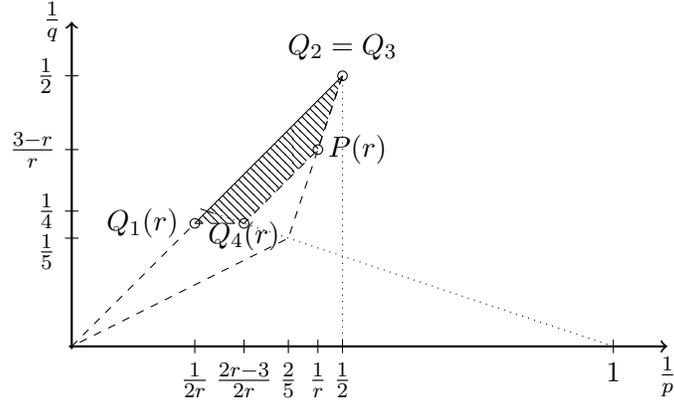
\begin{figure}[H]
\begin{tikzpicture}[scale=1.8] 

\begin{scope}[scale=2]
\draw[thick,->] (-.01,0) -- (2.2,0) node[below] {$ \frac 1 p$};
\draw[thick,->] (0,-.01) -- (0,1.2) node[left] {$ \frac 1 q$};

\draw (1,-0.025) -- (1,.025); 
\draw (1,0) node[below= 0.08cm] {$\frac 1 2$}; 
\draw (.025,1) -- (-.025,1) node[left] {$ \frac{1}{2}$};
%\node[circle,draw=lightgray, fill=lightgray, inner sep=0pt,minimum size=3pt] (b) at (1,1) {};
\draw (1,1) circle (.5pt);%empty circle at Q_2=Q_3

\draw[ dotted]  (1,0)  -- (1,1); 
\draw[ dashed] (0,0) -- (1/2.2,1/2.2); %(0,0) to Q_1(r)
%\draw[ dashed] (0,0) -- (1,1);
%\draw[ dashed] (1,0) -- (1,1);

%\node[circle,draw=black, fill=black, inner sep=0pt,minimum size=3pt] (b) at (0,0) {};

%%%%%%%%%%%%%%%%%%%%%%%%%%%%%%%%%%%%
%%%%%%%%%%%%New axis labels%%%%%%%%%%%%
%%%%%%%%%%%%%%%%%%%%%%%%%%%%%%%%%%%%

\draw (2/2.2,-0.025) -- (2/2.2,0.025);
\draw (2/2.2,0) node[below=0.08cm] {$\frac 1r$};

\draw (1/2.2,-0.025) -- (1/2.2,0.025);
\draw (1/2.2,0) node[below=0.08cm] {$\frac 1{2r}$};

\draw (2-3/2.2,-0.025) -- (2-3/2.2,0.025);
\draw (2-3/2.2,0) node[below=0.08cm] {$\frac {2r-3}{2r}$};

\draw (-0.025,6/2.2-2) -- (0.025,6/2.2-2);
\draw (0,6/2.2-2) node[left=0.08cm] {$\frac {3-r}{r}$};

%%%% conjectured point

%\node[circle,draw=black, fill=black, inner sep=0pt,minimum size=3pt] (b) at (1/2,1/2) {};

\draw (1/2-.025,1/2+.025/3) -- (1/2+.025,1/2-.025/3); %tick at conjectured point

\draw (-0.025,1/2) -- (.025,1/2);
\draw (0,1/2) node[left = 0.08cm] {\textcolor{black}{$ \frac{1}{4}$} };

\draw[ dotted, color=black] (2,0) -- (1/2,1/2) ;

\draw (2, -0.025) -- (2, .025);
\draw (2,0) node[below = 0.08cm] {\textcolor{black}{$1$} };

%%%% 5/2

%\node[circle,draw=black, fill=black, inner sep=0pt,minimum size=3pt] (b) at (4/5,2/5) {}; %dot at Q_4 [deleted]
%\draw (4/5,2/5) node[right = 0.08cm] {\textcolor{black}{$Q_4$} };%label at Q_4 [deleted]

\draw (1,1) node[above = 0.08cm] {\textcolor{black}{$Q_2=Q_3$} };

\draw (4/5, -0.025) -- (4/5, .025); %x-axis tick at 2/5
\draw (4/5,0) node[below = 0.08cm] {\textcolor{black}{$ \frac{2}{5}$} }; %x-axis label at 2/5

\draw (-0.025,2/5) -- (.025,2/5); %y-axis tick at 1/5
\draw (0,2/5-.05) node[left = 0.08cm] {\textcolor{black}{$ \frac{1}{5}$} }; %y-axis label at 1/5

%\draw (-0.025,1/5) -- (.025,1/5);
%\draw (0,1/5) node[left = 0.08cm] {\textcolor{blue}{$ \frac{1}{2r}$} };

\draw (1/2.2,1/2.2) node[left = 0.08cm] {\textcolor{black}{$Q_1(r)$} }; %label at Q_1(r)
%\node[circle,draw=lightgray, fill=lightgray, inner sep=0pt,minimum size=3pt] (b) at (1/2.2,1/2.2) {}; %dot at Q_1(r)

\draw (1/2.2,1/2.2) circle (.5pt);%empty circle at Q_1(r)

\draw (1.4/2.2,1.1/2.2) node[below] {\textcolor{black}{$Q_4(r)$} }; %label at Q_4(r)
%\node[circle,draw=lightgray, fill=lightgray, inner sep=0pt,minimum size=3pt] (b) at (1.4/2.2,1/2.2) {}; %dot at Q_4(r)

\draw (1.4/2.2,1/2.2) circle (.5pt);%empty circle at Q_4(r)

\draw (1/1.1,0.8/1.1) node[right] {\textcolor{black}{$P(r)$} }; %label at P(r)
%\node[circle,draw=lightgray, fill=lightgray, inner sep=0pt,minimum size=3pt] (b) at (1/1.1,0.8/1.1) {}; %dot at P(r)

\draw (1/1.1,0.8/1.1) circle (.5pt) ;%empty circle at P(r)

%\draw[ dashed, color=green] (4/5,0) -- (1/2,1/2) ;

%%%%%% closing the triangle

\draw[ dashed] (1,1)
-- (4/5,2/5) %Q_4
-- (0,0) ;
\draw[ color=black] (1/2.2,1/2.2) -- (1,1) ; %Q_1(r) to Q_2=Q_3

\fill[pattern=north west lines, pattern color=black] (1,1) -- (1/2.2,1/2.2) %Q_1(r)
-- (1.4/2.2,1/2.2) %Q_4(r)
--  (1/1.1,0.8/1.1) %P(r)
-- (1,1);%Q_2=Q_3

\draw[dashed, color=black] (1/2.2,1/2.2)  %Q_1(r)
-- (1.4/2.2,1/2.2) %Q_4(r)
-- (1/1.1,0.8/1.1) %P(r)
-- (1,1); %Q_2=Q_3

\end{scope}

\end{tikzpicture}

\caption{The region $\fQ_2(r)$ if $d=2$ and $2 < r \leq 5/2$ (Theorem~\ref{d=2thm}~ii).    The outer (dashed) triangle is the region of boundedness for the maximal operator. Shown with $r=2.2$. }
\label{fig:d=2 big r-second}
%\vspace{1cm}
\end{figure}

\begin{figure}[h]
\includegraphics[width=4in]{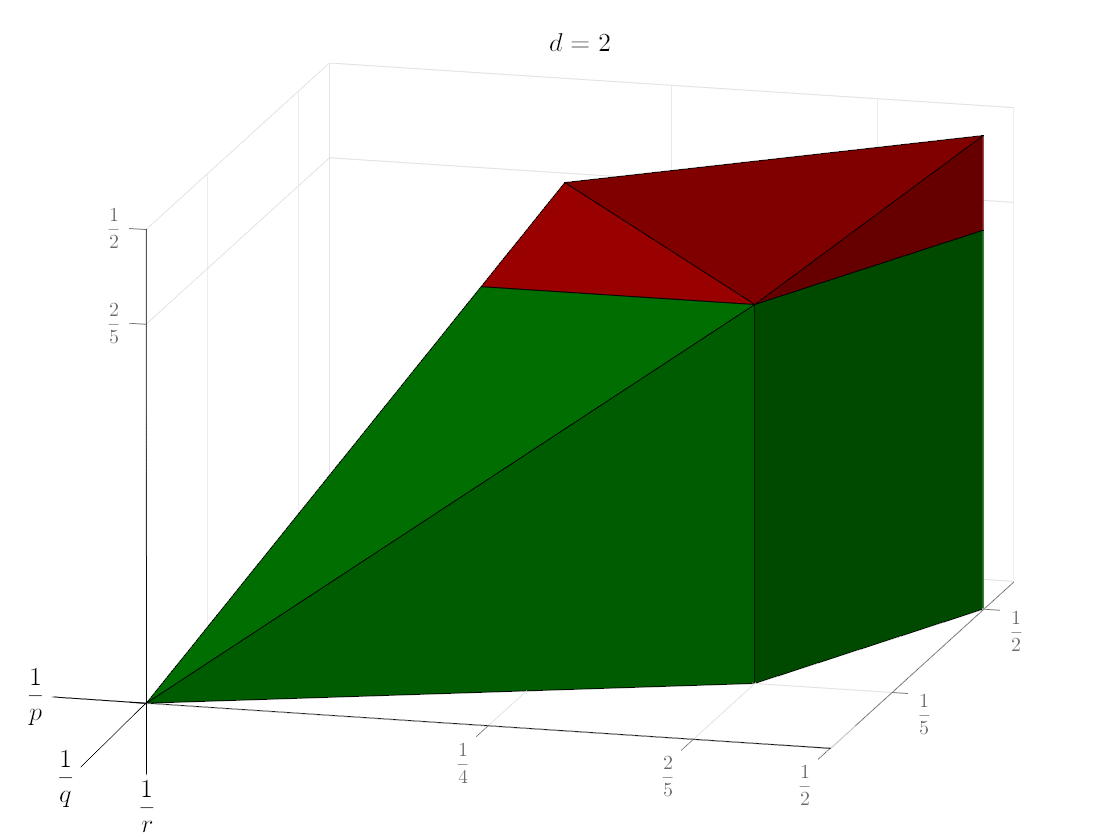}
\caption{A diagram of the typeset of $V_r^IA$ in $(\frac{1}{p},\frac{1}{q},\frac{1}{r})$-space for $d=2$. The green region corresponds to Figure \ref{fig:d=2 big r-first} and the red region corresponds to Figure \ref{fig:d=2 big r-second}.} 
\centering
\label{fig:3D d=2}
\end{figure}

\subsection*{Sparse domination} We now  formulate a sparse domination result for the global operator $V_rA$, $r>2$. Recall that a family of cubes $\mathfrak{S}$  in $\R^d$ is called \textit{sparse} if for
every $Q\in \fS$ there is a  measurable subset $E_Q\subset Q$ such that $|E_Q|\ge  |Q|/2$ and such that the sets on the family $\{E_Q : Q \in \fS\}$ are pairwise disjoint. In what follows we abbreviate $\langle f\rangle_{Q,s}=(|Q|^{-1} \int_Q |f|^s)^{1/s} .$

\begin{theorem}\label{cor:sparse}
Assume one of the following holds:
\begin{enumerate}[(i)]
    \item $d\ge 3$, $r>2$, and $(\tfrac 1p, \tfrac 1q)$ in the interior of $\fP_d(r)$.
    \item $d=2$, $r>2$ and $(\tfrac 1p, \tfrac 1q)$ in the interior of $\fQ_2(r)$.
\end{enumerate}
Then there is a constant $C=C(p,q)$ such that for each pair of compactly supported bounded functions $f_1$, $f_2$ there is a sparse family of cubes $\mathfrak{S}$ such that
\Be\label{eqn:sparsevar}
\int_{\R^d} V_rA f_1(x) f_2(x) \ud x\le C \sum_{Q \in \mathfrak{S}}|Q| \langle f_1 \rangle_{Q,p} \langle f_2 \rangle_{Q, q'},
\Ee 
where $\frac{1}{q}+\frac{1}{q'}=1$.
Furthermore, the $(1/p,1/q)$ range 
is sharp up to endpoints in the sense that no such result can hold if $(1/p,1/q)$ does not lie in the closure of $\fP_d(r)$, or $\fQ_2(r)$, respectively.
\end{theorem}
Theorem \ref{cor:sparse}  can be obtained  as an immediate consequence of a (more general) sparse domination result in \cite{BRS-sparse}, together with the $L^p$ results in \cite{JSW} and  Theorems \ref{dge3thm} and
\ref{d=2thm}; see \S \ref{sparse-sharp} and  \S\ref{sparse-section}.
Sparse domination is known to imply as a corollary a number of weighted inequalities in the context of Muckenhoupt and reverse Hölder classes. We refer the interested reader to \cite{BFP} for the weighted consequences for $V_rA$ of Theorem \ref{cor:sparse}.

\subsection*{Overview of the argument}  
Our positive results can be classified in the following 4 groups of estimates:
\begin{itemize}
    \item The $L^p \to L^q$ bounds for $V_r^I$ in the interior of the regions in Theorems \ref{dge3thm}, \ref{thm:intermediate r}, \ref{thm:small r} and \ref{d=2thm}. These can all be obtained from a single scale frequency analysis. More precisely, for a triple $(p,q,r)$, one obtains bounds which decay geometrically with respect to the frequency scale. By a Besov space embedding, these can be obtained from suitable space-time $L^p \to L^q(L^r)$ bounds for the spherical averages $A_t$, which in turn follow from interpolation (and a localization argument) among the basic estimates $L^1 \to L^\infty(L^1), L^2 \to L^2(L^2)$, $L^\infty \to L^\infty(L^\infty)$ and $L^1 \to L^1(L^1)$ together with key local smoothing estimates, such as the Guth--Wang--Zhang \cite{GWZ} result for $d=2$, or a Stein--Tomas type estimate for $d \geq 3$.
    \item The Lebesgue space bounds for $V_r^I$ at certain boundaries of the regions in Theorems \ref{dge3thm}, \ref{thm:intermediate r}, \ref{thm:small r} and \ref{d=2thm}. Except for $q=rd$ and the segment $[P(r),Q_4]$ in Theorem \ref{dge3thm}, the remaining claimed estimates along the boundary can be obtained by combining the single scale frequency bounds from the previous item and Bourgain's interpolation argument.
    \item The $L^p \to L^q$ bounds for $V_r^I$ at the boundary segment with $q=rd$ in Theorems \ref{dge3thm}, \ref{thm:intermediate r}, \ref{thm:small r}. These estimates are harder to obtain and require a more delicate analysis. In contrast to the previous cases, we use a multi-scale frequency analysis and the Fefferman--Stein sharp maximal function. Effective bounds follow from exploiting cancellation, the local nature of the estimates, and a careful real interpolation of certain localised pieces. This also subsumes the boundary segment $[P(r),Q_4]$ in Theorem \ref{dge3thm} and Corollary \ref{cor:lower boundary edge}.
    \item The restricted weak type estimate for $V_{p/d}$ in Theorem \ref{endptvarthm}. Again, this estimate is of a harder nature. In view of the better bounds satisfied by the \textit{long variation} operator, which were proven in \cite{JSW}, it suffices to prove the endpoint bound for the so-called \textit{short} variation operator, which corresponds to the $\ell^r(\mathbb{Z})$-norm of the map $k \to V_r^{I_k} A$, where $I_k=[2^k,2^{k+1}]$. We perform a single  scale in frequency but a multi-scale in time analysis, for which we combine the previous techniques. In particular, we use the Fefferman--Stein maximal function to deal with multi-scale time sums and prove estimates for single frequency pieces of the short variation operator. These can then be combined with Bourgain's interpolation argument to obtain the desired endpoint result.
\end{itemize}

\subsection*{Structure of the paper}  We start gathering some well known facts about spherical averages and function spaces in \S \ref{sec:preliminaries}. In \S \ref{sec:sharpness} we provide the examples showing the necessary conditions for our theorems. In \S \ref{freq-loc-section} we exploit the single frequency analysis to deduce the claimed bounds in the interior of the regions, as well as some restricted weak and strong type endpoints, in Theorems \ref{dge3thm}, \ref{thm:intermediate r}, \ref{thm:small r} and \ref{d=2thm}. The proof of the harder off-diagonal strong type boundary results in those theorems, and therefore Corollary \ref{cor:lower boundary edge}, is provided in \S\S \ref{maxopsect}-\ref{sec:strong type endpoint}. In \S \ref{sec:endpoint global} we prove the restricted weak type inequality for the global operator in Theorem \ref{endptvarthm}. Finally, the sparse domination result is 
discussed in \S \ref{sparse-section}.

\subsection*{Acknowledgements} We are indebted to Shaoming Guo for useful contributions at various stages of the project.
Some initial work on this  project was done  during the 
workshop {``Sparse domination of singular integral operators''} in October 2017, attended by four of the authors. We would like to thank the American Institute of Mathematics for hosting the
workshop, as well as the organizers Amalia Culiuc, Francesco Di Plinio and Yumeng Ou.  
D.B. was partially supported by the NSF grant DMS-1954479.  L.R. was partially supported by the Basque Government through the BERC 2018-2021 program, by the Spanish Ministry of Economy and Competitiveness: BCAM Severo Ochoa excellence accreditation SEV-2017-2018 and through project MTM2017-82160-C2-1-P, by the project RYC2018-025477-I, and by Ikerbasque. A.S. was partially supported by NSF grant DMS-1764295 and by a Simons fellowship. B.S.  was partially supported by NSF grant DMS-1653264.

\section{Preliminaries}\label{sec:preliminaries}

It will be convenient to consider the $t$-parameter as a variable. To this end, let $\chi\in C^\infty_c (\bbR)$  so that $\chi(t)=1$ for $t$ in a neighborhood of $[1,2]$ and supported in  $[1/2,4]$, and  define 
\begin{equation}\label{cA def}
\cA f(x,t) := \chi(t) A_t f(x). 
\end{equation}
In view of future frequency decompositions, let $\beta_0\in C^\infty_c(\bbR)$ so that $\beta_0(s)=1$ for $|s|<1/2$ and $\beta_0(s)=0$ for $|s|>1$. For every integer $j \geq 1$, set 
%For $j\ge 1$ let  
\Be\label{defofbetaj} 
\notag\beta_j(s)= \beta_0(2^{-j} s)-\beta_0(2^{1-j}s).
\Ee%, \quad j=1,2,\dots\Ee
For functions $g$ on $ \bbR$, and $l \in \N_0$, define the operators  $\Lambda_l$ by 
\begin{equation}\label{eq:freq Lambda}
\widehat {\Lambda_l g}(\tau)=\beta_l(\tau)\widehat g(\tau).
\end{equation}
For functions $f$ on $\R^d$, and $j \in \N_0$, define the operators $L_j$ by
\begin{equation}\label{eq:freq Lj}
\widehat{L_j f} (\xi)=\beta_j(|\xi|) \widehat{f}(\xi),
\end{equation}
and let $\widetilde{L}_j$ be a modification of $L_j$ satisfying $\widetilde{L}_j L_j=L_j$. 

 \subsection{$V_r$ and related function spaces} 
 It will be convenient to work with the Besov space $B^{1/r}_{r,1}$. 
 The Besov spaces $B^s_{p,q}(\bbR)$ can be defined using the dyadic frequency decompositions $\{\Lambda_l\}_{l=0}^\infty$ on the real line 
and we have  $\|u\|_{B^s_{p,q}} = (\sum_{l=0}^\infty [2^{ls}\|\Lambda_l u\|_{L^p}]^q)^{1/q}$.  From  the Plancherel--P\'olya inequality we know the embedding
\Be\label{BPe} B^{1/r}_{r,1}\hookrightarrow V_r \hookrightarrow B^{1/r}_{r,\infty}, \quad \Ee
see
\cite[Ch.1]{Triebel1983}. One can also consult the paper by Bergh and Peetre  \cite{BerghPeetre} (who however work with a different type  of variation space when $r=1$) or refer to \cite[Proposition 2.2]{GuoRoosYung}.
 Thus an inequality for the variation operator  $V_r^I\cA$   follows if we can control the $B^{1/r}_{r,1}$ norm of $t\mapsto \cA f(x,t)$.

Note that, by our definition, $V_1(\R)$ coincides with the space of bounded functions of bounded variations.
The fundamental theorem of calculus implies 
\begin{equation}\label{eq:V1 BV}
\| V_1^E A \|_{L^p \to L^q} \leq \| \partial_t \mathcal{A} \|_{L^p \to L^q(L^1(E))},
\end{equation}
so we shall focus on obtaining bounds for the right-hand side when studying $V_1^E A$. 

\subsection{Frequency decomposition in space}

Given $j \geq 0$, write \Be\label{eqn:Kjtdef} A_t L_j f=K_{j,t}*f,\Ee
where $L_j$ is as in \eqref{eq:freq Lj}, so that $\widehat {K_{j,t}}(\xi)  =\widehat  \sigma(t\xi) \beta_j(|\xi|)$. Note that $K_{j,t}$ is a Schwartz convolution kernel and therefore we restrict our attention to the case $j \geq 1$.

An immediate computation yields the following pointwise estimates for the convolution kernel.

\begin{lemma} For all $N\in \N_0$, there exists a constant $C_N>0$ such that
\begin{equation}\label{eq:trivial kernel}
    |\partial_t^{\varsigma} K_{j,t}(x)| \lesssim_{\varsigma} C_N 2^{j \varsigma} \frac{2^j}{(1+2^j \big| |x| - t \big|)^N}
    \end{equation}
holds for all $x \in \R^d$, all $t>0$ and all $\varsigma \in \N_0$. Consequently, 
\Be
\label{kernel-error}
|K_{j,t}(x)|\lesssim_N (2^j|x|)^{-N}\qquad \text{ if} \quad |x|\ge 10, \quad t\in [1/2,4].
\Ee
\end{lemma}

In analogy to the definition of $\cA$ in \eqref{cA def}, define 
\begin{equation*}
 \cA_j f(x,t) := \chi(t) A_t L_j f(x)=\chi(t) K_{j,t}* f(x).
\end{equation*}
We gather some estimates for $\cA_j$ when the inequalities involve $L^1$ or $L^\infty$ spaces.

First, from the trivial fact that $\| A_t f \|_{L^\infty} \lesssim \| f \|_{L^\infty}$ uniformly in $t \in \R$, one immediately has
\Be\label{maxinfty}
\| \mathcal{A}_j f\|_{L^{\infty}(L^\infty)} \lesssim \| f\|_{L^\infty}.
\Ee
Moreover, one has the following estimates for $L^1$ functions.
\begin{lemma} \label{V1lemma} For $1\le q\le\infty$,
\begin{equation*}
    \| \cA_j f\|_{L^q(L^1)} + 2^{-j} \|\partial_t \cA_j f\|_{L^q(L^1)}  \lc  \|f\|_{L^1}.
\end{equation*} 
\end{lemma}

\begin{proof} By \eqref{eq:trivial kernel} one has 
\begin{align}\label{pointwisekernelest}
\big| \cA_j f(x,t)\big| + 2^{-j} 
    \big|\partial_t \cA_j f(x,t)\big| 
    \lc   \int_{\R^d}|f(y)| \frac{2^j}{ (1+2^j| |x-y|- t|)^{N}}\,\ud y
\end{align}
for all $N \in \N_0$. Integrating in $t$ over the support of $\chi$ one sees that, for fixed $x$,
\begin{align*} 
&\int_{1/2}^4\big|\cA_j f(x,t)\big| \ud t + 2^{-j}
\int_{1/2}^4\big|\partial_t \cA_j f(x,t)\big| \ud t
\\ &\lc   \int_{\R^d} |f(y)| \int_{1/2}^4  \frac {2^j}{ (1+2^j| |x-y|- t|)^{N} }\ud t\,\ud y
\,\lc  \|f\|_{L^1}.  
\end{align*}
This gives the assertion for $q=\infty$. 

For $q=1$, the result follows from integrating in $x$ instead, using the decay in \eqref{pointwisekernelest} and taking into account that the integration in $t$ is over $[1/2,4]$.

The remaining cases $1 < q < \infty$ follow from combining the above through Young's convolution inequality.
\end{proof} 

\begin{corollary}\label{cor:1inftyr}
For $1\le r\le \infty,$
\[\|\cA_j f\|_{L^\infty(L^r)} \lc 2^{j(1-\frac 1r)} \|f\|_{L^1}.\]
\end{corollary}
\begin{proof}
Interpolate between
$$\| \cA_j \|_{L^\infty(L^\infty)} \lesssim 2^j \| f \|_{L^1},$$
which follows from \eqref{eq:trivial kernel}, and Lemma \ref{V1lemma} with $q=\infty$.
\end{proof}

\subsection{Oscillatory integral representation} Given $m \in \R$, let $S^m(\R^d)$ denote the class of all functions $a \in C^\infty (\R^d)$ satisfying
$$
|\partial^\alpha a_{}(\xi)| \lesssim_\alpha (1+|\xi|)^{m - |\alpha|}
$$
for all multiindex $\alpha \in \N_0^d$ and all $\xi \in \R^d$. Given $a \in S^m(\R)$, define
\Be 
\label{Tjs} T_j^{\pm}[a,f](x,t) =  \int_{\R^d} \beta_j ( |\xi|) a(t|\xi|) e^{i\inn{x}{\xi} \pm it|\xi|} \widehat f(\xi)\ud\xi.
\Ee
It is well known that the Fourier transform of the spherical measure is
\begin{equation*}
    \widehat{\sigma}(\xi)= (2 \pi)^{d/2} |\xi|^{-(d-2)/2} J_{\frac{d-2}{2}}( |\xi|)=  b_0(|\xi|)+ \sum_{\pm} b_{\pm}(|\xi|) e^{ \pm  i |\xi|},
\end{equation*}
where $b_0  \in C^\infty_c(\R)$ is  supported in $\{|\xi| \leq 1\}$ and $b_\pm \in S^{-(d-1)/2}(\R)$ are supported in $\{ |\xi| \geq 1/2\}$ (\textit{c.f.} \cite[Chapter VIII]{bigStein}). Thus one can write
%We recall some basic properties of the operators $\cA_j$. Stationary phase calculations show that 
\Be\label{Tjs-split}
\cA_j f(x,t)= 2^{-j(d-1)/2} (2\pi)^{-d} \sum_\pm T_j^\pm [a_{\pm},f] (x,t) \chi(t)
\Ee
where $a_{\pm} \in S^0(\R)$. We note  that the kernel estimate \eqref{eq:trivial kernel} could also be obtained through integration by parts in \eqref{Tjs} using the above representation.
It is clear from the expression of $T_j^\pm$ that
$$
\partial_t \big( T_j^\pm [a,f](x,t) \chi(t) \big) = T_j^\pm[a,f](x,t) \chi'(t)  +  T_j^\pm [\widetilde{a}, f](x,t) \chi(t)
$$
where $\widetilde{a}(\xi)= a'(t|\xi|)|\xi|  \pm i |\xi| a(\xi)$. This and Plancherel's theorem yield
\begin{equation}\label{eq:L2}
    \| \mathcal{A}_j f \|_{L^2(L^2)} \lesssim 2^{-j(d-1)/2} \| f \|_{L^2}, \qquad \| \partial_t \mathcal{A}_j f \|_{L^2(L^2)} \lesssim 2^{-j(d-3)/2} \| f \|_{L^2}.
\end{equation}

\subsection{A Stein--Tomas estimate}

In \cite{JSW}, in order to obtain $L^p$ bounds for the global $V_rA$, the estimate 
\begin{equation}\label{eq:p-p ST}
\Big\| \Big( \int_1^2 |e^{i t \sqrt{-\Delta}} L_j f |^2 \ud t \Big)^{1/2} \Big\|_{L^p} \lesssim 2^{j ( d(\frac{1}{2} - \frac{1}{p}) - \frac{1}{2} + \varepsilon )} \| f \|_{L^p}
\end{equation}
with $\eps>0$ 
is used  for $\frac{2(d+1)}{d-1} \leq  p < \infty$ if $d \geq 3$; it holds 
for  $4 <p< \infty$ if $d=2$.
This statement is closely related to estimates 
for Stein's square-function generated by Bochner--Riesz multipliers 
in \cite{Carbery1983}, \cite{Christ} and \cite{Seeger-crelle1986}, and the connection is given by the theorem of Kaneko and Sunouchi \cite{KanekoSunouchi}. See also  \cite{LeeRogersSeeger2014} for endpoint bounds and historical remarks, and \cite{LRSw},  \cite{Lee-Sanghyuk2018} for  recent  work on Stein's square function.
The  Stein--Tomas $L^2$ Fourier restriction theorem together with a localization result (\cf. Lemma \ref{localization-lemma} below) yields an analogue of \eqref{eq:p-p ST} with $\eps=0$ for $p \geq  \frac{2(d+1)}{d-1}$.  The method is well known \cite{FeBR} but we include the statement with a proof for  completeness.
\begin{lemma}\label{stein-squarefct}
Let $\tfrac{2(d+1)}{d-1}  \le q \le\infty$. Then for all $j\geq 0$,
\begin{align*}
\|\cA_j f\|_{L^q(L^2)}\lesssim  2^{-jd/q} \|f\|_{L^2}. 
\end{align*}
\end{lemma}

\begin{proof}
We use the oscillatory integral representation in \eqref{Tjs-split} and \eqref{Tjs}. We only discuss the estimate for $T_j^+[a, f](x,t) \chi(t)$ and abbreviate it with $T_j f(x,t)$ (the corresponding estimate for $T_j^-$ is analogous). It then  suffices to show 
\Be\notag  2^{-j(d-1)/2}\|T_j f\|_{L^q(L^2)} \lc 2^{-jd/q}  \|f\|_{L^2}, \qquad \tfrac{2(d+1)}{d-1}\le q\le\infty.\Ee

Let
\[\widetilde T_jg(x,t) = \chi(t) \int_{\R^d}  \beta_j(|\xi|)a(t|\xi|)e^{-it|\xi|} \widehat g(\xi,t) e^{i\inn x\xi} \ud \xi \] and observe that in view of the support of $\chi$ we have $\widetilde T_j g(\cdot,t)=0$ for $t\notin [1/2,4]$. 
By a duality argument, it suffices to show that for  $g  \in L^{p}(L^2)$
the inequality
\begin{equation}\label{dualform}
  \Big\|\int \widetilde T_j g(\cdot,t) \ud t \Big\|_{L^2}\lc 2^{j(\frac dp-\frac d2-\frac 12)} \|g\|_{L^p(L^2)},\qquad 
    1\le p\le \tfrac{2(d+1)}{d+3}
\end{equation}
holds.  By Plancherel's  theorem the square of the left-hand side is equal to
\begin{align*} 
&\int_{\R^d} \Big|\int\chi(t) \beta_j(|\xi|)a(t|\xi|)e^{-it|\xi|} \widehat g(\xi,t)\ud t\Big|^2\ud\xi\\&=
\int_0^\infty \int_{S^{d-1}}\Big| 
\int \chi(t) \beta_j(r)a(tr)e^{-itr} \widehat g(r\theta ,t) \ud t\Big|^2 \ud\theta\, r^{d-1} \ud r. 
\end{align*}
We now apply the Stein--Tomas inequality for the Fourier restriction operator for the sphere (valid for $1 \leq p\le 2(d+1)/(d+3)$), 
and see that the last expression is dominated by a constant times
\begin{align}\notag
    &\int_0^\infty \Big\| 
\int \chi(t) \beta_j(r)a(tr)e^{-itr} r^{-d} g(r^{-1}\cdot,t) \ud t \Big\|_{L^p}^2 r^{d-1} \ud r
\\ \notag
&\lc  \int_0^\infty \Big\| 
\int \chi(t) \beta_j(r)a(tr)e^{-itr}  g(\cdot,t) \ud t \Big\|_{L^p}^2 r^{\frac{2d}{p}-d-1} \ud r
\\
\label{afterST}
&\lc  \ 2^{2j(\frac dp-\frac d2-\frac 12)} \Big\|\Big(\int_0^\infty  \Big|
\int \chi(t) \beta_j(r) a(tr)e^{-itr}  g(\cdot,t) \ud t \Big|^2 \ud r\Big)^{1/2} \Big\|_{L^p}^2 
\end{align}
where in the last inequality we have used Minkowski's integral inequality. Next, observe that
\begin{align*}
&\int_0^\infty \Big|
\int \chi(t) \beta_j(r)a(tr)e^{-itr}  g(x,t) \ud t \Big|^2 \ud r\\
&=\,  \iint\int_0^\infty \chi(t) \chi(t') |\beta_j(r) |^2a(tr)\overline{a(t'r)} e^{i(t'-t)r} \ud r  \, g(x,t) \overline g(x,t') \ud t \ud t'.
\end{align*}
We integrate by parts in $r$ and then estimate the absolute value of the displayed  expression  by a constant times 
\begin{align*}
&\iint\frac{2^j}{(1+2^j|t-t'|)^2} |g(x,t)g(x,t') |\ud t\, \ud t'
\\ &= \int_{-\infty}^\infty\frac{2^j}{(1+2^j|h|)^2} \int|g(x,t)g(x,t+h) |
\ud t \ud h
\lc \int |g(x,t)|^2 \ud t
.\end{align*}
Using this in \eqref{afterST} yields \eqref{dualform} and hence the assertion.
\end{proof}

\subsection{Frequency decompositions in time}
In order to deduce Besov space estimates for $t \mapsto \cA_j f(x,t)$, we also work with a frequency decomposition in the $t$-variable.
We extend the definition of $\Lambda_l$ in \eqref{eq:freq Lambda}  to functions of $x$ and $t$ and apply  that decomposition to the operators $\mathcal{A}_j$ in the $t$-variable.

It is  useful to observe  that dyadic frequency decompositions in the  variable dual to $t$ essentially correspond  in our situation to dyadic frequency decompositions in the  variables dual to $x$. To see this, we show that the terms $\Lambda_l\cA_j$ are mostly negligible when $|j-l|\ge 10$. 
We write \[\La_l\cA_j f(x,t)= 2^{-j(d-1)/2}  (2\pi)^{-(d+1)} \sum_\pm \int_{\R^d} \ka^{\pm}_{j,l}(y,t ) f(x-y) \ud y \]
where, in view of \eqref{Tjs}, one has 
\begin{equation}\label{kajl-kernel}
    \ka^{\pm}_{j,l}(y,t )=%(2\pi)^{-d-1}
    \int_{\R} \int_{\R^d} e^{i\inn y\xi+it\tau}\beta_l(\tau)\beta_j(|\xi|)
    \int \chi(s)\, a_\pm(s\xi) e^{is (\pm|\xi|-\tau)}  \ud s \ud \xi\ud\tau .
\end{equation}

\begin{lemma}\label{errorlemma}
(i) For every $N \in \N_0$, there exists a finite $C_N >0$ such that \Be\label{ptwise-ka}
|\ka_{j,l}^\pm(y,t)|\leq C_N (1+|y|+|t|)^{-N} \min\{ 2^{-jN}, 2^{-lN}\}, \quad |j-l|\ge 10.
\Ee
(ii) Suppose $1\le p,r\le    q\leq \infty$. Then, there exists a finite $C_N(p,q,r)>0$ such that
\[ \|\Lambda_l \cA_j f\|_{L^q(L^r)} \le C_N(p,q,r) \min\{ 2^{-jN}, 2^{-lN}\} \| f \|_{L^p},
\quad |j-l| \geq 10.
\]
\end{lemma}
\begin{proof}
 Part (i) follows from \eqref{kajl-kernel} after multiple integration by parts in $s$ and subsequent integration by parts in $\xi, \tau$. Part (ii) is an immediate consequence of (i) using Minkowski's and Young's convolution inequality.
 \end{proof}

The above lemma allows one to only focus on the spatial frequency decomposition when looking for estimates of the type $L^p \to L^q(B_{r,1}^{1/r})$ for the operator $\mathcal{A}$ in most cases of interest.
In particular, we get the following.
\begin{corollary}\label{cor:trading-der}
Let  $s\in \bbR$, $1\le p,q,r \le\infty$. Then for all $j\in \bbN_0$, 
\[ \|\cA_j\|_{L^p\to L^q(B^s_{r,1})} \lc 2^{js} \|\cA_j\|_{L^p\to L^q(L^r)}+ C_N 2^{-jN} \]
\end{corollary}

\begin{proof} 
We write
$\| \cA_j f \big\|_{L^{q} (B^{s}_{r,1})} \le I+II$
where
\begin{align*}
  I&= \Big\| \sum_{\substack{l\ge0\\|j-l|\le 10}}  2^{l s } \Big\|\La_l \cA_j f\Big\|_{L^r(\bbR)} \Big\|_{L^{q}(\bbR^d)},
  \\
  II&= \Big\|  \sum_{\substack{l\ge 0\\ |j-l|>10}} 2^{ls } \Big\|\La_l \cA_j f\Big \|_{L^r(\bbR)} \Big\|_{L^{q}(\bbR^d)}.
  \end{align*}
  Clearly \[I\lc 2^{js} \|\cA_j f\|_{L^q(L^r)}\lc 2^{js} \|\cA_j\|_{L^p\to L^{q}(L^r)} \|f\|_{L^p}\]
  and by (ii) in Lemma \ref{errorlemma} \[ II \lc \sum_{l\ge 0} \min\{2^{-jN}, 2^{-lN}\} \|f\|_{L^p} \lc 2^{-jN} \|f\|_{L^p}.\]
  Combining both estimates, the assertion follows. 
\end{proof}
In certain endpoint estimates in \S\ref{sec:strong type endpoint}, we use an upgraded version of  Corollary \ref{cor:trading-der}  in conjunction with Littlewood--Paley theory,
as presented in the next lemma.

\begin{lemma} \label{lemma:besov reduction} 
Let $1\leq r < \infty$, $2 \leq q < \infty$, $1 < p < \infty$ such that $r,p \leq q$. Let $s \in \R$. 
Assume that for all $\{f_j\}_{j \geq 0}$ with $f_j \in L^p$,
\begin{equation}\label{eq:multi-scale assumption}
\Big\| \sum_{j\ge 0} \|\cA_j f_j\|_{L^r(\bbR)}\Big\|_{L^q(\bbR^d)}  \lc \Big(\sum_{j \geq 0} 2^{-jsq }\|
f_j\|_{L^p(\R^d)}^q\Big)^{1/q}
\end{equation}
holds. 
 Then
\begin{equation}\label{eq:Besov_conclusion} \big\| \cA f \big\|_{L^{q} (B^{s}_{r,1})} \lesssim  \|f\|_{L^p}. \end{equation}
\end{lemma}

\begin{proof}
Write $\| \cA f\|_{L^q(B_{r,1}^s)} \leq I+ II$, where $I$ and $II$ are as in the proof of Corollary \ref{cor:trading-der} but with an additional sum in the $j$-parameter. Recall that $\mathcal{A}_j f= \mathcal{A}_j (\widetilde{L}_j f)$. Applying the assumption \eqref{eq:multi-scale assumption} in $I$, one obtains
    \begin{align*} I&\lc
    \Big\|  \sum_{j=0}^\infty 2^{j s } \|\cA_j (\widetilde{L}_j f)\|_{L^r(\bbR)}  \Big\|_{L^{q}(\bbR^d)}
\lc \Big(\sum_{j=0}^\infty\|\widetilde L_j f\|_{L^p}^{q}\Big)^{\frac{1}{q}} 
\\
&\lc \Big\| \Big(\sum_{j=0}^\infty |\widetilde L_j f|^{q} \Big)^{\frac{1}{q}}\Big\|_{L^p} \lc
     \Big\| \Big(\sum_{j=0}^\infty |\widetilde L_j f|^{2} \Big)^{\frac{1}{2}}\Big\|_{L^p}
    \lc \|f\|_{L^p}
    \end{align*} 
since $q\ge 2$ and $1 < p \leq q < \infty$; note that the second line follows from Minkowski's inequality, the embedding $\ell^2 \hookrightarrow \ell^q$ and the Littlewood--Paley inequality. For the error term $II$, one applies (ii) in Lemma \ref{errorlemma} to obtain 
    \[
    II\lesssim_N \sum_{l\ge 0}\sum_{j\ge 0} 2^{ls}  \min\{2^{-lN}, 2^{-jN} \} \|f\|_{L^p}\lc \|f\|_{L^p}
    \]
for $N>s$. Combining both estimates, \eqref{eq:Besov_conclusion} follows.
\end{proof}

\begin{remark}
The previous lemma also extends to $q=\infty$ with the obvious notational modifications.
\end{remark}

\subsection{Bourgain's interpolation lemma} \label{sec:Binterpol} For the proof of restricted weak type inequalities we will repeatedly apply a result  of Bourgain \cite{Bourgain1985} that leads to restricted weak type inequalities in certain endpoint situations.
We cite the  abstract version of this lemma given in \cite[\S6.2]{CSWW1999} for the Lions--Peetre real interpolation spaces
(see \cite{BL1976}).

Let $\overline A=(A_0,A_1)$, $\overline B=(B_0,B_1)$ be compatible Banach spaces in the sense of interpolation theory. Let  $T_j:\overline A\to \overline  B$ be sublinear operators satisfying for all $j\in \bbZ$
\Be\label{eqn:Binterpol-assu}
\|T_j\|_{A_0\to B_0}\le C_0 2^{j\gamma_0}, \quad \|T_j\|_{A_1\to B_1} \le C_1 2^{-j\gamma_1},\,\quad\gamma_0, \gamma_1>0.
\Ee 
This assumption and real interpolation immediately gives $\|T_j\|_{\overline A_{\theta,\rho} \to \overline B_{\theta,\rho}}= O(1) $  for all $0<\rho \leq \infty$ 
and all $\theta =\gamma_0/(\gamma_0+\gamma_1)$, but one also gets a
weaker conclusion  for the sum of the operators.
%the operator can be summed in a re
 \begin{lemma} \label{lem:Binterpol} Suppose \eqref{eqn:Binterpol-assu} holds for all $j\in \bbZ$. Then 
 \Be\notag
 %\label{eqn:Binterpol-concl}
\Big\|\sum_j T_j\Big \|_{\overline A_{\theta,1}\to \overline B_{\theta,\infty}}\le C(\gamma_0,\gamma_1) C_0^{\frac{\gamma_1}{\gamma_0+\gamma_1}}C_1^{\frac{\gamma_0}{\gamma_0+\gamma_1}}.
\Ee
\end{lemma}

\section{Necessary conditions}\label{sec:sharpness}
In this section we modify known examples for the spherical maximal operators to give some necessary conditions for $L^p\to L^q$ boundedness of the local variation operator $V^I_rA$. For $r>\frac{d}{d-1}$ these  conditions show that $L^p\to L^q$ boundedness does not hold in the complement of the region $\fP_d(r)$ in Theorems \ref{dge3thm} and  \ref{thm:intermediate r} and the complement of $\fQ_2(r)$ in Theorem \ref{d=2thm}. For $1\le r\le \frac{d}{d-1}$ they show that $L^p\to L^q$ boundedness  does not   hold in the complement of $\fQ_d(r) $ defined in Theorem \ref{thm:small r}. They also show that $V_r^I$ is unbounded from any $L^p(\R^2)$ to any $L^q(\R^2)$ if $r<2$, that is, part (iii) in Theorem \ref{d=2thm}. Finally, we also prove sharpness of the sparse bounds in Theorem \ref{cor:sparse} up to the endpoints.

\subsection{Description of the edges} \label{sec:edges}
It will be helpful to make explicit the equations for the  edges of the boundedness regions in the above theorems.

(i) Consider the case $r>\tfrac{d^2+1}{d(d-1)}$ and the region  $\fP_d(r)$ in Theorem \ref{dge3thm}.  In this case the point $P(r)$ is on the line through $(0,0)$ and $Q_4$, which is given by 
 $\bigl\{\tfrac 1q=\tfrac{1}{dp}\bigr\}$. The boundary lines describing $\fP_d(r)$ are  
\begin{gather*}
\overline{P(r)Q_1(r)}= \bigl\{\tfrac 1q=\tfrac{1}{dr}\bigr\}, \qquad \overline{Q_1(r) Q_2} = \bigl\{\tfrac 1q=\tfrac 1p\bigr\}, \qquad \overline{Q_2 Q_3} = \bigl\{\tfrac 1p=\tfrac{d-1}{d}\bigr\}, \\
 \overline{Q_3 Q_4} = \bigl\{\tfrac 1q= \tfrac{d+1}{d-1}\tfrac 1p -1\bigr\}, \qquad \overline{Q_4 P(r)}= \bigl\{\tfrac 1q=\tfrac{1}{dp}\bigr\}.
\end{gather*}
If $d=2$, the points $Q_2$ and $Q_3$ coincide, and the lines $\overline{Q_1(r)Q_2}$, $\overline{Q_3Q_4}$, $\overline{Q_4P(r)}$ and $\overline{P(r)Q_1(r)}$ describe the quadrangle $\fQ_2(r)$ in Theorem \ref{d=2thm}, (i).

(ii) For the case $\tfrac{d}{d-1}<r\le \tfrac{d^2+1}{d(d-1)}$ the point $P(r)$ moves to the line  connecting $Q_3$ and $Q_4$ and only the part between $P(r)$ and $Q_3$ will be part of the boundary.
Note that for $r=\tfrac{d^2+1}{d(d-1)}$ the points $P(r)$ and $Q_4$ coincide so that the pentagon degenerates to a quadrangle. As $r\to \tfrac{d}{d-1}$ the point $P(r)$ moves to $Q_3$. 
The boundary lines of $\fP_d(r)$ in Theorem \ref{thm:intermediate r} are given in this case by
\begin{gather*}
 \overline{Q_1(r)Q_2} =\bigl\{\tfrac 1q=\tfrac 1p\bigr\}, \qquad\overline{Q_2Q_3} =\bigl\{\tfrac 1p=\tfrac{d-1}{d}\bigr\}, \\   \overline{Q_3P(r)} = \bigl\{\tfrac 1q= \tfrac{d+1}{d-1}\cdot\tfrac 1p -1\bigr\}, \qquad   \overline{P(r)Q_4(r)}= \bigl\{\tfrac{1}{q}=\tfrac 1p+ \tfrac{2}{r(d-1)} -1\bigr\},\\
\overline{Q_4(r)Q_1(r)}= \bigl\{\tfrac 1q=\tfrac{1}{dr}\bigr\}.
\end{gather*}
It is convenient to note, in view of \S\ref{freq-loc-section}, that the equation $\frac{1}{q}=\frac 1p+ \frac{2}{r(d-1)} -1$ is equivalent to $\frac 1r=\frac{d-1}{2}\big(\frac 1q+\frac 1{p'}\big)$.

Again, if $d=2$, the points $Q_2$ and $Q_3$ coincide, and the lines $\overline{Q_1(r)Q_2}$, $\overline{Q_3P(r)}$, $\overline{P(r)Q_4(r)}$ and $\overline{Q_4(r)(r)Q_1(r)}$ describe the quadrangle $\fQ_2(r)$ in Theorem \ref{d=2thm}, (ii).

(iii) In the case $1\le r<\tfrac{d}{d-1} $ we now have a quadrangle $\fQ_d(r)$ in Theorem \ref{thm:small r}, whose boundary lines are
\begin{gather*}
 \overline{Q_1(r)Q_2(r)} = \bigl\{\tfrac{1}{p}=\tfrac{1}{q}\bigr\}, \qquad \overline{Q_2(r)Q_3(r)} =\bigl\{\tfrac 1p=1-\tfrac{1}{r(d-1)}\bigr\},\\
 \overline{Q_3(r)Q_4(r)}= \bigl\{\tfrac 1q=\tfrac 1p+ \tfrac{2}{r(d-1)}-1\bigr\}, \qquad \overline{Q_4(r)Q_1(r)}= \bigl\{\tfrac 1q=\tfrac{1}{dr}\bigr\}.
\end{gather*}

We next list our necessary conditions for bounds on $V_r^I A$. We remark that the sharpness in the conditions \S\S\ref{subsec:ex1} -- \ref{knapp-simple} corresponds to the necessary conditions for the spherical maximal function $S^I$.

%\end{remarkno}
\subsection{The condition $p\le q$}\label{subsec:ex1} This is the standard necessary condition for translation operators mapping $L^p(\bbR^d)$ to $L^q(\bbR^d)$, see \cite{hormander1960}.

\subsection{The condition $p>\frac{d}{d-1}$} \label{steinexample} This is (a variant of) Stein's example for spherical maximal functions \cite{Stein1976}. Let $B$ be the ball of radius $1/10$ centered at the origin and let $f(y)=\bbone_B(y) |y|^{1-d} (\log|y|)^{-1}(\log\log |y|)^{-1}$. Then $f\in L^{\frac{d}{d-1},q}$ for all $q>1$, but for $1<|x|<2$ and $t(x)=|x|$ we have $A_{t(x)}f(x)=\infty.$

\subsection{The condition $d/q\ge 1/p$} 
For the condition $d/q\ge 1/p$ we just take the standard example for the spherical averages \cite{SS1997},  namely consider a fixed shell $S_{j,0}$  (as in \eqref{eqn:Sjn} below) and $g_j=\bbone_{ S_{j,0}}$ so that $\|g_j\|_{L^p} \le 2^{-j/p}$.  For $|x|\le 2^{-j-2}$ we have  $A_1g_j(x)\ge c>0$ and evaluating the $L^q$ norm over $\{x: |x|\le 2^{-j-2}\}$ we  get $\|V_r^IA g_j\|_{L^q}\ge 2^{-jd/q}$ and  obtain the necessity of  $d/q\ge 1/p$.

\subsection{The condition $\tfrac 1q\ge \tfrac{d+1}{(d-1)p}-1 $}\label{knapp-simple}
This is the standard Knapp example in \cite{SS1997}. Given $0 < \delta \ll 1$, one tests the maximal operator on $f_\delta$ being the  characteristic function of $\{y:|y'|\le \delta, \, |y_d|\le \delta^2\}$ and evaluates $A_{x_d} f_\delta(x)$ for $|x'|\le \delta$ and $1<x_d<2$.

\subsection{The condition $\frac 1p\le 1-\frac{1}{r(d-1)} $}\label{sec:disks}
In view of \S\ref{steinexample} this example is only relevant for $r<\frac{d}{d-1}$. 
For large $j$ define 
\Be\label{eqn:cjn}  c_{j,n}=-n 2^{-j} , \quad n=1,\dots, N\Ee where $N=2^{j-2}$.
Let  
$B_{j,n}$ be the ball of radius $2^{-j-4}$ centered at $c_{j,n}e_d$.

Let 
$f_j(x) = \sum_{n=1}^{N} (-1)^n  \bbone_{B_{j,n}}(x)$, so that 
\Be\notag\|f_j\|_{L^p} \lc N^{1/p} 2^{-jd/p}.\Ee

Consider \Be\label{eqn:R-region-def}  R=\{(x',x_d): |x'|\le (4d)^{-1} ,\,\, 1\le x_d\le 3/2\}.\Ee Note that
for $x\in R$ we have $|x-c_{j,n}e_d|\in [1,2]$; indeed, 
$|x-c_{j,n}e_d|\ge |x_d-c_{j,n}|\ge 1$ and 
$|x-c_{j,n}e_d|\le (|x_d-c_{j,n}|^2+ (4d)^{-2})^{1/2} \le 2$.

For $x\in R$ pick $t_n(x)=|x-c_{j,n} e_d|$ and observe that there is a constant $a>0$ such that 
$A_{t_{2\nu}(x) }f_j(x) \ge a 2^{-j(d-1)}$  and 
$A_{t_{2\nu-1}(x)}f_j(x) \le - a 2^{-j(d-1)}$ , and thus
\[
|A_{t_{2\nu}(x)}f_j(x) -A_{t_{2\nu-1}(x) }f_j(x)|\ge 2 a 2^{-j(d-1)}.\]

Hence, for any $r$ we get 
$V_r^I Af(x)\gtrsim N^{1/r} 2^{-j(d-1)}$ for $x\in R$  and thus for any $q>0$ 
\[\frac{ \|V_r^I A f_j\|_{L^q}}{\|f_j\|_{L^p}} \gc N^{\frac 1r-\frac 1p}2^{-j(d-1-\frac dp)}
\]  Since $N=2^{j-2}$ the assumption of $L^p\to L^q$ boundedness of $V^I_rA$ implies $\frac 1r\le \frac{d-1}{p'}$ or equivalently 
$\frac 1p\le 1- \frac{1}{r(d-1)}$.

\subsection{The condition $\tfrac 1q\ge \tfrac 1p+\tfrac{2}{(d-1)r}-1$, i.e.  $\frac{d-1}{2}(\frac 1q+\frac 1{p'}) \ge \frac 1r$}
  \label{sec:Knapp} 
  This is a variant of the example in \S\ref{knapp-simple}.
We let $c_{j,n}$ be as in \eqref{eqn:cjn}  and 
$P_{j,n}=\{y: |y'|\le 2^{-j/2-2}, |y_d-c_{j,n}|\le 2^{-j-4}\}$.
Let $N\le 2^{j-2}$.
Let $f_j=\sum_{n=1}^{N} (-1)^n  \bbone_{P_{j,n}}(x)$. Then $\|f_j\|_{L^p}\lc N^{1/p} 2^{-j\frac{d+1}{2p}}$. Let 
$\Om=\{x: |x'|\le 2^{-j/2-2}, \,\, 1\le x_d\le 3/2\}$ so that $|\Om|\approx 2^{-j(d-1)/2}$.
Let $t_n(x)=|x_d-c_{j,n}| \in [1,2]$. Then for $x\in \Om$,
$A_{t_{2\nu}(x)}f_j(x)\ge a 2^{-j(d-1)/2}$ and
$A_{t_{2\nu-1}(x)}f_j(x)\le - a 2^{-j(d-1)/2}$ for some constant $a>0$.
Hence  
%with $N=2^{j-2}$
$V_r^I A f_j(x)\gtrsim N^{1/r} 2^{-j\frac{d-1}{2}}$ and thus 
$\|V_r^I A f_j\|_{L^q}\gtrsim N^{1/r} 2^{-j\frac{d-1}{2} (1+\frac 1q)}$. Consequently with $N=2^{j-2}$
\[\frac{ \|V_r^I A f_j\|_{L^q}}{\|f_j\|_{L^p}} \gc N^{\frac 1r-\frac 1p} 
2^{-j\frac{d-1}{2} (1+\frac 1q)+j \frac{d+1}{2p} } \gc 2^{j (\frac 1r- \frac {d-1}{2}(\frac 1q+\frac 1{p'})) }.
\]  
Hence the condition $\frac{d-1}{2}(\frac 1q+\frac 1{p'}) \ge \frac 1r$ is necessary for $V_r^IA:L^p\to L^q$ to be bounded. Moreover, as $p \leq q$ by \S\ref{subsec:ex1}, this also implies that no $L^p(\R^2) \to L^q(\R^2)$ bounds hold for $r<2$. 

\subsection{The condition $d/q\ge 1/r$}  \label{sec:d/qge1/r} 
Consider the shells 
\Be\label{eqn:Sjn} S_{j,n} =\big\{y:  \big| |y|-1-n2^{j} \big| \le 2^{-j-2} \big\}  .\Ee
We set   $f_j=\sum_{n=1}^N (-1)^n \bbone_{S_{j,n}}$, with $N=2^{j-2}$. Then clearly $\|f_j\|_{L^p} \lc 1$ uniformly in $j$.

For $|x|\le 2^{-j-5}$  let $t_n(x)=1+ n2^{-j} \in [1,2]$. 
Then $A_{t_{2\nu} (x)}f_j(x) \ge  a$ and 
$A_{t_{2\nu-1} (x)}f_j(x) \le  -a$  for some $a$ independent of $j$. Hence 
$V_r^I f(x) \gtrsim N^{1/r} \approx 2^{j/r} $ for $|x|\le 2^{-j} $ and thus 
$\|V_r^I f_j\|_{L^q}\gc 2^{j(\frac 1r-\frac dq)}$. This implies the necessity of the condition $1/r\le d/q$.
\begin{remark} An alternative (more complicated) example for  the condition $d/q\ge 1/r$ is  in \cite[\S8]{JSW}.\end{remark}

\subsection{Sharpness of the sparse bounds} 
\label{sparse-sharp} 
The sparse domination result in Theorem~\ref{cor:sparse} is sharp, and this is immediate from the examples just described in this section. The argument, shown by Lacey in \cite[Section 5]{LaceySpherical} for the spherical maximal function, can be  extended in our context and even more general ones \cite[Proposition 7.2]{BRS-sparse}.

We exemplify this considering the example in \S\ref{sec:disks}, with the choice $N=2^{j-2}$. With $f_j$ as in this example we have $|f_j|=\bbone_U$ where $U$ is the union of the balls $B_{j,n}$ which is essentially a $2^{-j}$-neighborhood of the $x_d$-axis segment $[-1/4,0]$. $V_rAf_j$ is evaluated at $R$ as in \eqref{eqn:R-region-def}.
Then for large $j$ we have 
$$
\langle V_rA f_j,\bbone_R\rangle=\int_{\R^d} V_rA f(x)\bbone_R(x)\ud x\gc 2^{j(\frac 1r-d+1)}.
$$
On the other hand, suppose that $p<q$ and the sparse bound
$$
\int_{\R^d} V_rAf_j(x) \bbone_R(x) \ud x\le C_0 \sup_{\fS: \mathrm{sparse}} \Lambda^{\fS}_{p,q'} (f,\bbone_R) 
$$
holds for some positive $C_0$,  with $\La^\fS_{p,q'} (f,g)= \sum_{Q \in \mathfrak{S}}|Q| \langle f_j \rangle_{Q,p} \langle \bbone_R \rangle_{Q, q'}.$
By the definition of supremum there is a sparse collection $\fS_0$ such that 
$$
\int_{\R^d} V_rAf_j(x) \bbone_R(x) \ud x\le 2C_0 \sum_{Q \in \mathfrak{S}_0}|Q| \langle f_j \rangle_{Q,p} \langle \bbone_R \rangle_{Q, q'}.
$$
It is crucial in the example that \Be\label{eqn:sepsupp} \mathrm{dist}(\supp(f_j), R) \ge 1\Ee which implies that all cubes contributing to the sum have side length at least $1$.
Moreover, for each $l\ge 0$ there are only $O(1)$ cubes of sidelength $2^l$ contributing. For each such term we can estimate
\[ |Q| \langle f_j\rangle_{Q,p} \langle \bbone_R\rangle_{Q,q'} \lc |Q|^{\frac 1q-\frac 1p} 2^{-j\frac{d-1}{p} }
\]
and by summing over all terms (taking advantage of $p<q$) we obtain 
\[
2^{j(\frac 1r-d+1)} \lc
\langle V_rA f_j,\bbone_R\rangle=\int_{\R^d} V_rA f_j(x)\bbone_R(x)\ud x \lc  C_02^{-j\frac{d-1}{p} }
\]
and letting $j\to \infty$ we obtain the same necessary condition as in \S\ref{sec:disks}, i.e. $\frac{1}{p}\le 1-\frac{1}{r(d-1)}$.

The remaining examples in \S\S\ref{steinexample}--\ref{sec:d/qge1/r} yield similar necessary conditions for sparse bounds, and this is proved by essentially the same idea, always taking advantage of a support-separation property analogous to \eqref{eqn:sepsupp}. We leave the details to the reader.

\section{ $L^p\to L^q(L^r)$ estimates for $\mathcal{A}_j$}
\label{freq-loc-section}

In this section we prove $L^p \to L^q(L^r)$ bounds for the dyadic frequency localized operators $\mathcal{A}_j$ in the closure of the regions $\fP_d(r)$ and $\fQ_d(r)$ featuring in Theorems \ref{dge3thm}, \ref{thm:intermediate r}, \ref{thm:small r} and \ref{d=2thm}. This will lead to the proofs for $L^p \to L^q$ bounds for $V_r^I A$ if $(\tfrac{1}{p}, \frac{1}{q})$ belongs to the interior of $\fP_d(r)$ and $\fQ_d(r)$ respectively, as well as several restricted weak-type results through Bourgain's interpolation trick.

\subsection{Localization}\label{sec:localization}The following observation relies on the localization property \eqref{eqn:Kjtdef} of the kernel  $K_{j,t}$.

\begin{lemma} \label{localization-lemma} (i) 
 For $p_0\le p_1\le q_1 \leq q_0$, $1\le r\le\infty$, and every $N\in \bbN$,
\[\|\cA_j\|_{L^{p_1}\to L^{q_1}(L^{r})} \lc \|\cA_j\|_{L^{p_0}\to L^{q_0}(L^{r})} +C_N 2^{-jN}. \]

(ii) For $r_0\le r_1$, $1\le p\le q\le\infty$, 
%we have 
\[\|\cA_j\|_{L^p\to L^q(L^{r_0})} \lc \|\cA_j\|_{L^p\to L^q(L^{r_1})}  .\]
\end{lemma}

\begin{proof}
Assume that $\|\cA_j\|_{L^{p_0}\to L^{q_0}(L^{r})}<\infty$. Let $f\in L^{p_1}$. For $\fz\in \bbZ^d$ let $Q_\fz=\prod_{i=1}^d[\fz_i,\fz_i+1)$.
Let $Q_{\fz}^*$ be a cube centered at $\fz$ with side-length $20 d$.  Write
$f=\sum_\fz f_\fz$ with $f_\fz=f\bbone_{Q_{\fz}}$ and estimate 
\[\|\cA_j f\|_{L^{q_1}(L^r)}\le \Big\| \sum_\fz \bbone_{Q^*_\fz} \cA_j f_\fz\Big\|_{L^{q_1}(L^r)}+ \Big\|
\sum_\fz \bbone_{\R^d \backslash Q^*_\fz} \cA_j f_\fz\Big\|_{L^{q_1}(L^r)} \,=\,I+II.
\]
Since the $Q^*_\fz$ have bounded overlap, by Hölder's inequality for $q_1 \leq q_0$,
\[ I\lc \Big(\sum_\fz \|\bbone_{Q_\fz^*}\cA_j f_\fz\|_{L^{q_1}(L^r)}^{q_1} \Big)^{1/q_1} \lesssim \Big(\sum_\fz \|\cA_j f_\fz\|_{L^{q_0}(L^r)}^{q_1} \Big)^{1/q_1}.
\]
Applying the bound for the operator $\cA_j$,
\[
\Big(\sum_\fz \|\cA_j f_\fz\|_{L^{q_0}(L^r)}^{q_1} \Big)^{1/q_1}
\lc \|\cA_j\|_{L^{p_0}\to L^{q_0}(L^r)} 
\Big(\sum_\fz \| f_\fz\|_{L^{p_0}}^{q_1} \Big)^{1/q_1}
\] and, since $p_0\le p_1\le q_1$, we also have 
\[ \Big(\sum_\fz \| f_\fz\|_{L^{p_0}}^{q_1} \Big)^{1/q_1}
\lc
\Big(\sum_\fz \| f_\fz\|_{L^{p_1}}^{q_1} \Big)^{1/q_1}
\lc\Big(\sum_\fz \| f_\fz\|_{L^{p_1}}^{p_1} \Big)^{1/p_1}\lc \|f\|_{L^{p_1}}.
\]
Moreover, by \eqref{kernel-error} with $N>d$, 
\[ 
II\le \Big(\int \Big[ \int_{|y-x|\ge 1}  (2^j|x-y|)^{-N} |f(y)| \ud y\Big]^{q_1}
\ud x\Big)^{1/q_1} \lc_N 2^{-jN} \|f\|_{L^{p_1}}.
\]
Combining the two estimates we obtain
\[\|\cA_j f\|_{L^{q_1}(L^r)} \lc \big(\|\cA_j\|_{L^{p_0}\to L^{q_0}(L^r)} +C_N 2^{-jN}\big) \|f\|_{L^{p_1}},
\] which is the assertion in  part (i).

Part (ii) is immediate and simply follows from Hölder's inequality in the $t$-variable.
\end{proof}

\subsection{Interpolation}

Lemma \ref{stein-squarefct} can be extended to a larger range of exponents by interpolation with \eqref{maxinfty} and Lemma \ref{V1lemma} and by the localization property in Lemma \ref{localization-lemma}. We state this in more generality; see Figure \ref{fig:interpolation}.

%%%%%%%%%%%%%%%%%%%%%%%%%%%%
%%
%   single scale interpolation
%%
%%%%%%%%%%%%%%%%%%%%%%%%%%%%%

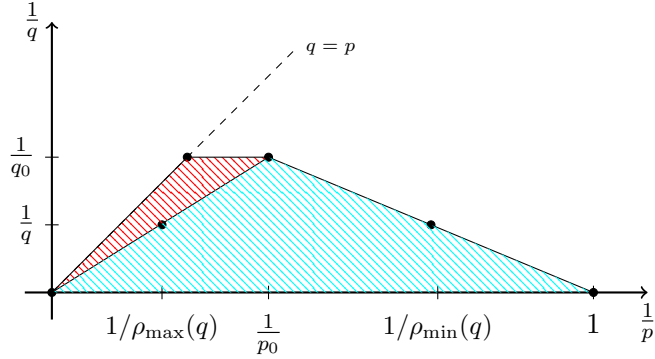
\begin{figure}[t]
\begin{tikzpicture}[scale=1.8] 

\begin{scope}[scale=2]
\draw[thick,->] (-.1,0) -- (2.2,0) node[below] {$ \frac 1 p$};
\draw[thick,->] (0,-.1) -- (0,1) node[left] {$ \frac 1 q$};
\draw[ dashed] (0,0) -- (0.9,0.9) node[right] {\tiny{$q=p$}};

\node[circle,draw=black, fill=black, inner sep=0pt,minimum size=3pt] (b) at (0,0) {};

\node[circle,draw=black, fill=black, inner sep=0pt,minimum size=3pt] (b) at (2,0) {};

\draw (2,-0.025) -- (2,.025) node[below= 0.25cm] {$1$}; 

\draw (4/5,-0.025) -- (4/5,.025) node[below= 0.25cm] {$\frac{1}{p_0}$}; 
\draw (.025,1/2) -- (-.025,1/2) node[left] {$ \frac{1}{q_0}$};
\node[circle,draw=black, fill=black, inner sep=0pt,minimum size=3pt] (b) at (4/5,1/2) {};

\draw  (2,0)  -- (4/5,1/2); 
\draw (0,0) -- (4/5,1/2);
%\draw[ dashed] (4/5,0) -- (4/5,1/2);

\node[circle,draw=black, fill=black, inner sep=0pt,minimum size=3pt] (b) at (1/2,1/2) {};

\node[circle,draw=black, fill=black, inner sep=0pt,minimum size=3pt] (b) at (5.6/4,1/4) {};

\draw (5.7/4,-0.025) -- (5.7/4,.025) node[below= 0.25cm] {\small{1/$\rho_{\mathrm{min}}(q)$}};

\node[circle,draw=black, fill=black, inner sep=0pt,minimum size=3pt] (b) at (1.22/3,1/4) {};

\draw (1.22/3,-0.025) -- (1.22/3,.025) node[below= 0.25cm]
{\small{1/$\rho_{\mathrm{max}}(q)$}};

\draw (.025,1/4) -- (-.025,1/4) node[left] {$ \frac{1}{q}$};

\fill[pattern=north west lines, pattern color=cyan] (0,0) -- (2,0) -- (4/5,1/2) --  (0,0);

\fill[pattern=north west lines, pattern color=red] (0,0) -- (4/5,1/2) -- (1/2,1/2) -- (0,0);

\draw (1/2,1/2) -- (4/5,1/2);
\draw (1/2,1/2) -- (0,0);

\end{scope}

\end{tikzpicture}

\caption{Interpolation and localization lemmas. If $\| \cA_j\|_{L^{p_0} \to L^{q_0}(L^{p_0})} \lesssim 2^{-jd/q_0}$, then $\| \cA_j\|_{L^{p} \to L^{q}(L^{p})} \lesssim 2^{-jd/q}$ in the blue triangle and $\| \cA_j\|_{L^{p} \to L^{q}(L^{\rho_{\mathrm{max}}(q) })} \lesssim 2^{-jd/q}$ in the red triangle.} 
\label{fig:interpolation}

\end{figure}

\begin{lemma} \label{p0q0lemma}
Let $p_0$ and $q_0$ such that $1\le p_0\le q_0 \leq \infty$.
Assume that 
\Be\label{p0q0-assu}
\sup_{j \geq 0} 2^{j d/q_0} \|\cA_j\|_{L^{p_0}\to L^{q_0}(L^{p_0})} \le C<\infty.\Ee
Let $q_0\le q\le \infty$ and define $\rho_{\min}(q)$ and $\rho_{\max}(q)$ by
\Be\label{rhominmax}
1-\frac{1}{\rho_{\min}(q)} =\frac{q_0}q\Big(1-\frac 1{p_0}\Big), \qquad
\frac 1{\rho_{\max}(q)}= \frac {q_0}{q}\frac 1{p_0}.
\Ee
Assume that $\rho_{\min}(q)\le p\le q$ and $0<r\le \min \{ p, \rho_{\max}(q)\} $.
Then
\[\sup_{j \geq 0} 2^{j d/q} \|\cA_j\|_{L^{p}\to L^{q}(L^{r})} <\infty.\]
\end{lemma}

\begin{proof} 
Note that $\rho_{\min}(q)\le \rho_{\max}(q)$ when $q\ge q_0$, with strict inequality when $q>q_0$, and 
$\rho_{\min}(q_0)= \rho_{\max}(q_0)=p_0$.  Assume $q>q_0$ and let $\vartheta=1-q_0/q$. Note that $(1-\vth)/p_0=1/\rho_{\max}(q)$ and
$(1-\vth)/p_0+\vth= 1/\rho_{\min}(q)$.
We interpolate \eqref{p0q0-assu} %\eqref{p0q0-assu-upgraded} 
with the inequality
\begin{equation*}
\sup_{j \geq 0} \|\cA_j\|_{L^{p_1}\to L^\infty(L^{p_1})} <\infty, \quad 1\le p_1\le \infty
\end{equation*}
for the choices  $p_1=1$ and $p_1=\infty$ 
(by Lemma \ref{V1lemma} and \eqref{maxinfty}) 
and obtain the $L^p\to L^q(L^p)$ inequality for $p=\rho_{\min}(q) $ and  $p=\rho_{\max}(q)$. A further interpolation gives
\[\sup_{j \geq 0} \|\cA_j\|_{L^p\to L^q(L^p)} \lc \big(1+ \sup_{j \geq 0} \|\cA_j\|_{L^{p_0}\to L^{q_0}(L^{p_0})} \big), \quad \rho_{\min}(q)\le p\le \rho_{\max}(q).
\] We now combine this with Lemma \ref{localization-lemma}  and see that
the $L^p\to L^q(L^r)$ estimates hold when
$\rho_{\min}(q)\le p\le \rho_{\max}(q)$ and $r\le p$ and moreover when
$\rho_{\min}(q)\le r\le \rho_{\max(q)}$ and $r\le p\le q$. 
\end{proof}

\subsection{Bounds for $\cA_j$}
The previous lemma and the estimates in \S \ref{sec:preliminaries} yield the following bounds; see Figure \ref{fig:single scale} for the regions.

%%%%%%%%%%%%%%%%%%%%%%%%%%%%
%%
%   single scale bounds d\geq 3
%%
%%%%%%%%%%%%%%%%%%%%%%%%%%%%%

\begin{figure}[t]
\begin{tikzpicture}[scale=1.5] 

\begin{scope}[scale=2]
\draw[thick,->] (-.1,0) -- (2.2,0) node[below] {$ \frac 1 p$};
\draw[thick,->] (0,-.1) -- (0,2.2) node[left] {$ \frac 1 q$};

\draw (2,-0.025) -- (2,.025) node[below= 0.25cm] {$ 1$}; 
\draw (-0.025,2) -- (.025,2) node[left= 0.25cm] {$ 1$}; 
\node[circle,draw=black, fill=black, inner sep=0pt,minimum size=3pt] (b) at (2,2) {};
\node[circle,draw=black, fill=black, inner sep=0pt,minimum size=3pt] (b) at (0,0) {};

\draw (1,-0.025) -- (1,.025) node[below= 0.25cm] {$\frac 1 2$}; 
\draw (.025,1) -- (-.025,1) node[left] {$ \frac{1}{2}$};
\node[circle,draw=black, fill=black, inner sep=0pt,minimum size=3pt] (b) at (1,1) {};

\draw[ dashed]  (2,0)  -- (1,1); 
\draw[ dashed] (0,0) -- (1,1);
%\draw[ dashed] (1,0) -- (1,1);
\draw[dashed] (2,0) -- (2,2);
\draw[dashed] (2,2) -- (6.75/4, 6.75/4);

%%%%%% Strichartz

\node[circle,draw=blue, fill=blue, inner sep=0pt,minimum size=3pt] (b) at (1,1/2) {};
\node[circle,draw=blue, fill=blue, inner sep=0pt,minimum size=3pt] (b) at (1/2,1/2) {};
\draw (-0.025,1/2) -- (.025,1/2);
\draw (0,1/2) node[left = 0.08cm] {\textcolor{blue}{$ \frac{d-1}{2(d+1)}$} };

\fill[pattern=north west lines, pattern color=cyan] (1/2, 1/2) -- (1,1/2) -- (1,1) --  (1/2,1/2);

\fill[pattern=north west lines, pattern color=magenta] (1, 1) -- (1,1/2) -- (2,0) --  (1,1);

\draw (1.5,1) node[above] {$A$};
\draw (0.87,0.6) node[above] {$B_1$};
\draw (0.75,0.44) node[above] {$B_2$};
\draw (1.2,0.5) node[above] {$C$};
\draw (0.52,0.3) node[above] {$D$};
\draw (1,0.2) node[above] {$E$};

\draw[dashed] (1,1/2) -- (0,0);

%%%% conjectured point

\node[circle,draw=red, fill=red, inner sep=0pt,minimum size=3pt] (b) at (2/3,2/3) {};

\draw (-0.025,2/3) -- (.025,2/3);
\draw (0,2/3) node[left = 0.08cm] {\textcolor{red}{$ \frac{d-1}{2d}$} };

\draw[ dashed, color=red] (2,0) -- (2/3,2/3) ;

\fill[pattern=north west lines, pattern color=blue] (1, 1/2) -- (1/2,1/2) -- (0,0) --  (1,1/2);

\fill[pattern=north west lines, pattern color=olive] (0,0) -- (1, 1/2) -- (2,0) -- (0,0);

%%%%% endpoint spherical

%\draw (-0.025,6.75/4) -- (.025,6.75/4);
%\draw (0,6.75/4) node[left = 0.08cm] {\textcolor{olive}{$ \frac{d-1}{d}$} };
\draw[dashed] (1,1) -- (6.75/4,6.75/4);
%\node[circle,draw=olive, fill=olive, inner sep=0pt,minimum size=3pt] (b) at (6.75/4,6.75/4) {};

\draw[dashed, color=gray] (1,1) -- (1,1/2);
\draw[dashed, color=gray] (1/2,.99/2) -- (1,1/2);

%\draw[dashed]  (6.75/4,1.25/4) -- (1 + 1/2 + 1/10, 1/8+1/16 + 1/80);
\fill[pattern=north west lines, pattern color=orange] (1,1) -- (2,0) -- (2,2) -- (1,1);

\end{scope}

\end{tikzpicture}

\caption{Regions for $L^p \to L^q(L^r)$ bounds for the single scale $\cA_j$ for $0 < r \leq 1$. As $r$ increases the regions shrink due to the constraints $r \leq p$ or $r \leq \tfrac{q(d-1)}{d+1}$.}
\label{fig:single scale}
%\vspace{1cm}
\end{figure}
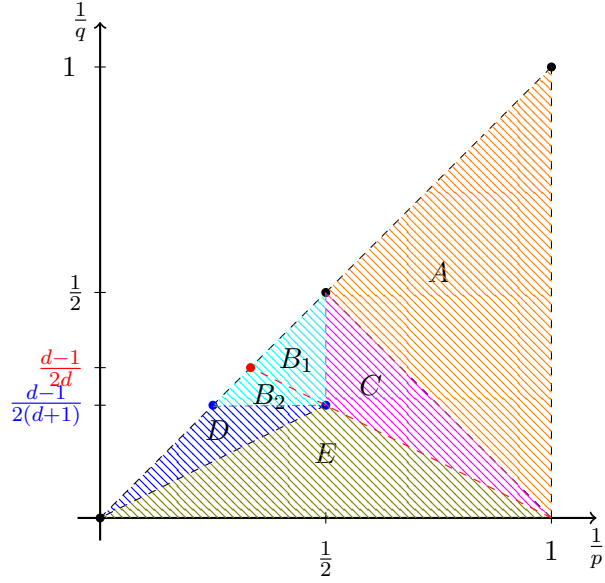

\begin{proposition}\label{prop:bounds A_j with ST}
Let $d \geq 2$.
\begin{enumerate}[(A)]
    \item Let $ 1 \leq p \leq 2$, $p \leq q \leq p'$ and $0 < r \leq p$. Then
$$
\|\mathcal{A}_j f\|_{L^q(L^r)} \lesssim 2^{-j(d-1)/p'} \| f \|_{L^p}.
$$
\item Let $2 \leq p \leq q \leq \frac{2(d+1)}{d-1}$. Let $0 < r \leq 2$. Then
$$
\| \mathcal{A}_j f \|_{L^q(L^r)} \lesssim 2^{-j\frac{d-1}{2} (\frac{1}{q} + \frac{1}{2})} \| f \|_{L^p}.
$$
\item Let $1 \leq p \leq 2$, $\frac{d-1}{d+1} \frac{1}{p'} \leq \frac{1}{q} \leq \frac{1}{p'}$ and $0 <r \leq p.$ Then
$$
\|\mathcal{A}_j f \|_{L^q(L^r)} \lesssim 2^{-j\frac{d-1}{2}(\frac{1}{q} + \frac{1}{p'})} \| f\|_{L^p}.
$$
\item Let $\tfrac{2(d+1)}{d-1}\le q\le \infty$, $ \frac{d-1}{d+1} \frac{1}{p} \leq \tfrac 1q\le \tfrac 1p$ and $0 <r \leq \frac{q(d-1)}{d+1}$.  % $\tfrac 1r\ge \tfrac{d+1}{d-1}\tfrac 1q.$
Then 
\[ \|\cA_j f\|_{L^q(L^r)} \lc 2^{-jd/q} \|f\|_{L^p}.\]

\item Let $\tfrac{2(d+1)}{d-1}\le q\le \infty$, $\frac{1}{q} \leq \frac{d-1}{d+1} \frac{1}{p}$, $\tfrac 1q\le \tfrac 1p\le  1- \tfrac{d+1}{d-1}\tfrac 1q,$ and $0 < r \leq p$. Then 
\[ \|\cA_j f\|_{L^q(L^r)} \lc 2^{-jd/q} \|f\|_{L^p}.\]
\end{enumerate}

\end{proposition}

\begin{proof}
The bounds in (A) for $r=p$ follow from interpolation of Lemma \ref{V1lemma} and the $L^2$-estimate \eqref{eq:L2}, whilst the remaining values of $0<r<p$ follow from (ii) in Lemma \ref{localization-lemma}.

The bounds in (D) and (E) are an application of Lemma \ref{p0q0lemma} with $p_0=2$, $q_0=\tfrac{2(d+1)}{d-1}$, which is the estimate in Lemma \ref{stein-squarefct}.

The bounds in (C) follow from interpolation of those in (A) if $q=p'$ and those in (E) if $\tfrac{1}{q}=\tfrac{1}{p'} \tfrac{d-1}{d+1}$, $1 \leq p \leq 2$.

Finally, the bounds in (B) follow from interpolation of the $L^2$ estimate \eqref{eq:L2} with the $L^p \to L^p(L^2)$ estimate in (D) for $p=\tfrac{2(d+1)}{d-1}$, and a further interpolation of those with the estimates in (C) for $p=2$.
\end{proof}

The above bounds on (A), (C) and (E) are sharp. However, the bounds in (B) and the $r$-range in (D) can be improved; for example, if information on the \textit{local smoothing} phenomenon for the wave equation is known. Recall that these estimates, first noted by Sogge in \cite{Sogge91}, are of the type
\begin{equation}\label{local smoothing conj}
\Big\| \Big( \int_1^2  |e^{i t \sqrt{-\Delta}} L_j f|^p \ud t \Big)^{1/p} \Big\|_{L^p} \lesssim 2^{j(\bar s_p - \sigma)} \| f \|_{L^p}
\end{equation}
for some $\sigma>0$ if $2 < p < \infty$, where $\bar{s}_p:=  (d-1) \big(\frac{1}{2}- \frac{1}{p}\big)$. It is conjectured that \eqref{local smoothing conj} holds for all $\sigma < \sigma_p$, where
\begin{equation*}
    \sigma_p:= \begin{cases}
    \begin{array}{ll}
        1/p & \text{ if } \quad  \frac{2d}{d-1} \leq  p < \infty,  \\
        \bar s_p & \text{ if } \quad   \,\, 2 \leq p \leq \frac{2d}{d-1}.
    \end{array}
    \end{cases}
\end{equation*}
This conjecture is strongest at $p=\tfrac{2d}{d-1}$. After contributions by many, it has recently been solved by Guth, Wang and Zhang \cite{GWZ} for $d=2$, and is known to hold for all $p \geq \tfrac{2(d+1)}{d-1}$ if $d \geq 3$ by the sharp decoupling inequalities of Bourgain and Demeter \cite{BD2015}. It is also expected that endpoint regularity results with $\sigma=1/p$ should hold if $p>2d/(d-1)$; see \cite{HNS2011} for results in this direction if $d \geq 4$. The validity of the local smoothing conjecture would imply the following bounds on spherical averages on the region (B). We remark that these improved bounds are only relevant for our variational bounds if $d=2,3$; for $d\geq 4$ the bounds in Proposition \ref{prop:bounds A_j with ST} will suffice (see the discussion after Theorem \ref{thm:small r} in the Introduction).

\begin{proposition}\label{prop:bounds Aj d=2 B}
Let $d\geq 2$. Assume that the local smoothing conjecture holds, that is, \eqref{local smoothing conj} holds at $p=\frac{2d}{d-1}$ for all $\sigma < 1/p.$
\begin{enumerate}
    \item[($B_1$)] If $ \tfrac{d-1}{d+1} \tfrac{1}{p'} \leq \tfrac{1}{q} \leq \tfrac{1}{p}$ and $2 < q \leq \tfrac{2d}{d-1}$, $2 < p \leq \tfrac{2d}{d-1}$ and $0 < r \leq p $, then
$$
\| \cA_j f \|_{L^q(L^r)} \lesssim 2^{-j \frac{d-1}{2} (\frac{1}{q} + \frac{1}{p'}) + j \varepsilon}  \| f \|_{L^p}
$$
for all $\varepsilon>0$.
\item[($B_2$)] If $\tfrac{1}{q} \leq \min \{ \tfrac{d-1}{d+1} \tfrac{1}{p'}, \tfrac{1}{p}\}$ and $\tfrac{2d}{d-1} \leq q \leq \tfrac{2(d+1)}{d-1}$ and $0 < r \leq p $, then
$$
\| \cA_j f \|_{L^q(L^r)} \lesssim 2^{-jd/q + j \varepsilon} \| f \|_{L^p}
$$
for all $\varepsilon>0$.
\end{enumerate}
In particular, the above estimates hold for $d=2$.
\end{proposition}

\begin{proof}
By the oscillatory integral representation in \eqref{Tjs-split} and \eqref{Tjs}, the estimate \eqref{local smoothing conj} implies
\begin{equation}\label{eq:LS conj av}
    \| \cA_j f \|_{L^{p}(L^p)} \lesssim 2^{-j\frac{d-1}{2} + j\varepsilon} \| f\|_{L^p}
\end{equation}
for $p=\frac{2d}{d-1}$. Interpolation of \eqref{eq:LS conj av} and Lemma \ref{stein-squarefct} yields
\begin{equation}\label{eq:critical line d=2}
    \| \mathcal{A}_j f \|_{L^q(L^p)} \lesssim 2^{-jd/q + j \varepsilon } \| f \|_{L^p}
\end{equation}
for $\tfrac{1}{q}=\tfrac{d-1}{d+1}\tfrac{1}{p'}$ and $2 < p \leq \tfrac{2d}{d-1} \leq q < \tfrac{2(d+1)}{d-1}$. Moreover, interpolation of \eqref{eq:LS conj av} and the $L^2$-estimate \eqref{eq:L2} yields
\begin{equation}\label{eq:2<p<4 d=2}
    \| \cA_j \|_{L^p(L^p)} \lesssim 2^{-j\frac{d-1}{2} + j \varepsilon} \| f \|_{L^p}
\end{equation}
for $2 < p \leq \tfrac{2d}{d-1}$. The region (B$_1$) then follows from interpolating \eqref{eq:critical line d=2} and \eqref{eq:2<p<4 d=2}. 

For the region (B$_2$), interpolate \eqref{eq:LS conj av} and \eqref{maxinfty} to obtain
\begin{equation}\label{eq:Ls with Linfty}
    \| \cA_j f \|_{L^q(L^q)} \lesssim 2^{-jd/q + j \varepsilon} \| f \|_{L^q}
\end{equation}
for all $\frac{2d}{d-1} \leq q \leq \infty$.
A further interpolation of \eqref{eq:Ls with Linfty} with \eqref{eq:critical line d=2} for $\tfrac{2d}{d-1} \leq q \leq \tfrac{2(d+1)}{d-1}$ yields the estimates in (B$_2$).

The assertion for $d=2$ follows since the local smoothing assumption was established in \cite{GWZ}.
\end{proof}

The range of $r$ in the estimates in (D) can also be improved to $0 < r \leq p$ using the known local smoothing estimates at $p=\tfrac{2(d+1)}{d-1}$ for all $\sigma < 1/p$. For our variational problem, this only becomes relevant if $d=2$, as otherwise the results in Proposition \ref{prop:bounds A_j with ST} will suffice. We note that the use of such local smoothing estimates induces an $\varepsilon$-loss with respect to (D) in Proposition \ref{prop:bounds A_j with ST}, although this will have no consequences on our proof in $d=2$. The $\varepsilon$-loss in the forthcoming proposition can be removed if $p>\tfrac{2(d-1)}{d-3}$ when $d \geq 4$ by the currently known sharp regularity estimates in \cite{HNS2011}. 

\begin{proposition}[Improved bounds in (D)]\label{prop:bounds Aj d=2 D}
Let $d \geq 2$. Let $ \tfrac{2(d+1)}{d-1} \le q\le \infty$, $ \tfrac{d-1}{d+1}\frac{1}{p} \leq \tfrac 1q\le \tfrac 1p$ and $r \leq p$.  Then 
\[ \|\cA_j f\|_{L^q(L^r)} \lc 2^{-j(d/q -  \varepsilon)}  \|f\|_{L^p}\]
for all $\varepsilon>0$.
\end{proposition}

\begin{proof}
By \eqref{Tjs-split}, the estimates \eqref{local smoothing conj} for $p \geq \tfrac{2(d+1)}{d-1}$ imply that, given any $\varepsilon>0$,
$$
\| \cA_j f\|_{L^p(L^p)} \lesssim 2^{-j(d/q - \varepsilon)} \| f \|_{L^q}
$$
holds for all $\tfrac{2(d+1)}{d-1} \leq q \leq \infty$. It then suffices to interpolate this with the estimates in Proposition \ref{prop:bounds A_j with ST}, (D), when $q=\tfrac{p(d+1)}{d-1}$ and $r=p$.
\end{proof}

\subsection{Bounds for $V_r^I\cA_j$}
Let $1 \leq r \leq \infty$. By the embedding \eqref{BPe}  and Corollary \ref{cor:trading-der},
$V_r A$ maps $L^p(\R^d)$ to $L^q(\R^d)$ if there exists an $\eps>0$ such that
%\db{and derivatives in $t$ are like derivatives in $x$}, if there exists $\varepsilon>0$ such that
\Be\label{eqn:eps-decay}
\| \mathcal{A}_j f \|_{L^q(L^r)} \lesssim 2^{-j(\frac 1r+\varepsilon)} \| f \|_{L^p},
\Ee
for all $f\in L^p$.
This will suffice to show all the bounds in the interiors of $\fP_d(r), \fQ_d(r)$ claimed in Theorems \ref{dge3thm}, \ref{thm:intermediate r}, \ref{thm:small r} and \ref{d=2thm}.

We start with the case $d\geq 3$. We will only have to identify in each region $A-E$ of Proposition \ref{prop:bounds A_j with ST} the conditions under which \eqref{eqn:eps-decay} holds and to relate this to the corresponding statements in the theorems in the introduction.

\begin{proposition}\label{prop:bounds with 1/r}
Let $d \geq 3$. The inequality \eqref{eqn:eps-decay} holds for some $\varepsilon>0$
 under the following conditions on $1 \leq p, q \leq \infty$, $0 < r \leq \infty$: 
\begin{enumerate}
\item[(A')] $ 1 \leq p \leq 2$, $p \leq q \leq p'$, and 
\begin{itemize}
\item[$\circ$]  $ \tfrac{d}{d-1} < r \leq p$; or 
\item[$\circ$] $\tfrac{2}{d-1} < r \leq \tfrac{d}{d-1}$ and $\tfrac{1}{p} < 1-\tfrac{1}{(d-1)r}$. 
\end{itemize}
\item[(B')]  $2 \leq p \leq q \leq \frac{2(d+1)}{d-1}$ and
\begin{itemize}
    \item[$\circ$] $\tfrac{2(d+1)}{d(d-1)} < r \leq 2$; or
    \item[$\circ$] $\tfrac{2}{d-1} < r \leq \tfrac{2(d+1)}{d(d-1)}$ and $\tfrac{1}{q} > \tfrac{2}{(d-1)r}-\tfrac{1}{2}$.
\end{itemize}
\item[(C')] $1 \leq p \leq 2$, $\frac{d-1}{d+1} \frac{1}{p'} \leq \frac{1}{q} \leq \frac{1}{p'}$, and
\begin{itemize}
\item[$\circ$] $\tfrac{d^2+1}{d(d-1)} <r \leq p$ and $\tfrac{1}{q} > \tfrac{d+1}{d-1}\tfrac{1}{p}-1$, $\tfrac{1}{p} < \tfrac{d-1}{d}$; or 
\item[$\circ$] $\tfrac{d}{d-1} < r \leq \min\{\tfrac{d^2+1}{d(d-1)}, p\}$ and $\tfrac{1}{q} > \tfrac{d+1}{d-1}\tfrac{1}{p}-1$, $\frac{1}{q}>\frac 1p+ \frac{2}{r(d-1)} -1$; or 
\item[$\circ$] $\tfrac{2}{d-1} < r \leq \tfrac{d}{d-1}$ and $\frac{1}{q}>\frac 1p+ \frac{2}{r(d-1)} -1$. 
\end{itemize}

\item[(D')] $\tfrac{2(d+1)}{d-1}\le q\le \infty$, $ \frac{d-1}{d+1} \frac{1}{p} \leq \tfrac 1q\le \tfrac 1p$ and $\tfrac{1}{q} > \tfrac{1}{dr}$ for $  \tfrac{2(d+1)}{d(d-1)}< r \le \tfrac{q(d-1)}{d+1}$.

\item[(E')] $\tfrac{2(d+1)}{d-1}\le q\le \infty$, $\frac{1}{q} \leq \frac{d-1}{d+1} \frac{1}{p}$, $\tfrac 1q\le \tfrac 1p\le  1- \tfrac{d+1}{d-1}\tfrac 1q$ and $\frac{1}{q}>\frac{1}{dr}$ for $ \frac{2(d+1)}{d(d-1)} < r \leq p$.
\end{enumerate}
\end{proposition}

\begin{proof}
It suffices to check that the exponents appearing in the inequalities $A-E$ in Proposition \ref{prop:bounds A_j with ST} are strictly greater than $1/r$ under the claimed conditions.
\begin{enumerate}[(A')]
\item The exponent in (A), Proposition \ref{prop:bounds A_j with ST}, is $\tfrac{d-1}{p'}$. Note that $\tfrac{d-1}{p'} > \tfrac{1}{r}$ is satisfied if $\tfrac{d}{d-1} < r \leq p$. Moreover, it also holds if 
$\frac{1}{p} < \frac{(d-1)r-1}{(d-1)r}$ and $ r \leq \tfrac{d}{d-1}$. The additional constraint $r>\tfrac{2}{d-1}$ follows since $p \geq 2$ in (A). Note this requires $d \geq 3$.

    \item The exponent in (B), Proposition \ref{prop:bounds A_j with ST}, is $\tfrac{d-1}{2}(\tfrac{1}{q}+\tfrac{1}{2})$. Note that $\tfrac{d-1}{2}(\tfrac{1}{q}+\tfrac{1}{2}) > \tfrac{1}{r}$ is satisfied if $\tfrac{2(d+1)}{d(d-1)} < r \leq 2$, as $q \leq \tfrac{2(d+1)}{d-1}$. Moreover, it also holds if $\frac{1}{q}>\frac{2}{(d-1)r} - \frac{1}{2}$ and $r \leq \tfrac{2(d+1)}{d(d-1)}$. The additional constraint $r>\tfrac{2}{d-1}$ follows since $q \geq 2$ in (B). Note this requires $d \geq 3$.

    \item The exponent in (C), Proposition \ref{prop:bounds A_j with ST}, is $\tfrac{d-1}{2}(\tfrac{1}{q}+\tfrac{1}{p'})$. Note that $\tfrac{d-1}{2}(\tfrac{1}{q}+\tfrac{1}{p'}) > \tfrac{1}{r}$ is satisfied if $\tfrac{d^2+1}{d(d-1)} < r \leq p$, as $\tfrac{1}{q} \geq \tfrac{d-1}{d+1} \tfrac{1}{p'}$. The additional constraint $\tfrac{1}{q}>\tfrac{d+1}{d-1} \tfrac{1}{p} - 1$ follows from $r \leq p$. Note that this and $q \geq p'$, also yield the additional constraint $\tfrac{1}{p} < \tfrac{d-1}{d}$. 
    
    For the remaining values $r \leq \tfrac{d
   ^2+1}{d(d-1)}$, it simply holds by the assumption $\frac{1}{q}>\frac 1p+ \frac{2}{r(d-1)} -1$. Note that $r\leq p$ is automatically satisfied if $r \leq \tfrac{d}{d-1}$. The lower bound $r>\tfrac{2}{d-1}$ follows from the assumption $\frac{1}{q}>\frac 1p+ \frac{2}{r(d-1)} -1$ with $q \geq p'$ and $p \leq 2$. This yields $\tfrac{2}{d-1} <  r \leq p \leq 2 $, which requires $d \geq 3$.

    \item The exponent in (D), Proposition \ref{prop:bounds A_j with ST}, is $\tfrac{d}{q}$. Note that $\tfrac{d}{q} > \tfrac{1}{r}$ is trivially satisfied if $\tfrac{1}{q} > \tfrac{1}{dr}$. The lower bound  $ r > \tfrac{2(d+1)}{d(d-1)} $, follows from $q \leq \tfrac{2(d+1)}{d-1}$. Note that when combined with $r \leq \tfrac{q(d-1)}{d+1}$ requires $d \geq 3$.

    \item The exponent in (E), Proposition \ref{prop:bounds A_j with ST}, is $\tfrac{d}{q}$. Note that $\tfrac{d}{q} > \tfrac{1}{r}$ is trivially satisfied if $\tfrac{1}{q} > \tfrac{1}{dr}$. %Moreover, as $r \leq p$, one has the additional condition $\tfrac{1}{q} > \tfrac{1}{dp}$. 
    The constraint $r > \tfrac{2(d+1)}{d(d-1)}$ follows from $q \geq \tfrac{2(d+1)}{d-1}$. Note the above constraints combined yield the additional condition $\tfrac{1}{d} \leq \tfrac{r}{q} \leq \tfrac{d-1}{d+1}$, which requires $d \geq 3$. \qedhere

\end{enumerate}
\end{proof}

We next turn to the case $d=2$. As observed in the proof of the previous proposition, the bounds in Proposition \ref{prop:bounds A_j with ST} do not yield any bound of the type \eqref{eqn:eps-decay} for $d=2$. We use instead the upgraded bounds from Propositions \ref{prop:bounds Aj d=2 B} and \ref{prop:bounds Aj d=2 D}.

\begin{proposition}\label{prop:bounds 1/r d=2}
Let $d=2$. The inequality \eqref{eqn:eps-decay} holds for some $\varepsilon>0$
 under the following conditions on $1 \leq p, q \leq \infty$, $0 < r \leq \infty$:
 \begin{enumerate}
     \item[(B$_1$')] $ \tfrac{1}{3p'}  \leq \tfrac{1}{q} \leq \tfrac{1}{p}$ and $2 < q \leq 4$, $2<p\le 4$, and
     \begin{itemize}
      \item[$\circ$] $5/2 < r \leq p $; or 
       \item[$\circ$] $2 < r \leq  \min\{ 5/2 , p\}$ and 
       $\frac{1}{q}>\frac 1p+ \frac{2}{r} -1$.
     \end{itemize}

\item[(B$_2$')]  $\tfrac{1}{q} \leq \min \{ \tfrac{1}{3p'} , \tfrac{1}{p}\}$, $4 \leq q \leq 6$ and $\frac{1}{q}>\frac{1}{2r}$ for $2 < r \leq p$.

\item[(D')]  $ 6 \le q\le \infty$, $ \tfrac{1}{3p}  \leq \tfrac 1q\le \tfrac 1p$ and  $\tfrac{1}{q} > \tfrac{1}{2r}$  for $3 < r \leq p$. 
\end{enumerate}
\end{proposition}

\begin{proof}
As in Proposition \ref{prop:bounds with 1/r}, it suffices to check that the exponents appearing in the inequalities $B_1$, $B_2$ in Proposition \ref{prop:bounds Aj d=2 B} and in Proposition \ref{prop:bounds Aj d=2 D} are strictly greater than $1/r$ under the claimed conditions.
\begin{enumerate}
    \item[(B$_1$')] The exponent in (B$_1$'), Proposition \ref{prop:bounds Aj d=2 B} is $\frac{1}{2} (\frac{1}{q} + \frac{1}{p'}) - \varepsilon$. Choosing $\varepsilon>0$ small enough, $\frac{1}{2} (\frac{1}{q} + \frac{1}{p'}) - \varepsilon > \tfrac{1}{r}$ is satisfied using $q \leq 3p'$ and $5/2 < r \leq p$. 
    If $r \leq \min\{5/2, p\}$, the required condition follows simply by assumption choosing $\varepsilon>0$ to be small enough. Note that $r>2$ follows from the assumptions $\tfrac{1}{2} (\frac{1}{q} + \frac{1}{p'}) > \tfrac{1}{r}$ and $p \leq q$.
    
    \item[(B$_2$')] The exponent in (B$_2$'), Proposition \ref{prop:bounds Aj d=2 B} is $2/q-\varepsilon$. Choosing $\varepsilon>0$ small enough, $2/q-\varepsilon>\frac{1}{r}$ is trivially satisfied by the assumption $\tfrac{1}{q}>\tfrac{1}{2r}$. 
    The lower bound $r>2$ follows from the assumptions $\tfrac{1}{q}>\tfrac{1}{2r}$ and $q \geq 4$.
        
    \item[(D')] The exponent in (D'), Proposition \ref{prop:bounds Aj d=2 D} is $2/q-\varepsilon$. Choosing $\varepsilon>0$ small enough, $2/q-\varepsilon>\frac{1}{r}$ is trivially satisfied by the assumption $\tfrac{1}{q}>\tfrac{1}{2r}$.  
    Note that the lower bound $r>3$ follows combining the assumptions $\tfrac{1}{q}>\tfrac{1}{2r}$ and $q \geq 6$. \qedhere
\end{enumerate}
\end{proof}

Combining Propositions \ref{prop:bounds with 1/r} and \ref{prop:bounds 1/r d=2} with the observations in  \S\ref{sec:edges} we get the following estimates for $V_r^IA$ for all $r\ge 1$. We use the trivial fact that $L^q(V_{r_0})$  is embedded in $L^q(V_{r_1})$ for $r_0<r_1$, which allows to overcome the $r \leq p$ or $r \leq \tfrac{q(d-1)}{d+1}$ constraints in the above Propositions.

\begin{corollary} Let $d\ge 3$. 
 $V_r^IA:L^p\to L^q$ is bounded
if one of the following conditions is satisfied:

(i) $(\tfrac 1p,\tfrac 1q)$ belongs to  the open 
 line segment $(Q_1(r), Q_2)$ or 
 the interior of the domain
$\fP_d(r)$ in Theorem \ref{dge3thm} ($r>\tfrac{d^2+1}{d(d-1)}$).

(ii) $(\tfrac 1p,\tfrac 1q)$ belongs to the open 
 line segment $(Q_1(r), Q_2)$ or 
 the interior of the domain
$\fP_d(r)$ in Theorem \ref{thm:intermediate r}  ($\tfrac{d}{d-1}<r\le \tfrac{d^2+1}{d(d-1)}$). 

(iii) $(\tfrac 1p,\tfrac 1q)$ belongs  to the open 
 line segment $(Q_1(r), Q_2(r))$ or 
 the interior of the domain $\fQ_d(r)$ in Theorem \ref{thm:small r}  ($1\le r\le \tfrac{d}{d-1}$ for $d\geq 4$ or $\tfrac{4}{3} < r \leq \tfrac{3}{2} $ for $d=3$).
 \end{corollary}
 
 \begin{corollary}
 Let $d=2$. $V_r^IA:L^p\to L^q$ is bounded
if one of the following conditions is satisfied:

(i) $(\tfrac 1p,\tfrac 1q)$ belongs  to the open 
 line segment $(Q_1(r), Q_2)$ or 
 the interior of the domain $\fQ_2(r)$  in Theorem \ref{d=2thm}, (i) ($r > \tfrac{5}{2}$).

(ii) $(\tfrac 1p,\tfrac 1q)$ belongs  to the open 
 line segment $(Q_1(r), Q_2)$ or 
 the interior of the domain $\fQ_2(r)$  in Theorem \ref{d=2thm}, (ii) ($2 < r\le \tfrac{5}{2} $).
\end{corollary}

\subsection{Various endpoint bounds}\label{sec:RWT}
We shall discuss various endpoint bounds that can be obtained by interpolation (in particular  Bourgain's  interpolation lemma as formulated in \S\ref{sec:Binterpol}).
This will settle all endpoint results claimed in our theorems except for a more sophisticated  strong type bound at the lower edges which will be discussed in the two subsequent sections.

We start by looking at the point $Q_3$.

\begin{lemma} \label{lem:Q3rwt} Let $d\ge 3$, $r>\frac{d}{d-1}$.  Let $p_3=\tfrac{d}{d-1}$, $q_3=d$. 

Then $\cA: L^{p_3,1}\to L^{q_3,\infty} (B^{1/r}_{r,1})$ is bounded. Consequently, $V_r^I A$ is of restricted weak type at $Q_3$ in Theorems \ref{dge3thm} and \ref{thm:intermediate r}.
\end{lemma}
\begin{proof} By standard embedding theorems, we can assume $r \leq 2$. For $r>\frac{d}{d-1}$ we have $\frac{d-1}{r'}-\frac 1r>0.$
We 
have from Corollary \ref{cor:1inftyr} and Proposition \ref{prop:bounds A_j with ST}, (A), \begin{align*}  \|\cA_j f\|_{L^\infty(L^r) } &\lc 2^{j(1-1/r)} \|f\|_{L^1},
    \\
    \|\cA_j f\|_{L^{r'}(L^r)} &\lc 2^{-j\frac{d-1}{r'}} \|f\|_{L^r},
    \end{align*} and by Corollary \ref{cor:trading-der}
    %Lemma \ref{errorlemma}
    \begin{align*}
    \|\cA_j f\|_{L^\infty(B_{r,1}^{1/r} ) } &\lc 2^{j} \|f\|_{L^1},
    \\
    \|\cA_j f\|_{L^{r'}(B_{r,1}^{1/r})} &\lc 2^{-j(\frac{d-1}{r'}-\frac 1r)} \|f\|_{L^r}. 
\end{align*}
The lemma then follows by applying  \S\ref{sec:Binterpol} to the last two inequalities. The bound for $V_rA$ is a simple corollary in view of \eqref{BPe}.
    \end{proof}

A similar argument yields a restricted weak type bound at $Q_4$.

\begin{lemma} \label{lem:Q4rwt} Let $d\ge 3$, $r>\frac{d^2+1}{d(d-1)}$ and $p_4=\frac{d^2+1}{d(d-1)}$, $q_4=\frac{d^2+1}{d-1}$. 

Then 
$\cA: L^{p_4,1}\to L^{q_4,\infty} (B^{1/r}_{r,1})$ is bounded. Consequently, $V_r^IA$ is of restricted weak type at $Q_4$ in Theorem \ref{dge3thm}.
\end{lemma}
\begin{proof}
By standard embedding theorems, we can assume $r \leq 2$.
By assumption on $r$ we have $\frac{d(d-1)}{(d+1)r'}-\frac 1r>0$.
It then suffices to interpolate using \S\ref{sec:Binterpol}
the inequalities
\begin{align*}
\| \cA_j f\|_{L^\infty(B_{r,1}^{1/r})} &\lc 2^j\|f\|_{L^1}
\\
\|\cA_j f\|_{L^{q_\circ} (B_{r,1}^{1/r}) } &\lc 
2^{-j(\frac{d}{q_\circ}-\frac 1r)} \lc \|f\|_{L^{p_\circ} }
\quad \text{ with } p_\circ= r,\quad q_\circ= \tfrac{d+1}{d-1}r';
\end{align*}
the last inequality follows from Proposition \ref{prop:bounds A_j with ST}, (E).
\end{proof}

\begin{corollary}\label{cor:Q2Q3Q4}  Let $d\ge 3$. Then the following hold:

(i) $V_r^IA:L^p\to L^q$ is bounded
if $(1/p,1/q)$ belongs to the open segment $(Q_3,Q_4)$ in Theorem \ref{dge3thm} ($r>\tfrac{d^2+1}{d(d-1)}$).

(ii) $V_r^IA:L^{p,1}\to L^q$ is bounded
if $(1/p,1/q)$ belongs to the half-open segment $[Q_2,Q_3)$ in Theorems \ref{dge3thm} and \ref{thm:intermediate r} ($r>\tfrac{d}{d-1}$).

\end{corollary}

\begin{proof}
Part (i) just follows from interpolation between Lemma \ref{lem:Q3rwt} and \ref{lem:Q4rwt}.

For part (ii), let $p=\tfrac{d}{d-1}$ and fix $q_2=\tfrac{d}{d-1}$ and $q_3=d$. For $\fz \in \Z^d$, let $Q_\fz=\prod_{i=1}^d[\fz_i,\fz_i+1)$ and let $Q_{\fz}^*$ be a cube centered at $\fz$ with sidelength $20 d$.  Write
$f=\sum_\fz f_\fz$ with $f_\fz=f\bbone_{Q_{\fz}}$. As $V_r^I A$ is local and the $Q_\fz^*$ have bounded overlap, by Hölder's inequality
\[\| V_r^I A f\|_{L^{q_2,\infty}}\le \Big( \sum_{\fz} \| \bbone_{Q_\fz^*} V_r^I A f_\fz \|_{L^{q_2,\infty}}^{q_2} \Big)^{1/{q_2}} \leq \Big( \sum_{\fz} \| \bbone_{Q_\fz^*} V_r^I A f_\fz \|_{L^{q_3,\infty}}^{q_2} \Big)^{1/{q_2}}.  \]
By Lemma \ref{lem:Q3rwt}, the right-hand side is further bounded by
$$
\Big( \sum_{\fz} \|  f_\fz \|_{L^{p,1}}^{q_2} \Big)^{1/{q_2}} \lesssim \| f \|_{L^{p,1}},
$$
as $p=q_2=p_3=\frac{d}{d-1}$. This implies that $V_r^I A$ is of restricted weak type at $Q_2$ if $r> \tfrac{d}{d-1}$. By interpolation between $Q_2$ and $Q_3$, one has that $V_r^I A$ is of restricted strong type on the open line segment $(Q_2,Q_3)$. Finally, the restricted strong type at $Q_2$ follows from the above localization argument, but using any of the just obtained $L^{p,1} \to L^{q}$ inequalities for $q_2 < q <q_3$ instead of the $L^{p,1} \to L^{q_3,\infty}$.
\end{proof}

\begin{remark}
One can obtain that $V_r^I A$ is of restricted weak type at $Q_2$ in Theorems \ref{dge3thm} and \ref{thm:intermediate r}  ($r>d/(d-1)$) by an application of \S\ref{sec:Binterpol} with the inequalities
\begin{align*}
    \|\cA_j f \|_{L^1(B_{r,1}^{1/r})} & \lesssim 2^{j} \| f \|_{L^1} \\
    \| \cA_j f \|_{L^2(B_{r,1}^{1/r})} & \lesssim  2^{-j(d-2)/2 + j/r }   \| f \|_{L^2}.
\end{align*}
Interpolation with the restricted weak type bound at $Q_3$ yields the restricted strong type bounds on the open line segment $(Q_2,Q_3)$. However, in order to deduce the restricted strong type at $Q_2$ we need to argue with the localization argument presented in the proof of Corollary \ref{cor:Q2Q3Q4} above.
\end{remark}

We next address the claimed bounds for $V_1^I A$ in Theorem \ref{thm:small r}.

\begin{lemma}\label{lem:1var} Let $d\ge 4$. The operator $\partial_t \cA$ maps $L^{\frac{d-1}{d-2},1} $ boundedly to $L^{d-1,\infty} (L^1) $. Consequently, $V_r^I A $ is of restricted weak type at $Q_3(1)$ in Theorem \ref{thm:small r}.
\end{lemma}
\begin{proof}
We have $\|\partial_t \cA_j f\|_{L^2(L^1)} \lc
\|\partial_t \cA_j f\|_{L^2(L^2)}  \lc 2^{-j\frac{d-3}{2} } \|f\|_{L^2}$. We interpolate the estimates
(obtained from Corollary \ref{cor:1inftyr} and Proposition \ref{prop:bounds A_j with ST} together with Corollary \ref{cor:trading-der})
\begin{align*}
\|\partial_t \cA_j f\|_{L^\infty(L^1)} &\lc 2^j \|f\|_{L^1}\\
  \|\partial_t \cA_j f\|_{L^2(L^1)}  &\lc 2^{-j\frac{d-3}2} \|f\|_{L^2}
\end{align*}
and obtain the conclusion by application of 
\S\ref{sec:Binterpol}.\end{proof}

\begin{corollary}
Let $d \geq 4$. The operator  $V_r^I A:L^{p,1} \to L^q$ is bounded if $(1/p,1/q)$ belongs to the half-open line segment $[Q_2(1), Q_3(1))$ in Theorem \ref{thm:small r}.
\end{corollary}

\begin{proof}
The restricted strong type bounds on $[Q_2(1),Q_3(1))$ can be obtained as in Corollary \ref{cor:Q2Q3Q4}.
\end{proof}

\begin{lemma} Let $d=3$. The operator $V_1^IA$ is bounded on $L^2(\bbR^3)$.
\end{lemma}

\begin{proof}By \eqref{eq:V1 BV} we have \[\|V_1^IA \|_{L^2} \le \Big\| \int_I |\partial_t \cA f(\cdot,t)|\, \ud t \Big\|_{L^2} \lc \Big(\iint |\partial_t \cA_t f|^2 \ud x\ud t\Big)^{1/2} \lc\|f\|_{L^2},
\]  by  \eqref{eq:L2} and orthogonality.
\end{proof}

\section{A maximal operator} \label{maxopsect}
%\subsection{A maximal operator} 
We first introduce an auxiliary  maximal function which will be crucial  in the proof of the endpoint bounds in \S \ref{sec:strong type endpoint}.

For $L\in \bbZ$ let  $\cQ_L$ be the family of all cubes in $\R^d$
with side length in $(2^{L-1}, 2^L]$. 
Given  $Q$ we  write \Be\label{defofL(Q)} L(Q)=L \quad \text{ if } Q\in \cQ_L.\Ee We use the slashed integral to denote an average, i.e.
\[\intslash_Q g(y)\ud y= \frac{1}{|Q|}\int_Q g(y)\ud y.
\]
For $x\in \bbR^d$ we let $\cQ_L(x)$ be the collection of all $Q\in \cQ_L$ containing $x$. 
Given $n =0,1,2,\dots$ and a sequence of functions $F=\{f_j\}_{j\geq 0}$, define the maximal function 
\Be\label{Mrn-def}
\fM_{r,n} F(x) = \sup_{j\ge n} \sup_{Q\in \cQ_{n-j}(x)} \intslash_Q \Big(\int|\cA_j f_j(y,t)|^r\ud t\Big)^{1/r}\ud y.
\Ee

The following result  should be compared with Lemma \ref{p0q0lemma}.
Away from the right boundary of the region in that lemma, we gain a crucial factor of $2^{-n\eps}$. Related statements can be found in \cite{RS2010}, \cite{PramanikRogersSeeger2011}  (see also  \cite{HNS2011} for  dual versions).

\begin{proposition} \label{maximal-op-prop} 
Let $p_0$ and $q_0$ such that $1< p_0\le q_0<\infty$.
Assume that 
\Be\label{p0q0-assumption}
\sup_{j \geq 0} 2^{j d/q_0} \|\cA_j\|_{L^{p_0}\to L^{q_0}(L^{p_0})} <\infty.\Ee
Let $q_0< q \leq  \infty$ and 
$\frac 1{\rho_{\max}(q)}= \frac {q_0}{q}\frac 1{p_0}$ and
$\frac{1}{\rho_{\min}(q)} =1-\frac{q_0}q(1-\frac 1{p_0})$.
Assume that  \Be\label{pqr-assu-with gain} \rho_{\min}(q)< p\le  q  \text{ and } 
\begin{cases} 
r\le p &\text{ if } \rho_{\min}(q)<p<\rho_{\max}(q),
\\
r<\rho_{\max}(q) &\text{ if } \rho_{\max}(q)\le p\le q.
\end{cases}
\Ee Then there exists $\varepsilon(p,q,r)>0$ such that
\Be\label{fMr-concl}
\|\fM_{r,n} F\|_{L^q}\le C_{p,q,r} 2^{-n\eps(p,q,r)} \Big(\sum_{j \geq n} 2^{-jd}\|f_j\|_{L^p}^q\Big)^{1/q}. 
\Ee 
\end{proposition}

For the proof we first observe a uniform estimate in $n$.

\begin{lemma} \label{Mrnuniformlemma}
Let $p_0\le q_0$ and assume  \eqref{p0q0-assumption} holds. Let $q>1$ and $q_0 \leq q \leq \infty$. Let
$\rho_{\min}(q)$, $\rho_{\max}(q)$
be as in \eqref{rhominmax}, and let
 $\rho_{\min}(q)\le p\le q$ and $0<r\le \min \{ p, \rho_{\max}(q)\} $. 
Then
\Be
    \label{Mrnuniform}
    \|\fM_{{r},n} F\|_{L^q} \lc  \Big(\sum_{j\ge n} 2^{-jd}\|f_j\|_{L^p}^{q}\Big)^{1/q}.
    \Ee
\end{lemma}

\begin{proof}
Let  $M_{HL}$ denote the Hardy--Littlewood maximal function. Then 
\begin{align*} |\fM_{r,n} F(x)| &\lc \sup_{x\in Q} \intslash_Q  \sup_{j \geq n} \Big(\int |\cA_j f_j(y,t)|^{r} \ud t\Big)^{1/r}\ud y
\\&\le M_{HL} \big[\sup_{j \geq n}\|\cA_j f_j\|_{L^{r}(\bbR)}\big](x)
\end{align*}
and therefore, since $r\le q$ and $q>1$
\begin{align*}\|\fM_{r,n} F\|_{L^q} &\lc  \big\|\sup_{j \geq n} \|\cA_j f_j\|_{L^{r}(\R)} \big\|_{{L^q}}
\lc \Big\|\Big( \sum_{j \geq n}\|\cA_j f_j \|_{L^{r}(\R) }^{q}\Big)^{1/q} \Big\|_{L^q}
\\ &= \Big(\sum_{j \geq n}
\|\cA_j f_j\|_{L^{q}(L^{r})}^{q} \Big)^{1/q} \lc 
\Big(\sum_{j \geq n} 2^{-jd} 
\| f_j\|_{L^p}^{q} \Big)^{1/q} ;
\end{align*}
here in the last step we have used Lemma \ref{p0q0lemma}.
\end{proof}

We now show how to gain over this inequality in the special case $r=p_0$.
\begin{lemma} \label{specialcase-p0-lemma}
Let $p_0\le q_0$ and assume \eqref{p0q0-assumption} holds. Then  for
$q_0\le q\le\infty$, $p_0\le p\le q$
\Be
\label{Mp0n-gain}
    \|\fM_{{p_0},n} F\|_{L^q} \lc  2^{-nd(\frac 1{q_0}-\frac 1q)} \Big(\sum_{j \geq n} 2^{-jd}\|f_j\|_{L^p}^{q}\Big)^{1/q}.
    \Ee
\end{lemma}
\begin{proof}
We use  real interpolation for the sublinear operator $\fM_{p_0,n}$. Then  \eqref{Mp0n-gain} follows from
\begin{subequations}
\Be
    \label{q=q0interpol}
    \|\fM_{{p_0},n} F\|_{L^{q_0}} \lc  \Big(\sum_{j \geq n} 2^{-jd}\|f_j\|_{L^p} ^{q_0}\Big)^{1/q_0} , \quad  p_0\le p\le q_0,
\Ee
    and
\Be
    \label{q=inftyinterpol}
    \|\fM_{p_0,n} F\|_{L^\infty} \lc  2^{-nd/q_0} \sup_{j \geq n} \|f_j\|_{L^p},  \quad p_0\le p\le \infty.
\Ee
\end{subequations}
Note that \eqref{q=q0interpol} immediately  follows by Lemma \ref{Mrnuniformlemma}.

We now show  \eqref{q=inftyinterpol}. Fix $x\in \bbR^d$, $j\geq n$ and $Q\in \cQ_{n-j}(x)$. Let $R_x$ be a cube of diameter $20d$  centered at $x$.
Then  split 
\begin{align*}
    \intslash_Q\Big(\int|\cA_j f_j(y,t)|^{p_0} \ud t\Big)^{1/p_0} \ud y \le  I(x)+II(x)
    \end{align*}
    where
    \begin{align*}
        I(x)&= \intslash_Q\Big(\int|\cA_j [\bbone_{R_x} f_j](y,t)|^{p_0} \ud t\Big)^{1/p_0} \ud y,
        \\
        II(x)&=\intslash_Q\Big(\int|\cA_j [\bbone_{R_x^\complement}f_j](y,t)|^{p_0} \ud t\Big)^{1/p_0} \ud y.
    \end{align*}
 Using   H\"older's inequality, then the assumption  \eqref{p0q0-assumption} and then again H\"older's inequality we get
 \begin{align*}
     I(x)&\le \Big(\frac{1}{|Q|}\int \|\cA_j [\bbone_{R_x} f_j](y, \cdot) \|_{L^{p_0}(\R)}^{q_0} \ud y\Big)^{1/q_0}
     \\&\lc |Q|^{-1/q_0} 2^{-jd/q_0} \|\bbone_{R_x}f_j\|_{L^{p_0}} \lc 2^{-nd/q_0} \|\bbone_{R_x} f_j\|_{L^{p_0}}
     \lc 2^{-nd/q_0} \| f_j\|_{L^p},
 \end{align*} 
 since $Q\in \cQ_{n-j}$ and $p\ge p_0$.
 
 Next we use estimate \eqref{kernel-error} (with $M>d$)
 \begin{align*} II(x)&\lc
\intslash_Q\int_{|y-w|\ge 1}  C_M(2^{j}|y-w|)^{-M}|f_{j}(w)| \ud w \ud y
\\ &\lc 2^{-j M} \|f_j\|_{L^p} \lc 2^{-n M} \|f_j\|_{L^p}.
\end{align*}
We combine the estimates for $I(x)$, $II(x)$ and then, after taking suprema in $x$, in $Q\in \cQ_{n-j}(x)$ and  in $j$, \eqref{q=inftyinterpol} follows.
\end{proof}

%%%%%%%%%%%%%%%%%%%%%%%%%%%%
%%
%   single scale interpolation
%%
%%%%%%%%%%%%%%%%%%%%%%%%%%%%%

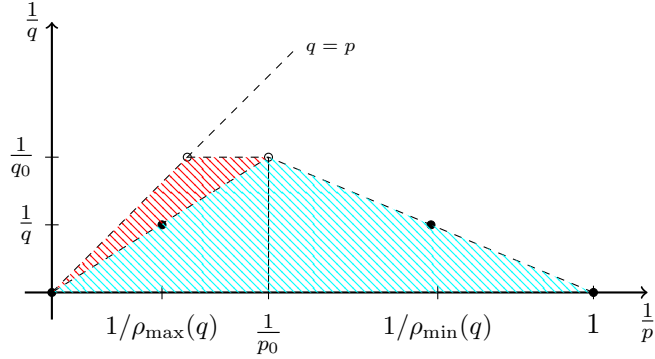
\begin{figure}[t]
\begin{tikzpicture}[scale=1.8] 

\begin{scope}[scale=2]
\draw[thick,->] (-.1,0) -- (2.2,0) node[below] {$ \frac 1 p$};
\draw[thick,->] (0,-.1) -- (0,1) node[left] {$ \frac 1 q$};
\draw[ dashed] (0,0) -- (0.9,0.9) node[right] {\tiny{$q=p$}};

%\draw (2,-0.025) -- (2,.025) node[below= 0.25cm] {$ 1$}; 
%\draw (-0.025,2) -- (.025,2) node[left= 0.25cm] {$ 1$}; 
%\node[circle,draw=black, fill=black, inner sep=0pt,minimum size=3pt] (b) at (2,2) {};
\node[circle,draw=black, fill=black, inner sep=0pt,minimum size=3pt] (b) at (0,0) {};

\node[circle,draw=black, fill=black, inner sep=0pt,minimum size=3pt] (b) at (2,0) {};

\draw (2,-0.025) -- (2,.025) node[below= 0.25cm] {$1$}; 

\draw (4/5,-0.025) -- (4/5,.025) node[below= 0.25cm] {$\frac{1}{p_0}$}; 
\draw (.025,1/2) -- (-.025,1/2) node[left] {$ \frac{1}{q_0}$};
\node[circle,draw=black, inner sep=0pt,minimum size=3pt] (b) at (4/5,1/2) {};

\draw[dashed]  (2,0)  -- (4/5,1/2); 
\draw[dashed] (0,0) -- (4/5,1/2);
%\draw[ dashed] (4/5,0) -- (4/5,1/2);
\draw (4/5, 0) -- (4/5,1/2);

\node[circle,draw=black, inner sep=0pt,minimum size=3pt] (b) at (1/2,1/2) {};

\node[circle,draw=black, fill=black, inner sep=0pt,minimum size=3pt] (b) at (5.6/4,1/4) {};

\draw (5.7/4,-0.025)--(5.7/4,.025) node[below= 0.25cm] {\small{1/$\rho_{\mathrm{min}}(q)$}};

\node[circle,draw=black, fill=black, inner sep=0pt,minimum size=3pt] (b) at (1.22/3,1/4) {};

\draw (1.22/3,-0.025)--(1.22/3,.025) node[below= 0.25cm] {\small{1/$\rho_{\mathrm{max}}(q)$}};

\draw (.025,1/4) -- (-.025,1/4) node[left] {$ \frac{1}{q}$};

\fill[pattern=north west lines, pattern color=cyan] (0,0) -- (2,0) -- (4/5,1/2) --  (0,0);

\fill[pattern=north west lines, pattern color=red] (0,0) -- (4/5,1/2) -- (1/2,1/2) -- (0,0);

\draw[dashed] (1/2,1/2) -- (4/5,1/2);
%\draw[dashed] (1/2,1/2) -- (0,0);

\end{scope}

\end{tikzpicture}

\caption{Bounds for $\fM_{r,n}$. At $r=p=p_0$, $q_0 < q\leq \infty$, we have a gain in $n$ given by Lemma \ref{specialcase-p0-lemma}. Interpolation with the uniform estimates from Lemma \ref{Mrnuniformlemma} for $r=p=\rho_{\min}(q)$ and $r=p=\rho_{\max}(q)$ yields the estimates in the interior of the blue triangle. The bounds in the red triangle follow by the localization argument.} 
\label{fig:max op int}

\end{figure}

The proof of Proposition \ref{maximal-op-prop} follows from (carefully) interpolating the two previous lemmas and a localization argument, as indicated in Figure \ref{fig:max op int}.

\begin{proof}[Conclusion of the proof of Proposition \ref{maximal-op-prop}]
We fix $q>q_0$. Observe that $\frac 1{\rho_{\min}(q)}-\frac 1{p_0}= (1-\frac{q_0}{q}) (1-\frac{1}{p_0})>0$ and $\frac{1}{p_0}-\frac{1}{\rho_{\max}(q)}=(1-\frac{q_0}{q})\frac 1{p_0}>0$ so that $\rho_{\min}(q)<p_0<\rho_{\max}(q)$. We first focus on the case $\rho_{\min}(q) < p < \rho_{\max}(q)$, for which it suffices to consider $r=p$; the corresponding inequality for smaller $r$  follows by H\"older's inequality on $[1/2,4]$. The remaining case $\rho_{\max}(q) \leq p \leq q$ will follow as a consequence of the previous range via a localization argument.

Let $p$ be as in \eqref{pqr-assu-with gain}. 
In what follows set $w(j)=2^{-jd}$ and let $\ell_w^q(L^p)$ be the space of $L^p$-valued sequences with \[\|F\|_{\ell^q_w(L^p)} = \Big(\sum_{j} 2^{-jd}\|f_j\|_{L^p}^q\Big)^{1/q}. \]

By linearization 
 it suffices to consider,    for any measurable choices of  positive integers $x\mapsto j(x)\in \bbN$, with $j(x) \geq n$, cubes $Q(x)\in \cQ_{n-j}(x)$, 
and  measurable  $L^{r'}(\bbR)$ valued functions $(x,y)\mapsto v(x,y,\cdot)$ in  $L^\infty(\bbR^{2d})$, 
the bilinear operator 
\[\cM_n
[F,v](x) =\intslash_{Q(x)}\int  v (x,y,s) 
 \cA_{j(x)} f_{j(x)} (y,s) \ud s \ud y
\]
and show that 
\Be \label{linearized-sit}
\|\cM_n [F,v]\|_{L^q} \lc  2^{-n\eps(p,q,r)} \|F\|_{\ell^q_w(L^p)} \|v\|_{L^\infty(L^{r'})},
\Ee

The  conclusion for $r\le p=p_0$ is immediate from Lemma \ref{specialcase-p0-lemma}, and in the study of the range $\rho_{\min}(q) < p < \rho_{\max}(q)$ we shall distinguish in what follows between  the $\rho_{\min}(q) < p< p_0$ and $p_0<p<\rho_{\max}(q)$. % and $\rho_{\max}(q)\le p\le q$.

\subsubsection*{The case $\rho_{\min}(q)<p\le p_0$} It suffices to prove \eqref{linearized-sit} for $r=p$.
We have from \eqref{Mrnuniform}
\Be\label{thetas=1}
\|\cM_n [F,v]\|_{L^q} \lc   \|F\|_{\ell^q_w(L^p)} \|v\|_{L^\infty(L^{p'})} \quad  \text{for  $p=\rho_{\min}(q)$}
\Ee
and from \eqref{Mp0n-gain} 
\Be\label{thetas=0}  \|\cM_n [F,v]\|_{L^q} \lc  2^{-nd(\frac 1{q_0}-\frac 1q)}  \|F\|_{\ell^q_w(L^{p_0} )} \|v\|_{L^\infty(L^{p_0'})}.
\Ee
One can interpolate \eqref{thetas=1} and \eqref{thetas=0}, noting that for $0<\theta<1$ and $\frac{1}{\rho_{\mathrm{min}}(q)}=1-\frac{q_0}{q} \frac{1}{p_0'}$,
\begin{equation*} \label{defoftheta}
\def\arraystretch{1.4}
\left. \begin{array}{l}
      \rho_{\min}(q)<p<p_0 \\ 
      \tfrac{1-\theta}{p_0}+\tfrac{\theta}{\rho_{\min}(q)}=\tfrac 1p
\end{array}\right\} \quad  \implies  \quad
(1-\theta) d(\tfrac{1}{q_0}-\tfrac 1q) =\tfrac{dp_0'}{q_0} \big(\tfrac{1}{\rho_{\min}(q)}-\tfrac{1}{p}\big),
\end{equation*}
and deduce
\begin{multline}\label{eps-pqr-smallp}
    \|\cM_n [F,v]\|_{L^q} \lc  2^{-n\eps(p,q,r)} \|F\|_{\ell^q_w(L^p)} \|v\|_{L^\infty(L^{p'})},
\\ \text{ for}\,\,\,  \rho_{\min}(q)\le p\le p_0 ,\quad
 \eps(p,q,p) = 
 \tfrac{dp_0'}{q_0} \big(\tfrac{1}{\rho_{\min}(q)}-\tfrac{1}{p}\big)>0
\end{multline}
with the implicit constants independent of the choices $j(x)$, $Q(x)$. Thus we also get 
\eqref{fMr-concl} for $r\le p$, $\rho_{\min}<p\le p_0$ and $\eps(p,q,r)=\eps(p,q,p)$ as in \eqref{eps-pqr-smallp}.

In order to  carry out the interpolation argument one uses Stein's interpolation theorem on analytic families of operators, with  an obvious analytic family suggested by the proof of the Riesz--Thorin theorem; we omit the details. Alternatively one can use Calder\'on's second method $[\cdot,\cdot]^\theta$, combining a result on multilinear maps with a result on Banach lattices such as $L^\infty(X)$, 
see  \cite[\S11.1]{Calderon-Studia}, 
\cite[\S13.6]{Calderon-Studia}. 

\subsubsection*{The case $p_0<p< \rho_{\max}(q)$} 
Again, it suffices to consider the case $r=p$.
Note that for $0 < \theta < 1$ and $\frac{1}{\rho_{\mathrm{max}}(q)}= \frac{q_0}{q} \frac{1}{p_0}$
\begin{equation}\label{defofvth}
\def\arraystretch{1.4}
\left. \begin{array}{l}
      p_0<p<\rho_{\max}(q) \\ 
      \tfrac{1-\vth}{p_0}+\tfrac{\vth}{\rho_{\max}(q)}=\tfrac 1p
\end{array}\right\} \quad  \implies  \quad
(1-\vartheta) d(\tfrac{1}{q_0}-\tfrac 1q)=\tfrac{dp_0}{q_0} \big(\tfrac 1p-\tfrac 1{\rho_{\max}(q)}\big) 
\end{equation}
We claim 
\begin{multline}\label{eps-pqr-middlep}
    \|\cM_n [F,v]\|_{L^q} \lc  2^{-n\eps(p,q,p)} \|F\|_{\ell^q_w(L^p)} \|v\|_{L^\infty(L^{p'})},
%\Big(\sum_j 2^{-jd}\|f_j\|_p^q\Big)^{1/q}, 
\\ \text{ for} \,\,\, p_0 \leq p \leq \rho_{\max}(q), \quad
\eps(p,q)=
\tfrac{dp_0}{q_0} \big(\tfrac 1p-\tfrac 1{\rho_{\max}(q)}\big)>0, 
\end{multline}
Given \eqref{defofvth},  the inequalities 
\eqref{eps-pqr-middlep} 
can then be deduced by the above indicated interpolation arguments  from
\Be\label{varthetas=1}
\|\cM_n [F,v]\|_{L^q} \lc   \|F\|_{\ell^q_w(L^p)} \|v\|_{L^\infty(L^{p'})} \quad \text{for $p=\rho_{\max}(q)$  \ }
\Ee
and 
\Be\label{varthetas=0}  \|\cM_n [F,v]\|_{L^q} \lc  2^{-nd(\frac 1{q_0}-\frac 1q)}  \|F\|_{\ell^q_w(L^{p_0} )} \|v\|_{L^\infty(L^{p_0'})}.
\Ee
Note that \eqref{varthetas=1}, \eqref{varthetas=0}  are immediate consequences of \eqref{Mrnuniform} and \eqref{Mp0n-gain}, respectively. 

\subsubsection*{The case $\rho_{\max}(q)\le p\le q$, $0<r<\rho_{\max}(q)$}
This case just follows by the localization argument used in the proof of Lemma \ref{localization-lemma} via the kernel estimates \eqref{kernel-error}, which allows to show that if
$$
\| \fM_{r,n} F \|_{L^q} \lesssim 2^{-n \varepsilon} \Big( \sum_{j \geq n} 2^{-jd} \| f_j \|_{L^{p_*}}^{q} \Big)^{1/q}
$$
holds for all $0 <r \leq r_*$ and some $1 \leq p_* \leq q$, then
$$
\| \fM_{r,n} F \|_{L^q} \lesssim 2^{-n \varepsilon} \Big( \sum_{j \geq n} 2^{-jd} \| f_j \|_{L^p}^{q} \Big)^{1/q}
$$
also holds for all $p_* \leq p \leq q$ and all $0 <r \leq r_*$. For fixed $q \in (q_0,\infty]$, the desired estimates for $\rho_{\max}(q) \leq p \leq q$ then follow from the above with input inequalities $r_*=p_*=\rho_{\max}(q)-\epsilon$ for $\epsilon$ arbitrarily small. Note that this argument has already been used in the context of  $\fM_{r,n}$ in the proof of \eqref{q=inftyinterpol} in Lemma \ref{specialcase-p0-lemma}. We omit the details. 
\end{proof}

\section{The sharp $L^p\to L^{pd}(L^p)$ bound}\label{sec:strong type endpoint}

In this section we will give  bounds for the sums of the operators $\cA_j$ which, in particular, will cover the crucial endpoint bound at $P(r)=(\tfrac 1r, \tfrac {1}{rd})$ in Theorem \ref{dge3thm}, as well as the remaining endpoints bounds stated in Theorems \ref{dge3thm}, \ref{thm:intermediate r} and \ref{thm:small r}.

\begin{proposition}\label{stpropnew} 
Let $1< p_0\le q_0<\infty$.
Assume that
\begin{equation}\label{eq:st prop endp hyp}
   \sup_{j \geq 0} 2^{jd/q_0} \|\cA_j\|_{L^{p_0} \to L^{q_0}(L^{p_0}) } 
  \le C_0\le \infty.
\end{equation}
Let $q_0< q< \infty$ and define  $\frac 1{\rho_{\max}(q)}= \frac {q_0}{q}\frac 1{p_0}$ and
$\frac{1}{\rho_{\min}(q)} =1-\frac{q_0}q(1-\frac 1{p_0})$. Assume that 
$p$, $r$ are as in \eqref{pqr-assu-with gain}, i.e. 
\Be\notag \rho_{\min}(q)< p\le  q  \text{ and } 
\begin{cases} 
r\le p &\text{ if } \rho_{\min}(q)<p<\rho_{\max}(q),
\\
r<\rho_{\max}(q) &\text{ if } \rho_{\max}(q)\le p\le q.
\end{cases}
\Ee

Then for all $\{f_j\}_{j \geq 0}$,
\Be\label{propconcl}\Big\| \sum_{j\ge 0} \|\cA_j f_j\|_{L^r(\bbR)}\Big\|_{L^q(\bbR^d)}  \le C(p,q) (1+C_0) 
\Big(\sum_{j \geq 0} 2^{-jd}\|f_j\|_{L^p(\R^d)}^q\Big)^{1/q}. \Ee
\end{proposition}

\begin{proof}  We first observe that by the monotone convergence theorem it suffices to show \eqref{propconcl} for any {\it finite } collection of functions $\{f_j\}_{j=0}^{\fn-1}$, with uniform bounds in $\fn\in \bbN$; moreover, all $f_j$ can be assumed to be in the Schwartz class. From  Lemma \ref{p0q0lemma} we get
\Be \label{Ndep}  \Big\| \sum_{j=0}^{\fn-1}  \|\cA_j f_j\|_{L^r(\bbR)}\Big\|_{L^q(\bbR^d)}  \lc \fn ^{1-1/q} 
\Big(\sum_j 2^{-jd}\|f_j\|_{L^p(\R^d)}^q\Big)^{1/q}, \quad q_0\le q \le \infty \Ee
and it is our task to remove the $\fn$-dependence in this estimate for $p,q,r$ as in the statement of the Proposition.

For a function $G\in L^{q_0}(\bbR^d)$  we recall the definition of
the Fefferman--Stein sharp maximal function
\[G^\#(x):=\sup_{x\in Q} \intslash_Q \Big|G(y)-\intslash_Q G(w) \ud w \Big| \ud y \]  
which satisfies the bound $\|G\|_{L^q}\le c(q) \|G^\#\|_{L^q} $ for every $q$ with $q_0<q<\infty$, and the implicit constant in this inequality is independent of the $L^{q_0}$-norm of $G$. This was proved in \cite{FSmult}. We  may apply this inequality to 
\[G(x) =\sum_{j\ge 0} \Big(\int |\cA_j f_j(x,t) |^r\ud t\Big)^{1/r},\]
as its $L^{q_0}$-norm is finite by \eqref{Ndep}; recall the sum is assumed to be finite. Let $\cQ(x)=\cup_{L\in \bbZ}\cQ_L(x)$, i.e. the family of cubes containing $x$. 
We estimate \[ G^\sharp (x) \lc \cG_I(x)+\cG_{II}(x)+\cG_{III}(x)\] where, with  $L(Q)$ as in \eqref{defofL(Q)},
\begin{equation*}
    \cG_I(x):=\,\, \sup_{\substack{Q\in \cQ(x)\\ L(Q)\le 0 }} \intslash_Q \Big|\!\! \sum_{0\le j\le -L(Q)}\!\! \Big( \|\cA_j f_j (y, \cdot)\|_{L^r(\R)} - \intslash_Q \|\cA_jf_j(w, \cdot)\|_{L^r(\R)}  \ud w \Big) \Big|\ud y,
    \end{equation*}
  \begin{align*}  \cG_{II} (x) &:= 
    \sup_{\substack{Q\in \cQ(x)\\ L(Q)\le 0 }} \intslash_Q  \sum_{ j\ge  -L(Q)} \|\cA_j f_j(y, \cdot)\|_{L^r(\R)}  \ud y,
    \\\
    \cG_{III}(x)&:= 
    \sup_{\substack{Q\in \cQ(x)\\ L(Q)> 0 }} \intslash_Q  \sum_{ j\ge 0} \|\cA_j f_j(y, \cdot)\|_{L^r(\R)} \ud y.
\end{align*}

The estimate for $\cG_{III}$ follows from the estimate for $\cG_{II}$. To see this let 
\[ U(y)= \sum_{j\ge 0} \|\cA_j f_j(y, \cdot)\|_{L^r(\R)}, \qquad U_*(w)=\sup_{\substack{Q\in \cQ(w) \\ L(Q)=0}} \intslash_Q  U(y) \ud y.\]
Given a cube $\widetilde Q\in \cQ(x)$ with $L(\widetilde Q)>0$ we may tile $\widetilde Q$ into cubes of  side length $1$ and get
\[\intslash\ci{\widetilde Q}U(y) \ud y \le \intslash\ci{\widetilde Q}  U_*(w) \ud w \le M_{HL} [U_*](x)\]
where $M_{HL}$ denotes the Hardy--Littlewood maximal operator. By a very  crude estimate we can replace $U_*$ by $\cG_{II} $ and get 
\Be \cG_{III}(x) \le M_{HL}[\cG_{II}] (x) .\label{GIIIptwise} \Ee

The term  $\cG_{II}$ is the most interesting  but it  has been already taken care of in \S\ref{maxopsect}. 
 We can write, with $U_j(y):=\|\cA_j f_j(y, \cdot)\|_{L^r(\R)} $
 \begin{align*} 
 \cG_{II}(x) 
 &=\sup_{\substack{Q\in \cQ(x)\\L(Q)\le 0}}\intslash_Q\sum_{n=0}^\infty U_{-L(Q)+n} (y) \ud y
 \\
& \le \sum_{n=0}^\infty
 \sup_{\substack{Q\in \cQ(x)\\L(Q)\le 0}}\intslash_Q
 U_{-L(Q)+n} (y) \ud y
  =\sum_{n=0}^\infty \sup_{j\ge n} \sup_{Q\in \cQ_{n-j}(x)} \intslash_Q U_j(y) \ud y .
 \end{align*}
 Hence,  
 with $\fM_{r,n}$ defined in \eqref{Mrn-def} and $F=\{f_j\}_{j \geq 0}$, we get
 \begin{align*}
     \cG_{II}(x) \le \sum_{n\ge 0} \fM_{r,n} F(x). 
 \end{align*}
  From  Proposition \ref{maximal-op-prop}, \eqref{GIIIptwise} and the $L^q$-boundedness of the Hardy-Littlewood maximal operator
 we obtain 
 \Be\label{GIIineq} 
    \|\cG_{II}\|_{L^q}+ \|\cG_{III}\|_{L^q}
    \lc \Big(\sum_{j \geq 0} 2^{-jd} \|f_j\|_{L^p}^q\Big)^{1/q} 
    %\text{ for $\rho_{\min(q)}< p<  \rho_{\max (q)}$, $r\le p$ .}
\end{equation}
for the range of $(p,q,r)$ assumed in the proposition.

It remains to  consider the term $\mathcal{G}_I$, where we can use the cancellative properties of the $\#$-function.
%ich amounts to the contribution of the small cubes. 
We will show that
\begin{equation}
    \label{GIclaim}
    \|\cG_I\|_{L^q}\lc \Big(\sum_{j \geq 0} 2^{-jd} \|f_j\|_{L^p}^q\Big)^{1/q} \quad \text{for $\rho_{\min}(q)\le p\le q$.}
\end{equation}
For $n\ge 0$ define
\begin{equation*}
    \cG_{I,n} (x):=  
    %\sup_{j\le n}
    \sup_{j\ge 0} \sup_{\substack{Q\in \cQ_{-n-j}(x)}} \intslash_Q \Big| \|\cA_{j} f_j(y,\cdot)\|_{L^r(\R)} - \intslash_Q  \|\cA_jf_j(w,\cdot)\|_{L^r(\R)} \ud w \Big|\ud y
    \end{equation*} 
    and, arguing as for $\cG_{II}$, we observe that 
    $\cG_I(x)\le \sum_{n\ge 0} \cG_{I,n}(x). $
 Our claim \eqref{GIclaim} will follow from the estimate
    \begin{equation}
    \label{GIclaimn}
    \|\cG_{I,n}\|_{L^q}\lc 2^{-n} \Big(\sum_{j \geq 0} 2^{-jd} \|f_j\|_{L^p}^q\Big)^{1/q} \quad \text{for $\rho_{\min}(q)\le p\le q$}
\end{equation}
uniformly in $n$. In order to show this,  note that by the triangle inequality
\[\cG_{I,n}(x)\le  
\sup_{j\ge 0} \sup_{\substack{Q\in \cQ_{-n-j}(x)}} \intslash_Q \intslash_Q  \big\| \cA_j f_j(y,\cdot)- \cA_j f_j(w,\cdot) \big\|_{L^r(\R)} \ud w\ud y.
\]
Write $\cA_j f_j=\widetilde L_j\cA_j f_j$ and let $\theta_j$ be the convolution kernel of $\widetilde L_j$. Then for $j\le n$, $x,y,w \in Q$, $Q\in \cQ_{-n-j}$
\begin{align*}
&\Big(\int \Big| \cA_{j} f_{j}(y,t) -  \cA_{j} f_j(w,t) \Big|^r \ud t\Big)^{1/r} 
\\&\le  \Big(\int\Big| \int (\theta_j(y-z)-\theta_j(w-z)) \cA_j f_j(z,t) \ud z \Big|^r  \ud t\Big)^{1/r} 
\\
&\le\int
\int_{0}^1 |\inn{y-w} {\nabla\theta_j(w+\tau(y-w)-z)}| \ud \tau  \Big(\int\big|
 \cA_{j}f_{j} (z, t)\big|^r \ud t\Big)^{1/r} \ud z. 
\end{align*}
Since $1+2^j|w+\tau(y-w)-z| \approx  1+2^j|x-z|$ for $x, y,w\in Q$, $Q\in \cQ_{-n-j} (x)$, $\tau\in [0,1]$, we can estimate the last expression by
\begin{align*} 
C_N 2^{-n-j} \int   \frac{ 2^{j(d+1)}}
{(1+2^j|x-z|)^N }\Big(\int  |\cA_{j}f_{j} (z,t)|^r\ud t\Big)^{1/r}  \ud z
\end{align*}
and hence we get  for $n\ge 0$ and $N>d$
%\begin{align*}
  \begin{equation*}
  |\cG_{I,n}(x)| \lc 2^{-n}   M_{HL} \big[ \sup_{j \geq 0}\|\cA_j f_j\|_{L^r(\R)} \big] (x).
  \end{equation*}
This implies, using the $L^q$ boundedness of the Hardy--Littlewood maximal operator $M_{HL}$ and Lemma \ref{p0q0lemma},
\begin{align*}
\|\cG_{I,n}  \|_{L^q} & \lc 2^{-n}   \big\| \sup_{j \geq 0} \|\cA_j f_j \|_{L^r(\R)}\big\|_{L^q}
\lc 2^{-n} \Big\| \Big(\sum_{j \geq 0} \|\cA_j f_j \|_{L^r(\R)}^q\Big)^{1/q}\Big\|_{L^q}
\\ & = 2^{-n}\Big(\sum_{j \geq 0}\|\cA_j f_j\|^q_{L^q(L^r)}\Big)^{1/q} \lc
2^{-n} \Big(\sum_{j \geq 0} 2^{-jd} \|f_j\|^q_{L^p}\Big)^{1/q},
\end{align*}
which is \eqref{GIclaimn}  and thus \eqref{GIclaim} is proved.
The proof is complete after combining \eqref{GIclaim} and  \eqref{GIIineq}.
\end{proof}

As \eqref{eq:st prop endp hyp} holds with $p_0=2$, $q_0=\frac{2(d+1)}{d-1}$ by Lemma \ref{stein-squarefct}, Proposition \ref{stpropnew} yields the following.

\begin{corollary} \label{stpropcor}
Assume that %$\frac{2(d+1)}{d-1}\le q\le \infty$, and suppose that 
\Be \label{eq:cor exp endp triangle}
\tfrac{2(d+1)}{d-1}< q < \infty, \quad \tfrac{d+1}{d-1}\tfrac 1 q   <  \tfrac 1p< 1- \tfrac{d+1}{d-1}\tfrac 1q, \quad r \leq p
\Ee
or 
\Be \label{eq:cor exp endp left}
\tfrac{2(d+1)}{d-1} < q < \infty, \quad  \tfrac 1q \leq \tfrac 1p  \leq \tfrac{d+1}{d-1}\tfrac 1 q   , \quad r < \tfrac{q(d-1)}{d+1}.
\Ee
Then for all $\{f_j\}_{j \geq 0}$,
\Be\notag\Big\| \sum_{j\ge 0} \|\cA_j f_j\|_{L^r(\bbR)}\Big\|_{L^q(\bbR^d)}  \lc \Big(\sum_{j \geq 0} 2^{-jd}\|f_j\|_{L^p(\R^d)}^q\Big)^{1/q}. \Ee
\end{corollary}

Combining this with Lemma \ref{lemma:besov reduction} one obtains the strong type bound at $P(r)$ in Theorem \ref{dge3thm} (that is, Corollary \ref{cor:lower boundary edge}).

\begin{proposition} \label{stpropLr} 
Let $d\ge 3$,  $\frac{d^2+1}{d(d-1)}<p<\infty$. 
Then
\begin{equation}\label{stBesovr} \| \cA f \|_{L^{pd} (B^{1/p}_{p,1})} \le  \|f\|_{L^p}. \end{equation}
Moreover,
\begin{equation}\label{stVar-r} 
\| V^I_p A f \|_{L^{pd}}  \le  \|f\|_{L^p}. \end{equation}
\end{proposition}

\begin{proof} The bound \eqref{stVar-r} follows from \eqref{stBesovr}  via \eqref{BPe}. 

To prove \eqref{stBesovr}  we apply Corollary \ref{stpropcor} and Lemma \ref{lemma:besov reduction} with $s=d/q=1/p$. We  verify the assumption \eqref{eq:cor exp endp triangle} for $q=pd$, $r=p$. The condition
$\tfrac{d+1}{d-1}\tfrac 1 q<\tfrac 1p$ is satisfied (for $q=pd$) when $d^2-2d-1>0$, which holds when $d\ge 3$.
The condition $\tfrac 1p< 1- \tfrac{d+1}{d-1}\tfrac 1q$ is satisfied (for $q=pd$) if $p>\frac{d^2+1}{d(d-1)}$. The condition $q > \frac{2(d+1)}{d-1}$ is also satisfied if $q=pd$, $p>\frac{d^2+1}{d(d-1)}$. In particular, the latter implies that $q=pd>2$ in this range, so the hypothesis of Lemma \ref{lemma:besov reduction} are also satisfied and thus \eqref{stBesovr} follows.
\end{proof}

Arguing in a similar way, we obtain the remaining claimed endpoint bounds in Theorems \ref{dge3thm}, \ref{thm:intermediate r} and \ref{thm:small r}.

\begin{proposition} \label{prop:endp strong big r}
Let $d\ge 3$ and $r>\frac{d^2+1}{d(d-1)}$.

(i)  Let $r\le p\le rd$. Then the  operators $\cA: L^p\to L^{rd}(B^{1/r}_{r,1})$  and $V_r^I A: L^p\to L^{rd}$ are bounded. That is, $V_r^I:L^p \to L^q$ is bounded if $(1/p,1/q)$ belongs to the closed segment $[Q_1(r),P(r)]$ in Theorem \ref{dge3thm}.

(ii) Let $\frac{d^2+1}{d(d-1)}<p\le r$. Then the operators 
 $\cA: L^p\to L^{pd}(B^{1/r}_{r,1})$  and $V_r^I A: L^p\to L^{pd}$ are bounded. That is, $V_r^I:L^p \to L^q$ is bounded if $(1/p,1/q)$ belongs to the half-open segment $[P(r),Q_4)$ in Theorem \ref{dge3thm}.
\end{proposition}

\begin{proof}
For part (i) we use again Corollary \ref{stpropcor} and Lemma \ref{lemma:besov reduction} with $s=d/q=1/r$. The condition \eqref{eq:cor exp endp triangle} yields the ranges $\frac{rd(d-1)}{rd(d-1) -(d+1)} < p < \frac{rd(d-1)}{d+1}$ and $r \leq p$. Note that $r > \frac{rd(d-1)}{rd(d-1) -(d+1)}$ if and only if $r > \frac{d^2+1}{d(d-1)}$; moreover, the condition $q>\frac{2(d+1)}{d-1}$ (for $q=rd$) is satisfied in the range $r > \frac{d
^2+1}{d(d-1)}$ whenever $d^2-2d-1>0$, which holds for $d \geq 3$. This settles the range $r \leq p < \frac{rd(d-1)}{d+1}$. The range $\frac{rd(d-1)}{d+1} \leq p \leq rd$ corresponds to \eqref{eq:cor exp endp left}. Note that the condition $r < \frac{q(d-1)}{d+1}$ requires (for $q=rd$) $d^2-2d-1>0$, which holds when $d \geq 3$. The condition $q > \frac{2(d+1)}{d-1}$ was already verified for \eqref{eq:cor exp endp triangle}. This concludes the bounds in (i).

Part (ii) follows from standard embedding theorems from Proposition \ref{stpropLr}.
\end{proof}

\begin{proposition}
Let $d\ge 4$ and $1\le r\le \frac{d^2+1}{d(d-1)}$ or $d=3$ and $4/3<r\le 5/3$. Let $\frac{rd(d-1)}{rd(d-1)-d-1} < p\le rd$. Then the  operators $\cA: L^p\to L^{rd}(B^{1/r}_{r,1})$  and $V_r^I A: L^p\to L^{rd}$ are bounded. That is, $V_r^I:L^p \to L^q$ is bounded if $(1/p,1/q)$ belongs to the half-open segment $[Q_1(r),Q_4(r))$ in Theorems \ref{thm:intermediate r} and \ref{thm:small r}.
\end{proposition}

\begin{proof}
This follows arguing as in the proof of Proposition \ref{prop:endp strong big r}. The only difference is that in \eqref{eq:cor exp endp triangle}  the relevant range for $p$ (for $q=rd$) if $r \leq \frac{d^2+1}{d(d-1)}$ is  $\frac{rd(d-1)}{rd(d-1)-d-1} < p < \frac{rd(d-1)}{d+1}$. Moreover, the condition $q> \frac{2(d+1)}{d-1}$ requires (for $q=rd$) that $r>\frac{2(d+1)}{d(d-1)}$. As $r \geq 1$, this condition is satisfied if $d^2-3d-2>0$, which holds for $d \geq 4$. If $d=3$, we require the restriction $r > 4/3$.
\end{proof}

\section{An endpoint bound for the global variation}\label{sec:endpoint global}
The purpose of this section is to prove Theorem \ref{endptvarthm}.

It will be useful to work with the standard homogeneous Littlewood--Paley decomposition. We define
$P_j f$, $j\in \bbZ$  by \[\widehat {P_j f} (\xi)=\big( \beta_0(2^{-j} |\xi|)- \beta_0(2^{1-j}|\xi|)\big)
\widehat f(\xi)\] so that $P_j$ localizes to frequencies of size $\approx 2^j$. We have $P_j=L_j$ for $j\ge 1$ and $L_0f=\sum_{j\le 0} P_j f$ for $f\in L^p$, $p\in (1,\infty)$ with convergence in $L^p$.
It will also be convenient to use reproducing operators $\widetilde P_j$ localizing to slightly larger frequency annuli, with $\widetilde P_j P_j=P_j$ for $j\in \bbZ$.

Let $j\geq 0$. We recall the definition
\[\cA_j f(x,t) = \chi(t) A_t L_j f(x)= \chi(t) K_{j,t}*f(x) 
\text{ where }  \widehat {K_{j,t}}(\xi)  =\widehat  \sigma(t\xi) \beta_j(|\xi|) .
\]
We  combine this with dyadic dilations, and define for $k\in \bbZ$, $t\in [1/2,4]$,
\begin{equation}\label{eq:def Ajk} 
\cA_j^kf(x,t) = \chi(t) A_{2^k t} P_{j-k} f(x)= \chi(t) K^k_{j,t} *f (x) 
\end{equation}
where
$K^k_{j,t}= 
2^{-kd} K_{j,t}(2^{-k}\cdot)$. 
Below we shall often use a scaled version of \eqref{kernel-error}, namely
\Be
\label{kernel-error-k}
|K_{j,t}^k(x)|\lc_N 2^{-kd} (2^{j-k}|x|)^{-N},\qquad |x|\ge 10\cdot 2^k, \quad t\in [1/2,4].\Ee

We start recording the following special case of Proposition \ref{prop:bounds A_j with ST}, which will be relevant for the proof of Theorem \ref{endptvarthm} (when $p=q$).

\begin{corollary}\label{cor:stein-r-fct}
For
$2\le r\le \infty$, 
$\frac{r(d+1)}{d-1}\le q\le \infty$, $r\le p\le q$ we have 
\begin{align*}
\|\cA_j f\|_{L^q(L^r)}\lesssim  2^{-jd/q} \|f\|_{L^p}. 
\end{align*}
\end{corollary}

By Corollary \ref{cor:stein-r-fct} and rescaling we have
\Be\label{Ajk-Lp}
\|\cA_j^k\|_{L^p\to L^p(L^r)} \lc 2^{-jd/p}, \quad \tfrac{r(d+1)}{d-1}\le p\le \infty.
\Ee
One can improve over this result and extend it to sums in $k$ whenever $\tfrac{r(d+1)}{d-1} < p <\infty$.
\begin{proposition}\label{global-fixedj}
For $2\le r<\infty$, $\frac{r(d+1)}{d-1}< p<\infty $ we have, for all $j\ge 1$, 
\Be\label{eq:global fixed j}
\Big\| \Big(\sum_{k\in \bbZ} \|\cA_j^k f \|_{L^r(\bbR)}^r\Big)^{1/r} \Big\|_{L^p(\bbR^d)} \lc 2^{-jd/p} \|f\|_{L^p(\R^d)}.
\Ee
\end{proposition} 

\begin{proof}
As in the proof of Proposition \ref{stpropnew}, by the monotone convergence theorem, it suffices to show \eqref{eq:global fixed j} for any finite collection $\{\cA_j^k\}_{k \in K}$, with uniform bounds on the cardinality of the finite set $K \subset \Z$. 

We use again the sharp maximal function of Fefferman--Stein. Let
\begin{equation*}
G(x)=    \big(\sum_{k \in K} \|\cA_j^k f(x,\cdot)\|_{L^r(\R)}^r \big)^{1/r},
\end{equation*}
which has finite $L^{p_0}$ norm whenever $p_0= \tfrac{r(d+1)}{d-1}$; note that \eqref{Ajk-Lp} and Minkowski's inequality imply that 
$$
\| G \|_{L^{p_0}} \lesssim |K|^{1/r} 2^{-jd/p} \| f \|_{L^{p_0}}.
$$
By the bound $\| G \|_{L^p} \lesssim_p \| G^\# \|_{L^p}$,
it suffices to show  $\| G^\# \|_{L^p} \lesssim 2^{-jd/p} \| f \|_{L^p}$, uniformly on the finite set $K$ for $p_0 < p < \infty$. By the triangle inequality, we dominate
\begin{align*}
%\Big\| 
\notag
G^\#(x) & 
 \leq \sup_{Q\in \cQ(x)} \intslash_Q \intslash_Q \Big(\sum_{k \in \Z}  \|\cA_j^k f(y,\cdot) -  \cA_j^k f(w,\cdot) \|_{L^r(\R)}^r   \Big)^{1/r} \ud w  \ud y \\
& \leq 2 \, \cG f(x) +\sum_{n=1}^\infty \cU_n f(x)
\end{align*}
where
\begin{equation*} \cG f(x):=\sup_{L\in \bbZ}\sup_{Q\in \cQ_L(x)} \intslash_Q\Big(\sum_{k\le L} \|\cA_j^k f(y,\cdot) \|_{L^r(\R)}^r\Big)^{1/r} \ud y
\end{equation*}
and
\begin{equation*}
\cU_n f(x):=\sup_{L\in \bbZ}\sup_{Q\in \cQ_L(x)} \intslash_Q\intslash_Q \big\|\cA_j^{L+n} f(y,\cdot)- \cA_j^{L+n} f(w,\cdot) \big\|_{L^r(\R)} \ud w\ud y.
\end{equation*}
We claim that  for $\frac{r(d+1)}{d-1} \le p \le \infty$, $2 \leq r < \infty$ we have 
\Be\label{Gbound}
\|\cG f\|_{L^p} \lc 2^{-j \alpha(r)} \|f\|_{L^p}, \quad \text{ for some } \alpha(r)>d/p
\Ee 
and
 \Be\label{Unbound} 
\|\cU_n f\|_{L^p} \lc \begin{cases}
 2^{(n-j)d(\frac{d-1}{r(d+1)}-\frac 1p)} 2^{-jd/p} \|f\|_{L^p}, &\text{ if $1\le n\le j$,}
\\
 2^{j-n} 2^{-jd/p}\|f\|_{L^p}, &\text{ if $n>j$.}
\end{cases}  
\Ee
The desired bound $\| G^\# \|_{L^p} \lesssim 2^{-jd/p} \| f \|_{L^p}$  follows summing in $n$ if $p>\frac{r(d+1)}{d-1}$.

\subsection*{Proof of \eqref{Gbound}} We prove inequalities for $\cG$ on $L^r$ and $L^\infty$ which will yield \eqref{Gbound} by interpolation.

Let $p_0=\frac{r(d+1)}{d-1}$. We first observe the inequalities 
\begin{align}
 &\|\cA_j^k f \|_{L^2(L^r)} \lc 2^{-j(\frac{d-2}{2}+ \frac 1r)} \|f\|_{L^2}
 \label{2rbd}
    \\
    &\|\cA_j^k f\|_{L^{p_0}(L^r)} \lc 2^{-jd/p_0} \|f\|_{L^{p_0}}
    \label{standardp0}
    \end{align}
uniformly in $k$. The estimate \eqref{2rbd}  holds from  the bounds $\|\cA^k_j \|_{L^2\to L^2(L^2)}\lc 2^{-j(d-1)/2}$ and $\|\partial_t \cA^k_j \|_{L^2\to L^2(L^2)}\lc 2^{-j(d-3)/2}$  (which follow from \eqref{eq:L2}) and the  Sobolev embedding theorem, and
 \eqref{standardp0} is \eqref{Ajk-Lp} with $p=p_0$.
Since $2\le r< p_0$ and $\frac{d-2}{2}+\frac 1r>\frac d{p_0}$  one has by interpolation that
\begin{equation}\label{eq:decay (r,r,r)}
\|\cA_j^k f\|_{L^{r}(L^r)} \lc 2^{-j\alpha(r)} \|f\|_{L^{r}}, \quad \text{for some } \alpha(r)>d/p_0. 
\end{equation}
This implies 
\begin{align}
    \|\cG f\|_{L^r} &\lc \Big\| M_{HL}\big[  (\sum_{k \in \Z}\|\cA_j^k \widetilde P_{j-k} f\|_{L^r(\R)}^r)^{1/r} \big] \Big\|_{L^r}\notag
    \\
    &\lc \Big(\sum_{k \in \Z} \|\cA_j^k \widetilde P_{j-k} f\|_{L^r(L^r)}^r \Big)^{1/r} \notag
    \\ &\lc 2^{-j\alpha(r)} \Big(\sum_{k \in \Z} \| \widetilde P_{j-k} f\|_{L^r}^r \Big)^{1/r} \lc{2^{-j\alpha(r)}}\|f\|_{L^r}
    \label{cGr}
\end{align}
by Littlewood--Paley theory, since $r\ge 2$.

We now prove an $L^\infty$ bound. Fix $x$, $L$, $Q\in \cQ_L(x)$ and   let $B_x^L$ be the ball centered at $x$ with radius $d 2^{L+10} $. Using Hölder's inequality and the embedding $\ell^1 \subseteq \ell^r$ for $r \geq 1$ we estimate  
\begin{align*} 
\intslash_Q\Big(\sum_{k\le L} \big\|\cA_j^k  f(y,\cdot)\big \|_{L^r(\R)}^r\Big)^{1/r} \ud y 
\le I(x)+II(x)
\end{align*}
where
\begin{align*}
    I(x)&= 
    \Big(\intslash_Q\sum_{k\le L} \big\|\cA_j^k  [\bbone_{B_x^L} f] (y,\cdot)\big\|_{L^r(\R)}^r\ud y \Big)^{1/r} 
    \\
    II(x)&=\intslash_Q\sum_{k\le L} \big\|\cA_j^k [\bbone_{\R^d \backslash B_x^L} f] (y,\cdot) \big\|_{L^r(\R)}\ud y. 
\end{align*}
We have, in view of $\cA_j^k=\cA_j^k \widetilde{P}_{j-k}$ and using \eqref{eq:decay (r,r,r)},
\begin{align*}
    I(x)&\lc 2^{-Ld/r} \Big(\sum_{k\le L} \big\|\cA_j^k  \widetilde P_{j-k}[\bbone_{B_x^L}f] \big\|^r_{L^r(L^r)} \Big)^{1/r} 
    \\
    &\lc 2^{-j\alpha(r)}
    2^{-Ld/r} \Big(\sum_{k\le L} \big\|\widetilde P_{j-k}[\bbone_{B_x^L}f] \big\|^r_{L^r} \Big)^{1/r} 
    \\
    &\lc 2^{-j\alpha(r)}
    2^{-Ld/r} \|\bbone_{B_x^L}f\|_{L^r} 
    \lc 2^{-j\alpha(r)}
    \|f\|_{L^\infty},
\end{align*}
using $r \geq 2$ and Littlewood--Paley theory in the third inequality. For the term $II(x)$ we use \eqref{kernel-error-k} and estimate
\begin{align*}
    II(x)&\lesssim_N \sum_{k\le L} \intslash_Q \int_{\R^d \backslash B_x^L} \frac{2^{-kd}}{(2^{j-k}|y-w|)^{N}} |f(w)| \ud w\, \ud y
    \\
    &\lesssim 2^{-jN} \sum_{k\le L} 2^{(L-k)(d-N)} \|f\|_{L^\infty} \lc 2^{-jN}\|f\|_{L^\infty},
\end{align*} 
where $N >d$. We combine the estimates for $I(x)$ and $II(x)$ to obtain
\Be \label{cGinfty}\|\cG f\|_{L^\infty} \lc 2^{-j\alpha(r)} \|f\|_{L^\infty}. \Ee 
Interpolating \eqref{cGr} and \eqref{cGinfty} and noting that that $\alpha(r) > d/p_0 \geq d/p$ for $p \geq p_0$ we obtain \eqref{Gbound}.

\subsection*{Proof of \eqref{Unbound} for $1\le n\le j$}  This case is similar to that of $\cG$. 
Let $p_0=\frac{r(d+1)}{d-1}$. 
We get the asserted estimate by interpolating  between the inequalities
\begin{align}
\label{Unp0}
\|\cU_n f\|_{L^{p_0}} &\lc 2^{-jd/p_0} \|f\|_{L^{p_0}}
\\
\label{Uninfty}
\|\cU_n f\|_{L^\infty} &\lc 2^{(n-j)d/p_0}\|f\|_{L^\infty}.
\end{align}

To see \eqref{Unp0}, we use $\cA_j^k = \cA_j^k \widetilde{P}_{j-k}$ and \eqref{standardp0} to estimate 
\begin{align*}
    \|\cU_n f\|_{L^{p_0}} &\lc \Big\| M_{HL} \big[\big(\sum_{k \in \Z} \|\cA_j^k \widetilde{P}_{j-k} f\|_{L^r(\R)} ^{p_0})^{1/p_0}  \big ] \Big\|_{L^{p_0}}
    \\
    &\lc \Big(\sum_{k \in \Z} \|\cA_j^k \widetilde{P}_{j-k} f\|_{L^{p_0}(L^r)} ^{p_0})^{1/p_0}  
    \\
    &\lc 2^{-jd/p_0} \big(\sum_{k \in \Z} \|\widetilde{P}_{j-k} f\|_{L^{p_0}}^{p_0})^{1/p_0} 
    \lc 2^{-jd/p_0} \|f\|_{L^{p_0}},
\end{align*}
using that $p_0 \geq 2$ and Littlewood--Paley theory in the last inequality.

To see \eqref{Uninfty}
we   fix $x$, $L$, $Q\in \cQ_L(x)$,  let $B_x^{L+n}$ be the ball centered at $x$ with radius $d 2^{L+n+10}$ and estimate 
\begin{align*} 
&\intslash_Q \intslash_Q\big\|\cA_j^{L+n} f(y,\cdot)- \cA_j^{L+n} f(w,\cdot) \big \|_{L^r(\R)} \ud w\ud y 
\\
&\lc \intslash_Q\big\|\cA_j^{L+n} \widetilde P_{j-L-n} f(y,\cdot)\big \|_{L^r(\R)} \ud y 
\le III(x)+IV(x)
\end{align*}
where
\begin{align*}
    III(x)&= 
    \Big(\intslash_Q \big\|\cA_j^{L+n}  [\bbone_{B_x^{L+n}} f] (y,\cdot)\big\|_{L^r(\R)}^{p_0}\ud y \Big)^{1/p_0} ,
    \\
    IV(x)&=\intslash_Q \big\|\cA_j^{L+n} [\bbone_{\bbR^d\setminus B_x^{L+n}} f] (y,\cdot) \big\|_{L^r(\R)}\ud y .
\end{align*}
We get by \eqref{Ajk-Lp}
\begin{align*}
III(x)&\lc 2^{-Ld/p_0} 2^{-jd/p_0} \|
\bbone_{B_x^{L+n}} f\|_{L^{p_0}}
\lc 2^{(n-j)d/p_0} \|f\|_{L^\infty}.
\end{align*}
Moreover, by \eqref{kernel-error-k},
\begin{align*}
    IV(x)&\lc  \intslash_Q \int_{\bbR^d\setminus B_x^{L+n}} \frac{2^{-(L+n)d}}{(2^{j-L-n}|y-w|)^{N}} |f(w)| \ud w\, \ud y
    \lc 2^{-jN}\|f\|_{L^\infty}
\end{align*} 
for any $N>d$. The estimates for $III(x)$ and $IV(x)$ yield \eqref{Uninfty} for $1\le n\le j$.

\subsection*{Proof of \eqref{Unbound} for $n>j$}
Here we use cancellation and note that for $x\in Q$  
%In analogy to  \eqref{useofcancel} we now have for $x\in Q$
\begin{multline}\label{useofcancel2}\notag
\intslash_Q \intslash_Q\big\|\widetilde P_{j-L(Q)-n} g(y,\cdot)- \widetilde P_{j-L(Q)-n}g(w,\cdot) \big \|_{L^r(\R)}\ud w\ud y \\ \lc
2^{j-n} M_{HL} [\|g\|_{L^r(\R)}](x).
% \quad \text{ if }  \,\,\, x\in Q.
\end{multline}
Using this with $g=\cA_j^{L+n}f=\widetilde{P}_{j-L-n} \cA_j^{L+n}$ and the Fefferman--Stein inequality for sequences of Hardy--Littlewood maximal functions, we may then estimate 
\begin{align*} 
&\Big\| \sup_{L \in \Z} \sup_{Q\in \cQ_L(x)} \intslash_{Q }\intslash_Q\big\|\cA_j^{L+n} f(y,\cdot)- \cA_j^{L+n} f(w,\cdot) \big \|_{L^r(\R)} \ud w\ud y \Big\|_{L^p(dx)}
\\
&\lc 2^{j-n} 
\Big\| \sup_{k \in \Z} M_{HL}\big[\|\cA_j^{k} f\|_{L^r(\R)}\big]  \Big\|_{L^p}
\lc 2^{j-n} 
\Big(\sum_{k \in \Z} \big\|\cA_j^{k} \widetilde P_{j-k}f\|_{L^p(L^r)}^{p} \Big)^{1/p} 
\\
&\lc 2^{j-n} 2^{-jd/p} 
\Big(\sum_{k \in \Z} \|\widetilde P_{j-k}f\|_{L^p}^{p} \Big)^{1/p} 
\lc 2^{j-n} 2^{-jd/p} \|f\|_{L^p}
\end{align*}
for $\frac{r(d+1)}{d-1} \leq p \leq \infty$,  using \eqref{Ajk-Lp}  in the third inequality and $p \geq 2$ and Littlewood--Paley theory in the last inequality. Thus \eqref{Unbound} is established  for $n>j$.

This finishes the proof of the proposition.
\end{proof}

\begin{remark}
The difficulty for putting the pieces together comes because it is assumed $\tfrac{r(d+1)}{d-1} < p $. If one had $r \geq p$, one can simply put pieces together by standard Littlewood--Paley theory as, for instance, in \eqref{cGr}.
\end{remark}

A consequence  of Proposition \ref{global-fixedj}
is the following restricted weak type bound.
\begin{proposition}\label{global-Besov} For $d\ge 3$, $r\ge 2$, 
\Be \label{eqn:globalrwt} \Big\| \Big(\sum_{k\in \bbZ} \Big\|\sum_{j\ge 1} \cA_j^k f \Big \|_{B^{1/r}_{r,1}(\bbR)}^r\Big)^{1/r} \Big\|_{L^{rd,\infty} (\bbR^d)} \lc \|f\|_{L^{rd,1}(\bbR^d)}.
\Ee
\end{proposition} 

\begin{proof} 
 Write \[\La_l\cA^k_j f(x,t)=  2^{-j(d-1)/2} (2\pi)^{-(d+1)}\sum_\pm \int 2^{-kd}\ka^{\pm}_{j,l}(2^{-k}y,t ) f(x-y) dy \]
where $\ka^\pm_{j,l}$ is as in \eqref{kajl-kernel}.

We first  show that for all $j\ge 1$,
$2\le r<\infty$, $\frac{r(d+1)}{d-1}< p<\infty$,  
\Be \label{globalBesovineq-j}
\Big\| \Big(\sum_{k\in \bbZ} \|\cA_j^k f \|_{B^{1/r}_{r,1}(\bbR)}^r\Big)^{1/r} \Big\|_{L^p(\bbR^d)} \lc 2^{-j(d/p -1/r)} \|f\|_{L^p(\R^d)}.
\Ee
Indeed, by Proposition \ref{global-fixedj}, for $|j-l|\le 10$
\Be \label{eq:main contr global} \Big\| \Big(\sum_{k\in \bbZ} 2^{l/r}\|\La_{l}\cA_j^k f \|_{L^r(\bbR)}^r\Big)^{1/r} \Big\|_{L^p(\bbR^d)} \lc 2^{-j(d/p -1/r)} \|f\|_{L^p(\R^d)}.
\Ee
Moreover for $|j-l|\ge 10$, we get from \eqref{ptwise-ka}
\begin{align}\notag 
\Big\| \Big( \sum_{k\in \bbZ} & 2^{l/r}\|\La_{l}\cA_j^k f \|_{L^r(\bbR)}^r\Big)^{1/r} \Big\|_{L^p(\bbR^d)} 
\\ & \lc_N \min\{ 2^{-j(N-\frac 1r)}, 2^{-l(N-\frac 1r)}\}
\Big\| \Big(\sum_{k\in \bbZ} \big| M_{HL}[ \widetilde{P}_{j-k}f]|^r\Big)^{1/r}\Big\|_{L^p(\R^d)} \notag \\
& \lesssim_N \min\{ 2^{-j(N-\frac 1r)}, 2^{-l(N-\frac 1r)}\} \| f \|_{L^p(\R^d)}, \label{eq:error terms global}
\end{align}
using that the Fefferman--Stein and Littlewood--Paley inequalities together imply
\[\Big\| \Big(\sum_{k\in \bbZ} \big| M_{HL}[ \widetilde{P}_{j-k}f]|^r\Big)^{1/r}\Big\|_{L^p(\R^d)}\lc_p \|f\|_{L^p(\R^d)}, \quad 1<p<\infty,\,\, r\ge 2.
\]
Then \eqref{globalBesovineq-j} follows summing over $\ell \geq 0$ in \eqref{eq:main contr global} and \eqref{eq:error terms global}.

We finish the proof using Bourgain's interpolation trick (see \S\ref{sec:Binterpol}).
Consider 
\[\fA_j f(x):=   \Big(\sum_{k\in \bbZ} \|\cA_j^k f(x,\cdot) \|_{B^{1/r}_{r,1}(\bbR)}^r\Big)^{1/r}.\] Note that $\frac{r(d+1)}{d-1}<rd$ when $d^2-2d-1>0$, that is, $d\ge 3$.
Let $p_0$, $p_1$ be such that $\frac{r(d+1)}{d-1}< p_0<rd<p_1$. By \eqref{globalBesovineq-j} we have that $\| \fA_j f\|_{L^{p_i}} \lc 2^{-j(d/p_i-1/r)} \|f\|_{L^{p_i}}$, $i=0,1$ 
and then a  restricted weak type $(rd, rd)$ inequality for $\sum_{j\ge 0} \fA_j$ follows from Lemma \ref{lem:Binterpol}. This implies the assertion. 
\end{proof}

\begin{proof}[Conclusion of the proof of Theorem \ref{endptvarthm}] 
Following \cite{JSW}, we write 
\begin{equation}\notag \label{eq:pointwise variation}
V_r A f(x) \le V_r^\dyad Af(x) + V_r^\sh Af(x)
\end{equation}
where 
\begin{equation*}
V^{\dyad}_r Af(x):= \sup_{N \in \N}\sup_{k_1<\dots<k_N} \Big(\sum_{i=1}^{N-1}|A_{2^{k_{i+1}}}f(x)- A_{2^{k_i}}f(x)|^r \Big)^{1/r}
\end{equation*}
is the \textit{dyadic or long variation operator} and
\begin{equation*}
V_r^\sh Af(x):=\Big(\sum_{k\in \bbZ}|V_r^{I_k} A f(x)|^r\Big)^{1/r}
\end{equation*}
is the \textit{short variation operator}, using only variation within the dyadic intervals $I_k=[2^k, 2^{k+1}]$; recall that $V_r^{I_k} Af(x)$ denotes the $r$-variation of $t\mapsto A_t f(x)$ over the interval $I_k$. It then suffices to establish the claimed bound in Theorem \ref{endptvarthm} for the operators $V_r^{\dyad}$ and $V_r^\sh$.

Regarding $V_r^\dyad A$, the inequality
\begin{equation}\notag \label{eq:Vr dyad bound}
\|V^{\dyad}_r Af\|_{L^p}\lc_{p,r} \|f\|_{L^p}, \quad r>2,\,\,1<p<\infty 
\end{equation}
was proved in \cite{JSW}.  This of course implies a $L^p$ bound for $V_{p/d}$ if $p >2d$ and, in particular, the claimed restricted weak type bound follows by the embedding $L^{p,1}\hookrightarrow L^{p} \hookrightarrow L^{p,\infty}$.

We next proceed with $V_r^\sh A$. Since $\chi(t)=1$ on $I=[1,2]$ we get
\[ V_r^{I_k} Af(x)\le \Big| \sum_{j=0}^\infty \cA_j^k f(x,\cdot)\Big |_{V_r(I)}\]
by the definition of $\cA_j^k$ in \eqref{eq:def Ajk}. 
The term corresponding to $j=0$ is easily estimated by a square function
\[ \Big(\sum_{k\in \bbZ} |\cA_0^k f(x,\cdot)|_{V_r(I)}^r\Big)^{1/r}
\lc \Big(\sum_{k\in \bbZ} \int_{1}^2 \big|\tfrac{d}{dt} \cA_0^k f(x,t) \big |^2 \ud t\Big)^{1/2}.
\] 
We claim
for $1<p<\infty$ 
\Be \label{jeqzero}
\Big\|\Big(\sum_{k\in \bbZ} \int_{1}^2 \big|\tfrac{d}{dt} \cA_0^k f(\cdot,t) \big |^2 \ud t\Big)^{1/2}\Big\|_{L^p} \le C_p \|f\|_{L^p}.
\Ee
Since $\chi'(t)=0$ for $1\le t\le 2$ we have
\[ \tfrac d{dt} \widehat {\cA_0^k f}(\xi,t) = \chi(t) \inn{2^k\xi}{\nabla \widehat\sigma( 2^kt\xi)    } \beta_0(2^k|\xi|)\widehat f(\xi)
\]
Using Plancherel's theorem  and interchanging sums and integrals one  gets
\eqref{jeqzero} for $p=2$.
We then invoke standard Calder\'on--Zygmund theory in the Hilbert-space setting 
(see \cite[ch. II.5]{Stein70SI})  to see that \eqref{jeqzero} holds in the full range $1<p<\infty$. It follows that for $r\ge 2$
\[\Big\|
\Big(\sum_{k\in \bbZ} |\cA_0^k f(x,\cdot)|_{V_r(I)}^r\Big)^{1/r}
\Big\|_{L^p(dx)} \lc_p \|f\|_{L^p}
\]
which is stronger than the required  $L^{p,1}\to L^{p,\infty} $ bound.

It remains to consider  the cases $j\ge 1$. By the embedding  \eqref{BPe} we have 
\Be \notag
\Big(\sum_{k\in \bbZ}\Big| \sum_{j\ge 1} \cA_j^k f(x, \cdot)\Big|^r_{V_r(I)} \Big)^{1/r} 
\lc 
\Big(\sum_{k\in \bbZ}\Big\| \sum_{j\ge 1} \cA_j^k f(x, \cdot )\Big\|^r_{B^{1/r}_{r,1}(I)} \Big)^{1/r}. 
\Ee
We apply the restricted weak type inequality of  Proposition \ref{global-Besov} to the expression on the right-hand side to conclude the desired bound for $V_r^\sh$. This finishes the proof.
\end{proof}

\begin{remark} If in two dimensions one has the conjectured local smoothing {\it endpoint}  results for $p>4$ then one can also show the restricted weak type  $(2r, 2r)$ estimate 
\eqref{eqn:globalrwt} 
for $r>2$.  The conjectured endpoint estimate in the assumptions seems currently out of reach. \end{remark}

\medskip

\medskip

\section{A sparse domination result}\label{sparse-section}

We conclude the paper with a discussion  of the sparse domination result for the global $V_rA$ in Theorem \ref{cor:sparse}. Roos and two of the authors proved in \cite{BRS-sparse} a general sparse domination result for multi-scale operators, which has a version for variation-norm operators associated to convolutions with compactly supported distributions. Such a result is a mechanism to upgrade Lebesgue space bounds to sparse bounds; in particular, and up to endpoints, the range of sparse bounds is uniquely determined by the $L^p \to L^q$ mapping properties of the local variation operator. Thus, the sparse bounds in Theorem \ref{cor:sparse} are an immediate consequence of the results in this paper, the bounds for the global variation operator in \cite[Theorem 1.4]{JSW} and  the sparse domination result in \cite[Proposition 7.2]{BRS-sparse}.

For completeness, we state the result in \cite{BRS-sparse}. We let $u$ be a compactly supported distribution, define the dilate in the sense of distributions by
$\inn{u_t}{f}=\inn{u}{f(t\cdot)} $ and let $Tf(x,t)= f*u_t$. For fixed $x$ let $V_r T f$ denote the $r$-variation norm  of $t\mapsto Tf(x,t)$.
As before let $I=[1,2]$ and $V_r^I f(x)$ the corresponding variation norm over $I$. 

\begin{theorem}[{\cite[Prop. 7.2]{BRS-sparse}}]\label{thm:Vu}
 Let $1<p\le q<\infty$, and $u\in \cS'(\bbR^d)$  with compact support in  $\bbR^d\setminus \{0\}.$
%Suppose that $1<p\le q<\infty$. 

(i) Suppose that
\Be\label{eqn:pp-qq-var}  \|V_rT \|_{L^p\to L^{p,\infty} } 
+\|V_r T\|_{L^{q,1}\to L^{q}} <\infty, \Ee
\Be \label{eqn:pq-var}  \|V^I_r T \|_{L^p\to L^q}  
<\infty,  \Ee
and  that there is an $\varepsilon>0$ so that for all $\la\ge 2$,
and all  Schwartz function  $f$  
with $\supp\,\widehat f\subset \{\xi: \la/2<|\xi|<2\la \} $,
\Be  \label{eqn:pqregV} \|V^I_r Tf\|_{L^q}\le C \la^{-\eps} \|f\|_{L^p}.\Ee
Then there is a constant $C=C(p,q)$ such that for each pair of compactly supported bounded functions $f_1$, $f_2$ there is a sparse family of cubes $\fS(f_1,f_2) $ such that 
\Be\label{eqn:sparsebound} \int V_rT f_1(x)\, f_2(x) dx \le 
C \sum_{Q \in \mathfrak{S} (f_1,f_2)}|Q| \langle f_1 \rangle_{Q,p} \langle f_2 \rangle_{Q,q'} .
\Ee

(ii) Suppose that in addition $p<q$, and suppose that \eqref{eqn:sparsebound} holds with a constant independently of $f_1, f_2$.  Then conditions \eqref{eqn:pp-qq-var}, \eqref{eqn:pq-var} hold.
\end{theorem}

\begin{proof}[Proof of Theorem \ref{cor:sparse}]
We let $u$ be surface  measure on the unit sphere.
As discussed in the introduction the inequalities in \eqref{eqn:pp-qq-var} were already proved in the relevant ranges of Theorem \ref{cor:sparse} in \cite[Theorem 1.4]{JSW}. The inequalities  \eqref{eqn:pq-var} and  \eqref{eqn:pqregV} in the asserted ranges  follow  from  
the single-scale frequency bounds in Propositions
 \ref{prop:bounds with 1/r} and \ref{prop:bounds 1/r d=2}.
 Thus the sparse bounds in Theorem \ref{cor:sparse} are a  consequence of part (i) of Theorem \ref{thm:Vu}. The sharpness of the sparse bounds follows from part (ii); see also \S \ref{sparse-sharp} for a direct argument.
\end{proof} 

%\subsection*{Note} On behalf of all authors, the corresponding author states that there is no conflict of interest.

\bibliography{Reference}

\bibliographystyle{amsplain}

\end{document}